\documentclass[smallextended]{svjour3opt} 
\usepackage{etex}
\usepackage{todonotes}
\usepackage{lscape} 
\usepackage{url,palatino}
\usepackage{verbatim, amssymb,amsmath}
\usepackage{varwidth}
\usepackage[export]{adjustbox}
\usepackage{multirow}

\usepackage{mathrsfs}
\usepackage{graphics,amsopn,amstext,amsfonts,color}
\usepackage{bm}
\usepackage{setspace}
\usepackage{subfig}
\usepackage{graphicx}
\usepackage{epstopdf}
\usepackage{caption}
\usepackage{multirow}
\usepackage{booktabs}
\usepackage{paralist}
\usepackage{colortbl}
\usepackage{appendix}
\usepackage{soul,caption,subfig}
\usepackage[bookmarksnumbered=true,hidelinks]{hyperref}
\usepackage{graphics,latexsym,amsfonts,verbatim,amsmath}
\usepackage{algorithm}
\usepackage{algorithmicx}
\usepackage{algpseudocode}
\usepackage{tikz}
\usetikzlibrary{calc}
\usetikzlibrary{plotmarks}

\usepackage{pgfplots}

\usepackage{pgf}
\usepackage{pst-func}

\definecolor{pinegreen}{rgb}{0.0, 0.47, 0.44}
\usepackage{pgfpages}
\usepackage{todonotes}
\usepackage{hhline} 
\usepackage{empheq}
\usepackage{cases}

\def\L{{\mathcal L}}

\def\Re{\mathbb{R}}

\def\hat{\widehat}

\def \R{\mathcal{R}}

\def \Ze{{\mathbb{Z}}}

\def\L{{\mathcal L}}

\def\Re{{\mathbb R}}

\newcommand{\exclude}[1]{}

\newcommand{\bfx}{\bm{x}} \newcommand{\y}{\bm{y}} 
\newcommand{\bfz}{\bm{z}}

\DeclareMathOperator{\Proj}{Proj}

\newcommand*{\QEDA}{\hfill\ensuremath{\square}}

\algdef{SE}[DOWHILE]{Do}{doWhile}{\algorithmicdo}[1]{\algorithmicwhile\ #1}%

\def\L{{\mathcal L}}

\def\Re{\mathbb{R}}

\def\hat{\widehat}

\def \R{\mathcal{R}}

\def \Ze{{\mathbb{Z}}}

\def\L{{\mathcal L}}

\def\Re{{\mathbb R}}

\def\x{\vect{x}}

\newcommand{\vect}[1]{\boldsymbol{\bm{#1}}}

\newcommand{\subparagraph}{}
\usepackage{titlesec}

\usepackage[numbers,sort&compress,comma]{natbib}
\allowdisplaybreaks

\usepackage[para,online,flushleft]{threeparttable}

\def\Ce{\mathbb{C}}
\DeclareMathOperator{\epi}{epi}
\DeclareMathOperator{\hyp}{hyp}

\usepackage{wrapfig}
\usepackage[font=small]{caption}

\title{Second-Order Conic and Polyhedral Approximations of the Exponential Cone:\\
	Application to Mixed-Integer Exponential Conic Programs}

\date{\today}
\titlerunning{SOC and Polyhedral Approximations of the Exponential Cone}

\author{Qing Ye \and Weijun Xie}

\institute{Qing Ye \at
	Virginia Tech, Blacksburg, VA\\
	\email{yqing1@vt.edu}
	\and 
	Weijun Xie \at
	Virginia Tech, Blacksburg, VA\\
	\email{wxie@vt.edu}
}
\allowdisplaybreaks

\makeatletter
\pgfmathdeclarefunction{erf}{1}{%
	\begingroup
	\pgfmathparse{#1 > 0 ? 1 : -1}%
	\edef\sign{\pgfmathresult}%
	\pgfmathparse{abs(#1)}%
	\edef\x{\pgfmathresult}%
	\pgfmathparse{1/(1+0.3275911*\x)}%
	\edef\t{\pgfmathresult}%
	\pgfmathparse{%
		1 - (((((1.061405429*\t -1.453152027)*\t) + 1.421413741)*\t
		-0.284496736)*\t + 0.254829592)*\t*exp(-(\x*\x))}%
	\edef\y{\pgfmathresult}%
	\pgfmathparse{(\sign)*\y}%
	\pgfmath@smuggleone\pgfmathresult%
	\endgroup
}
\makeatother

\allowdisplaybreaks
\begin{document}
	\maketitle

	\begin{abstract} 
Exponents and logarithms are fundamental components in many important applications such as logistic regression, maximum likelihood, relative entropy, and so on. Since the exponential cone can be viewed as the epigraph of perspective of the natural exponential function or the hypograph of perspective of the natural logarithm function, many mixed-integer convex programs involving exponential or logarithm functions can be recast as mixed-integer exponential conic programs (MIECPs). However, unlike mixed-integer linear programs (MILPs) and mixed-integer second-order conic programs (MISOCPs), MIECPs are still under development. To harvest the past efforts on MILPs and MISOCPs, this paper presents second-order conic (SOC) and polyhedral approximation schemes for the exponential cone with application to MIECPs. To do so, we first extend and generalize existing SOC approximation approaches in the extended space, propose new scaling and shifting methods, prove approximation accuracies, and derive lower bounds of approximations. We then study the polyhedral outer approximation of the exponential cones in the original space using gradient inequalities, show its approximation accuracy, and derive a lower bound of the approximation. When implementing SOC approximations, we suggest learning the approximation pattern by testing smaller cases and then applying it to the large-scale ones; and for the polyhedral approximation, we suggest using the branch and cut method for MIECPs. Our numerical study shows that the proposed methods show speed-ups over solver MOSEK for MIECPs, and the scaling, shifting, and polyhedral outer approximation methods work very well.

\end{abstract}

\section{Introduction}
Conic programming has recently witnessed a rapid growth of interest in both industry and academia and has been successfully applied to many domains (see, e.g.,  \cite{nemirovski2007advances,lobo1998applications,alizadeh2003second,vandenberghe1996semidefinite,wolkowicz2012handbook}). Notably, the recent version of solver MOSEK \citep{aps2021mosek} can solve quite large-scale continuous exponential conic programs (ECPs). However, mixed-integer exponential conic programs (MIECPs) are still under development. It is known that polyhedral approximation results and strong valid inequalities developed during the last one or two decades can significantly enhance the capability of existing solvers such as Gurobi and CPLEX for solving mixed-integer linear programs (MILPs) and mixed-integer second-order conic programs (MISOCPs). 
To harvest these past efforts on MILPs and MISOCPs,
we study second-order conic (SOC) approximations and polyhedral approximations of the exponential cone in the hope of solving MIECPs efficiently. 

\subsection{Setting}

This paper focuses on the following MIECP
\begin{align}\label{eq_miecp}
	\min_{\bm{x}\in X\subset \Ze^{t}\times \Re^{n-t}}\left\{\bm{c}^\top\bm{x}: \left(\bm{a}_{i1}^\top \bm{x}+b_{i1},  \bm{a}_{i2}^\top \bm{x}+b_{i2},  \bm{a}_{i3}^\top \bm{x}+b_{i3}\right)\in 	K_{\exp}(0),\forall i\in [m]\right\}
\end{align}
where set $X$ is mixed-integer and compact, $t\in [0,n]$, and $\bm{A}_1,\bm{A}_2,\bm{A}_3\in \Re^{m\times n}, \bm{b}_1,\bm{b}_2,\bm{b}_3\in \Re^m$ are data. For a given $\alpha\in \Re$, we let $K_{\exp}(\alpha)$ denote the parametric exponential cone
\begin{subequations}\label{eq_ec12}
	\begin{align}\label{eq_ec1}
		K_{\exp}(\alpha)&=\left\{\bm{x}\in \R^{2}\times \Re: \left(x_1/x_2, x_3/x_2\right)\in\hyp\log(\alpha)\right\} \\
		&=\left\{\bm{x}\in \R^{2}\times \Re: \left(x_3/x_2, x_1/x_2\right)\in\epi\exp(\alpha)\right\}\label{eq_ec2}
	\end{align}
\end{subequations}
where the interval $\R:=[1/M,M]$ for some positive parameter $M>1$, 
set $\hyp\log(\alpha)$ denotes the hypograph of the logarithm function parameterized by $\alpha$
\[\hyp\log(\alpha)=\left\{(\bm{x},t): \log(\bm{x})\geq t-\alpha \right\},\]
and set $\epi\exp(\alpha)$ denotes the epigraph of the exponential function parameterized by $\alpha$
\[\epi\exp(\alpha)=\left\{(\bm{x},t): \exp(\bm{x}-\alpha)\leq t \right\}.\]
We remark that for the modeling purpose, one can assume that $M$ can be $\infty$ and by default, we let $0\times \infty=0$ and $1/\infty=0$, under which set $K_{\exp}(\alpha)$ remains to be closed; however, for the analytical purpose, our approximation results only hold when $M$ is finite.
Besides, the parameter $\alpha\in \Re$ is meant to quantify the error of the proposed approximations. It is also worthy of mentioning that when $\alpha=0$, i.e., no approximation error, the exponential cone $K_{\exp}(0)$ in  \eqref{eq_ec12} can be viewed as the hypograph of perspective of the logarithm function or the epigraph of perspective of the exponential function. These two interpretations motivate us to study distinct approximation schemes.

As the exponential cone $K_{\exp}(0)$ is fundamental to an MIECP, this paper plans to develop SOC approximations of the exponential cone $K_{\exp}(0)$ of the form
\begin{align}\label{eq_sc}
	K_{sc}(n^s,m^s)=\left\{\bm{x}\in \Re^{n^s} : \left((\bm{a}^s_{i1})^\top \bm{x}+b^s_{i1},  (\bm{a}^s_{i2})^\top \bm{x}+b^s_{i2},  (\bm{a}^s_{i3})^\top \bm{x}+b^s_{i3}\right)\in \L_3,\forall i\in [m^s]
	\right\},
\end{align}
where the Lorentz cone is defined as $\L_{q+1}=\{\bm{y}\in \Re^{q+1}: y_{q+1}\geq \sqrt{\sum_{i\in [q]}y_i^2}\}$ with $q\in \Ze_{++}$
and $\bm{A}_1^s,\bm{A}_2^s,\bm{A}_3^s\in \Re^{m^s\times n^s}, \bm{b}_1^s,\bm{b}_2^s,\bm{b}_3^s\in \Re^{m^s}$ are data. We would like to show that for any $\epsilon>0$, there exists a second-order cone $K_{sc}(n^s,m^s)$ with $n^s+m^s=O(f^s(1/\epsilon,M))$ 
depending on $1/\epsilon$ and $M$ such that
\[K_{\exp}(-\epsilon)\subseteq K_{sc}(n^s,m^s) \subseteq K_{\exp}(\epsilon),\]
where $f^s(1/\epsilon,M)$ is a polynomial function.
According to the seminal work \citep{ben2001polyhedral} on the polyhedral approximations of SOC programs, there exists a polyhedral approximation of the exponential cone $K_{\exp}(0)$ such that
\begin{align}\label{eq_lp}
	K_{p}(n^p,m^p)=\left\{\bm{x}\in \Re^{n^p} : (\bm{a}^p_{i})^\top \bm{x}\geq b^{p}_{i},\forall i\in [m^p]
	\right\},
\end{align}
where $\bm{A}^p\in \Re^{m^p\times n^p}, \bm{b}^p\in \Re^{m^p}$ are data. Using the results in \cite{ben2001polyhedral}, for any $\epsilon>0$ polyhedral approximation such that
\[K_{\exp}(-\epsilon)\subseteq K_{p}(n^p,m^p) \subseteq K_{\exp}(\epsilon),\]
the size of polyhedral approximation should be in the order of the size of second-order cone $K_{sc}(n^s,m^s)$ multiplied by the $\log(1/\epsilon)$, i.e., $n^p+m^p=O((n^s+m^s)\log(1/\epsilon))$. Since this type of polyhedral approximations of SOC programs has been successfully implemented in solvers such as Gurobi and CPLEX, this paper mainly focuses on polyhedral approximation in the original space, i.e., we mainly consider $n^p=3$ in \eqref{eq_lp}.

\subsection{Relevant Literature}
Many interesting nonlinear programs can be recast as an ECP (see, e.g., \cite{serrano2015algorithms} for an overview). 
Below are some examples: 
\begin{itemize}
	\item (Entropy) The hypograph of the entropy function $f(x)=-x\log(x)$ is exponential-conic representable
	\[\hyp f:=\left\{(x,t): -x\log(x)\geq t\right\}=\left\{(x,t): (1,x,t)\in K_{\exp}(0)\right\}.\]
	
	\item (Kullback-Leibler Divergence) The epigraph of the Kullback-Leibler divergence or relative entropy function $f(x,y)=x\log(x/y)$ is exponential-conic representable
	\[\epi f:=\left\{(x,y,t): x\log(x/y)\leq t\right\}=\left\{(x,y,t): (y,x,-t)\in K_{\exp}(0)\right\}.\]
	
	\item (Logistic Regression) Given $n$ data points $\{(\bfx_i,y_i)\}_{i\in[n]}\subseteq \Re^d\times \{0,1\}$, the logistic regression with $L_1$ penalty can be formulated as the following unconstrained optimization problem 
	\[\min_{\bm{\theta}\in \Re^d} \sum_{i\in [n]}\left[-y_i \log(h_{\bm{\theta}}(\bfx_i))-(1-y_i) \log(1-h_{\bm{\theta}}(\bfx_i))\right]+ \lambda \|\bm{\theta}\|_1,\]
	where $h_{\bm{\theta}}(\bfx_i)=[1+\exp(-\bm{\theta}^{\top}\bfx_i)]^{-1}$ denotes the sigmoid function. The logistic regression can be recast as the following ECP
	\[\min_{\bm{\theta}\in \Re^d, \bm{p}_1,\bm{p}_2} \left\{ \sum_{i\in [n]} t_i+\lambda \|\bm{\theta}\|_1: \begin{array}{l}\displaystyle
		p_{i1}+p_{i2}=1, \left(p_{i1}, 1,(1-2y_i)\bm{\theta}^{\top}\bfx_i-t_i\right)\in K_{\exp}(0), \forall i \in [n],\\\displaystyle
		\left(p_{i2}, 1,-t_{i}\right)\in K_{\exp}(0), \forall i \in [n] 
	\end{array}\right\}.\]
	
	\item (Geometric Programming) Geometric programming problems  \cite{boyd2007tutorial} are nonlinear optimization problems of the form
	\[\min_{\bm{x}}  \left\{ f_{0}(\bfx): f_{i}(\bfx)\leq 1, \forall i \in [m],x_j\geq 0, \forall j \in [n] \right\}.\]
	Above, function $f_{i}(\bm{x})=\sum_{k\in [T_i]}c_{ik}\prod_{j \in [n]}x_{j}^{a_{ijk}}$ is a posynomial for each $i \in [0, m]$, where $T_i\in \Ze_+$, $c_{ik}\in\Re_+$, and $a_{ijk} \in\Re$ for each $k\in [T_i]$ and $j\in [n]$. 
	Letting $x_j=\exp(y_j)$ for all $j \in [n]$ and substituting back the expressions, the epigraph of the revised ``posynomial" $f_{i}(\bm{y})$ is exponential-conic representable
	\[\epi f_i=\left\{(\bm{y},t): \sum_{k\in [T_i]} p_{ik} \leq 1, \left(p_{ik}, 1, \bm{a}_{ik}^{\top}\y+\log(c_{ik})-t \right)\in K_{\exp}(0), \forall k\in [T_i]\right\}.\]
	Therefore, the geometric programming problems can be equivalently transformed as the following ECPs
	\[\min_{\bm{y}}  \left\{ t_0: (\bm{y},t_0)\in \epi f_0, (\bm{y},1)\in\epi f_i,\forall i \in [n]\right\}.\]
\end{itemize}

The ECPs or MIECPs also exist in many areas such as finance, defense, and healthcare. For example, the works in manpower planning \citep{jaillet2018strategic}, electric vehicle charging management \citep{chen2021exponential}, Kullback-Leibler divergence constrained distributionally robust optimization \citep{kocuk2020conic} have been using ECP approaches. 
Many other works might not have been fully aware of but can be reformulated as ECPs or MIECPs. For instance, in \cite{patriksson2008survey}, the authors used the limited resources on-hand to maximize the probability of searching for objects, where the hypograph of the searching probability, in fact, can be represented as an exponential cone. Thus, their problem can be formulated as an MIECP. 
In \cite{somers2020sparse}, the authors proposed an ECP for sparse resource allocation in control of spreading processes. Their proposed method could be applied to minimize the spreading rate of epidemics and wildfires. Recently, the work in \cite{zhu2021joint} proposed a joint estimation and robustness optimization framework that could be modeled with ECPs. 
As an important class of ECP problems, geometric programming has a variety of applications, including telecommunication \cite{chiang2005geometric}, circuit design \cite{boyd2005digital}, and power control \cite{chiang2007power}. 

Recently, several works focused on solving continuous ECPs. For example, solver SCS \cite{o2016conic} is a first-order method using operator splitting and homogeneous self-dual embedding, which can handle symmetric cones as well as the power and exponential cones. In \cite{serrano2015algorithms}, the authors analyzed the theoretical properties of some algorithms, such as primal and dual barrier methods, and extended an interior-point conic solver ECOS \citep{domahidi2013ecos} to solve continuous ECPs. Alfonso \cite{papp2021alfonso} is a conic solver for nonsymmetric cones, including power and exponential cones, where the implementation is based on a homogeneous interior-point method proposed by \cite{skajaa2015homogeneous}. The authors in \cite{dahl2021primal} proposed a new primal-dual interior-point algorithm for continuous ECPs with a novel higher-order search direction. The work \cite{badenbroek2021algorithm} analyzed theoretical properties of the algorithm implemented in MOSEK and proposed a new one for solving nonsymmetric conic optimization such as ECPs. 
The literature of generic approaches for solving ECPs, especially MIECPS, is relatively sparse. 
It is known that MOSEK can solve MIECPs  \citep{dahl2021primal} as early as 2018, which is mainly based on the interior point method and branch and bound. 
Parallelly, the work in \cite{coey2020outer} presented Pajarito, a generic branch and bound algorithm with outer approximation using primal and dual information for solving mixed-integer convex problems involving positive semidefinite, second-order, and exponential cones. 
Different from their methods, our methods are based on second-order approximations and gradient-based polyhedral approximations. 

Our works on SOC approximations are motivated but different from the two seminal ones. The first interesting work is \cite{fawzi2019semidefinite}, which studied the semidefinite approximations of the matrix logarithm. They proposed a special function, whose integral 
over a particular domain is equal to the logarithm function and is semidefinite programming representable. However, their analysis is quite limited to a particular function, and mixed-integer semidefinite programming is known to be notoriously difficult to solve. On the contrary, we propose a large family of generating functions, which are SOC representable, and whose integrals from $-1$ to $1$ are equal to the logarithm function. 
Our analysis is much simplified, and given an approximate solution, we propose a new scaling method, which works the best in our numerical study.
The second interesting work is \cite{ben2001lectures}, which studied the SOC approximations of the exponential function. We analyze their approximation errors and propose a new shifting method with a given approximate solution. The numerical study shows that the proposed shifting method works well while other exponential function-based methods can have numerical issues. Besides, different from both works, we also derive lower bounds for the approximation errors and study polyhedral outer approximation based on gradient inequalities.


\subsection{Summary of Contributions} 
This paper generalizes the existing and develops new SOC approximations of the exponential cone \eqref{eq_ec12}. We also study polyhedral outer approximation based on gradient inequalities. Our main contributions are summarized as below:
\begin{enumerate}[(i)]
	\item We propose a generic SOC approximation framework of the exponential cone based on the logarithmic form \eqref{eq_ec1}. We prove the approximation accuracy using the Gaussian quadrature. We tailor the proof to three special classes of SOC approximation. In particular, our proposed scaling approximation scheme using an approximate solution is numerically demonstrated to work the best.
	\item We prove the approximation accuracies of the SOC approximations of the exponential cone using the exponential form \eqref{eq_ec2}. We show that the even-order Taylor expansions of the exponential function can be written as the sum of squares and disprove a SOC representation proposed by \cite{ben2001lectures}. We develop a new shifting method using an approximate solution, which overcomes the numerical issues caused by other approximation schemes using the exponential form \eqref{eq_ec2}.
	\item We study the minimum number of variables and SOC constraints needed to approximate the exponential cone to a desirable accuracy, i.e., study a lower bound of SOC approximations for the exponential cone. 
	\item We propose a polyhedral outer approximation of the exponential cone in the original space using gradient inequalities and study its upper and lower bounds for the number of inequalities to achieve a desirable approximation accuracy.
	\item Our numerical study shows that the proposed scaling, shifting, and polyhedral outer approximation methods outperform solver MOSEK for solving MIECPs and can achieve up to 20 times speed-ups.
\end{enumerate}
The main theoretical approximation complexities of this paper are displayed in Table~\ref{table_summary}. 


\begin{table}[htbp]
	\centering
	\caption{Summary of Approximation Methods of the Exponential Cone}
	\vspace{-5pt}
	\label{table_summary}
	\begin{threeparttable}
		\scriptsize\setlength{\tabcolsep}{1.0pt}
		\setlength\extrarowheight{2.0pt}
		\begin{tabular}{c|c|c|c|c|c}
			\hline
			& \multicolumn{4}{c|}{SOC Approximations\tnote{i}}  & Polyhedral Approximation\tnote{i}\\ \hline\hline
			\multirow{2}{*}{Methods} & \multicolumn{3}{c|}{Logarithm-based \tnote{iii}}  &  & \multirow{2}{*}{Original Space}\\ \cline{2-4}
			&  Example \ref{emp_1}     &  Example \ref{emp_2}      &  Example \ref{emp_3}  & &  \\ \cline{1-4} \cline{6-6} 
			\multirow{2}{*}{Complexity\tnote{ii}} &   $O(\sqrt{M}$   &  $O(\log\log(M)$     &   $O(\log(1-\delta)/\log(\delta)$   & & \multirow{2}{*}{$\Theta(\log(M)+1/\sqrt{\epsilon})$}    \\ 
			&  $+ \log(\sqrt{M^3 /\epsilon}))$     & $+\sqrt{\log(M)+\log(1/\epsilon)})$   & $+\log(\epsilon)/\log(\delta))$ & &   \\ \cline{1-4} \cline{6-6} 
			Outer Approximation & No & No & No & & Yes\\ \hhline{=|=|=|=|=|=}
			& \multicolumn{4}{c|}{Exponential-based \tnote{iv}} &  \\ \cline{1-5}  
			Methods   &   Section \ref{sec_exp_lim}   &   Section \ref{sec_exp_taylor}   &   Section \ref{sec_exp_lim} Shift  &  Section \ref{sec_exp_taylor} Shift & \\ \cline{1-5} 
			\multirow{2}{*}{Complexity\tnote{ii}}   &  $O(\log(M^2\log M)$   &  $O(\log(M^2\log(M)/s)$    &  \multirow{2}{*}{$O(\log(\delta^2/\epsilon))$}    &  $O(\log(\delta/s)$ & \\ 
			& $+\log(1/\epsilon))$     & $+\log(1/\epsilon)/s)$    &    & $+\log(1/\epsilon)/s)$ & \\ \cline{1-5} 
			Outer Approximation & Yes & No & Yes & No &\\ \hhline{======}
			Lower Bound & \multicolumn{4}{c|}{$\Omega (1+\log\log(M)/\log(1/\epsilon))$} & $\Omega (\log\log(M)+\log(1/\epsilon))$ \\ \hline
		\end{tabular}
		\begin{tablenotes}
			\item[i] $\epsilon$ is the approximation accuracy;
			\item[ii] Complexity is the number of variables, SOC and linear constraints needed;
			\item[iii] $\delta\in(0,1)$ is quality of the approximate solution $\hat{x}_1/\hat{x}_2$ such that $1-\delta\leq (x_1/x_2)/(\hat{x}_1/\hat{x}_2) \leq 1+\delta$;
			\item[iv] $\delta>0$ is quality of the approximate solution $\hat{x}_3/\hat{x}_2$ such that $\left|x_3/x_2-\hat{x}_3/\hat{x}_2\right| \leq \delta$.
		\end{tablenotes}
	\end{threeparttable}
\end{table}

\noindent\textbf{Notation.}  The following notation is used throughout the paper. 
We use bold letters (e.g., $\vect{x},\vect{A}$) to denote vectors and matrices and use corresponding non-bold letters to denote their components. 
Given an integer $n$, we let $[n]:=\{1,2,\ldots,n\}$, $[0, n]:=\{0, 1,2,\ldots,n\}$, and use $\Re_+^n:=\{\vect{x}\in \Re^n:x_i\geq0, \forall i\in [n]\}$. Given a function $f:\Re^n\rightarrow\Re$, its epigraph and hypograph are $\epi f=\{(\bm{x},t): f(\bm{x})\leq t\}$ and $\hyp f =\{(\bm{x},t): f(\bm{x})\geq t\}$, respectively. For the sake of simplicity, the logarithm function and exponential function, if not being specified, are both meant to be natural ones in this paper. 
Additional notation will be introduced as needed.

\noindent\textbf{Organization.}	 The remainder of the paper is organized as follows. 
Section \ref{sec_log} develops and analyzes the approximation schemes using the logarithmic form \eqref{eq_ec1}. Section \ref{sec_exp} develops and analyzes the approximation schemes using the exponential form \eqref{eq_ec2}. Section \ref{sec_bound} shows lower bounds of approximations. Section \ref{sec_outer} shows polyhedral outer approximation using gradient inequalities and proves the upper and lower bounds of approximations. Section \ref{sec_numerical} shows the numerical illustration. Finally, Section \ref{sec_conclusion} concludes this paper.

\section{Approximations Using the Logarithmic Form \eqref{eq_ec1}}\label{sec_log}
In this section, we study approximations of the exponential cone using the logarithm form \eqref{eq_ec1}. The key idea is to find a family of functions whose integrals are equal to the logarithm function.
\subsection{Approximating the Logarithm Function Based on Integrals of Generating Functions}
Since it is unlikely that the logarithm function $\log(x)$ is SOC representable, we plan to approximate the logarithm $\log(x)$ with $x\in \R$ by rewriting it as the integral of an SOC representable function. 

To begin with, let us define the generating functions for the logarithm function. 
\begin{definition}[Generating Functions]\label{def_gen}A function $\phi:[-1,1]\times \R\rightarrow \Re$ is a ``\textit{generating function}" for the logarithm function $\log(x)$ if it satisfies
	\begin{itemize}
		\item[(a)] For any given $t\in [-1,1]$, function $\phi(t,x)$ is concave and continuous, and its hypograph
		\[\hyp {\phi}(t):=\left\{(x,v): \phi(t,x)\geq v\right\}\]
		is SOC representable in the domain $\R$; and
		\item[(b)] For any $x\in \R$, the following identity must hold
		\begin{align}\label{eq_integral}
			\int_{-1}^1 \phi(t,x)dt=\log(x)+\textrm{const}. 
		\end{align}
	\end{itemize}
\end{definition}
Note that the generating functions are useful to derive SOC approximations of the logarithm function since the one-dimensional integration \eqref{eq_integral} admits efficient approximations (e.g., using Gaussian quadratures). In fact, there exist many such generating functions. In Section \ref{sec_ex1}-Section \ref{sec_ex3}, we show three examples, which are quite effective in our numerical study.

Although generating functions in Definition \ref{def_gen} are SOC representable, it still remains to address the left-hand integration of \eqref{eq_integral}. Fortunately, the modern numerical methods are quite mature in one-dimensional integration. Thus, we propose to approximate the left-hand integration of \eqref{eq_integral} using well-known Gaussian quadruple (see, e.g., Chapter 19 \citep{trefethen2019approximation}). Suppose $\{t_k\}_{k\in [N]}\subseteq [-1,1]$ be $N$ quadrature points and their corresponding positive weights $\{w_k\}_{k\in [N]}$. Then the logarithm function $\log(x)$ can be tightly approximated by
\begin{align*}
	\int_{-1}^1 \phi(t,x)dt=\log(x)+\textrm{const} \approx \sum_{k\in [N]}w_k \phi(t_k,x).
\end{align*}
Accordingly, the hypograph of $\log(x)$ can be approximated by
\begin{equation}\label{eq_log_approx}
	\begin{aligned}
		\hyp {\log}(\alpha)\vert_{\alpha=0}&:=\left\{(x,\nu)\in \R\times\Re:\log(x)\geq \nu-\alpha\vert_{\alpha=0}\right\}\\
		&\approx \left\{(x,\nu)\in \R\times\Re: \sum_{k\in [N]}w_k v_k\geq \nu+\textrm{const}, (x,v_k)\in \hyp {\phi}(t_k),\forall k\in [N]\right\}:= \hat{H}_{\phi,N},
	\end{aligned}
\end{equation}
where constant $\alpha$ is useful to characterize the approximation accuracy, and the hypograph $ \hyp {\phi}(t)$ is supposed to be SOC representable. The Gaussian quadrature is quite accurate to approximate one-dimensional integration. In fact, according to Chapter 19 \citep{trefethen2019approximation}, the error of Gaussian quadrature decays geometrically.
\begin{theorem}\label{prop_approx_log}
	Suppose that there exists some $\rho>1$ such that
	\[\sup_{x\in \R, z\in \Ce}\left\{|\phi(z,x)|: z=\frac{\rho}{2}\exp(i\theta)+
	\frac{\rho^{-1}}{2}\exp(-i\theta):\theta\in [0,2\pi]
	\right\} \leq L<\infty.\]
	Then the approximation accuracy of Gaussian quadrature satisfies
	\[\left|\log(x)+\textrm{const}-\sum_{k\in [N]}w_k \phi(t_k,x)\right|\leq \frac{64 L\rho^{-2N}}{15(\rho^2-1)}.\]
	Thus, achieving the $\epsilon-$approximation accuracy with $\epsilon>0$, we can choose the number of points $N=O(\log_{\rho}(L/\epsilon))$ such that
	$\hyp {\log}(-\epsilon)\subseteq  \hat{H}_{\phi,N} \subseteq \hyp {\log}(\epsilon)$.
\end{theorem}

		The key ingredients in Theorem \ref{prop_approx_log} are two parameters $\rho$ and $L$, which require effort to derive. Fortunately, we can derive these parameters for the examples in Section \ref{sec_ex1}-Section \ref{sec_ex3} explicitly.
		
	\subsubsection{Example 1}\label{sec_ex1}
	\begin{example}\label{emp_1}
		Suppose $\phi_1(t,x)=\frac{a(t+1)^{a-1}(x-1)}{2^a+(t+1)^a(x-1)}$ for some positive parameter $a>0$. Then we have $\int_{-1}^1 \phi_1(t,x)dt=\log(x)$ and for any $t\in [-1,1]$, the hypograph of $\phi_1(t,\cdot)$ is SOC representable, i.e.,
		\begin{align}\label{eq_phi_rep_1}
			&\hyp {\phi_1}(t):=\left\{(x,v)\in \R\times \Re: \begin{aligned}
				&r_1=a-(t+1)v\geq 0, \\
				&r_2=a(t+1)^{a-1}(x-1)-2^a v \geq 0,\\
				&(v\sqrt{2^{a+2}(t+1)},r_1-r_2,r_1+r_2)\in \L_3
			\end{aligned}
			\right\}.
		\end{align}
	\end{example}
	%
	%
Next, we study the number of variables and SOC constraints needed to achieve the $\epsilon-$approximation accuracy using Example \ref{emp_1}.
	\begin{corollary}\label{cor_phi_1} For the generating function in Example \ref{emp_1} with $a=1$, we can choose $\rho=(M+1)/(M-1)+2/\sqrt{M}$, and
		\[L=\frac{1}{\frac{M+1}{M-1}-\frac{\rho}{2}-\frac{\rho^{-1}}{2}}.\]
		Thus, achieving the $\epsilon-$approximation accuracy requires the number of points $N=O(\sqrt{M}\log(\sqrt{M^3 /\epsilon}))$  such that
		$\hyp {\log}(-\epsilon)\subseteq  \hat{H}_{\phi_1,N} \subseteq \hyp {\log}(\epsilon)$. That is, there should be $O(\sqrt{M}\log(\sqrt{M^3 /\epsilon}))$ number of variables and SOC constraints in the representation of set $\hat{H}_{\phi_1,N}$, which admits
		\begin{equation}
			\hat{H}_{\phi_1,N} = \left\{(x,\nu)\in \R\times\Re: \begin{aligned}
				&\sum_{k\in [N]}w_k v_k\geq \nu,\\ 
				&r_{1k}=a-(t_k+1)v_k\geq 0, \forall k\in [N],\\
				&r_{2k}=a(t_k+1)^{a-1}(x-1)-2^a v_k \geq 0, \forall k\in [N],\\
				&(v_k\sqrt{2^{a+2}(t_k+1)},r_{1k}-r_{2k},r_{1k}+r_{2k})\in \L_3, \forall k\in [N]
			\end{aligned}\right\}.
		\end{equation}
	\end{corollary}
	\begin{proof}
		We observe that $|z|$ can be upper bounded by
		\[|z|=\left[ \frac{1}{4}\left(\rho+\rho^{-1}\right)^2\cos^2\theta+\frac{1}{4}\left(\rho-\rho^{-1}\right)^2\sin^2\theta \right]^{1/2} \leq \frac{\rho}{2}+\frac{\rho^{-1}}{2}.\]
		Then, we have
		\[\left|\phi(z,x)\right|=\left|\frac{x-1}{2+(z+1)(x-1)}\right|=\frac{1}{\left|\frac{x+1}{x-1}+z\right|}\leq\frac{1}{\left|\left|\frac{x+1}{x-1}\right|-|z|\right|}\leq\frac{1}{\frac{M+1}{M-1}-\frac{\rho}{2}-\frac{\rho^{-1}}{2}}:=L,\]
		where the first inequality is because of triangle inequality and the second one holds since $x\in \R$ and we choose 
		\[\rho= (M+1)/(M-1)+2/\sqrt{M}\in\left(1,\frac{M+1}{M-1}+\sqrt{\left(\frac{M+1}{M-1}\right)^2-1} \right)\]
		such that $|x+1|/|x-1|\geq(M+1)/(M-1)>|z|$. According to the choice of $\rho$ and $L$, we have
		\[\frac{L}{\rho^2-1}=\frac{1}{(\rho^2-1)(\frac{\rho^2-1}{2\rho}-\frac{2}{\sqrt{M}})}=\frac{2\rho}{(\rho^2-1)[(\rho-\frac{2}{\sqrt{M}})^2-\frac{4}{M}-1]}=\frac{\rho M(M-1)^2}{(\rho^2-1)(4M-2)}\leq \frac{\rho(M-1)^3}{16}\leq\frac{3M^3}{16}.\]
		
		We then split the proof into the following two cases:
		\begin{itemize}
			\item[Case 1. ] If $M\in [1,4]$, then $\rho\geq 7/3$ and $L\leq 7/2$. The result holds.
			
			\item[Case 2. ] If $M\geq4$, we have \[\frac{\log\left(\sqrt{64 L/15(\rho^2-1)\epsilon}\right)}{\log(\rho)}\leq \frac{\log\left(\sqrt{4M^3 /5\epsilon}\right)}{\log\left(1+2/\sqrt{M}\right)}\leq\sqrt{M}\log\left(\sqrt{4M^3 /5\epsilon}\right),\]
			where the first inequality is due to $L/(\rho^2-1)\leq3M^3/16$ and $t/(1+t)\leq\log(1+t)$ for any $t>0$ and the second one is due to $M\geq4$.
			Therefore, we can choose 
			$N=O(\sqrt{M}\log(\sqrt{M^3 /\epsilon})).$ 
		\end{itemize}	
		%
		Finally, the SOC representation of $\hat{H}_{\phi_1,N}$ follows from the fact that the hypograph of $\phi_1(t_k,x)$ is SOC representable for each $k\in [N]$. 
		\QEDA
	\end{proof}
	The result in Corollary \ref{cor_phi_1} holds only when $a=1$ in Example~\ref{emp_1}. Our numerical study 
	shows that letting $a=1$ performs nearly the best among all the testing instances. For general $a>0$, we are unable to derive closed-form $\rho$ and $L$ and thus leave it to the interested readers. 
	
	\subsubsection{Example 2}\label{sec_ex2}
	\begin{example}\label{emp_2}
		Suppose $\phi_2(t,x)=\frac{2^{s}(x^{1/2^{s}}-1)}{2+(t+1)(x^{1/2^{s}}-1)}$ for some positive parameter $a>0$ and positive integer $s\in \Ze_{++}$. Then we have $\int_0^1 \phi_2(t,x)dt=\log(x)$ and for any $t\in [0,1]$, the hypograph of $\phi_2(t,\cdot)$ is SOC representable, i.e.,
		\begin{align}\label{eq_phi_rep_2}
			&\hyp {\phi_2}(t):=\left\{(x,v)\in \R\times \Re: \begin{aligned}
				&r_i\geq 0, \forall i\in [s],\\
				&(2r_1,x-1,x+1)\in \L_3, \\
				&(2r_{i+1},r_{i}-1,r_{i}+1)\in \L_3, \forall i\in [s-1],\\
				&\gamma_{1}=2^s-(t+1)v\geq 0, \\
				&\gamma_{2}=r_s-1-2^{1-s} v \geq 0,\\
				&(v\sqrt{2^{3-s}(t+1)},\gamma_{1}-\gamma_{2},\gamma_{1}+\gamma_{2})\in \L_3
			\end{aligned}
			\right\}.
		\end{align}
		This generating function was studied in \cite{fawzi2019semidefinite} and was shown to be semidefinite conic representable. 
	\end{example}
Next, we study the number of variables and SOC constraints needed to achieve the $\epsilon-$approximation accuracy using Example \ref{emp_2}.

		\begin{corollary}\label{cor_phi_2} For the generating function in Example \ref{emp_2}, we can choose $\rho=(M^{1/2^{s}}+1)/(M^{1/2^{s}}-1)$, $s=N$, and
			\[L=\frac{2^{s}}{\frac{M^{1/2^{s}}+1}{M^{1/2^{s}}-1}-\frac{\rho}{2}-\frac{\rho^{-1}}{2}}.\]
			Thus, achieving the $\epsilon-$approximation accuracy requires the number of points $N=O(\log\log(M)+\sqrt{\log(M)+\log(1/\epsilon)})$. That is, there should be $O(\log\log(M)+\sqrt{\log(M)+\log(1/\epsilon)})$ number of variables and SOC constraints in the representation of set $\hat{H}_{\phi_2,N}$, which admits
			\begin{equation}
				\hat{H}_{\phi_2,N} = \left\{(x,\nu)\in \R\times\Re: \begin{aligned}
					&\sum_{k\in [N]}w_k v_k\geq \nu, r_i\geq 0, \forall i\in [s],\\
					&(2r_1,x-1,x+1)\in \L_3, \\
					&(2r_{i+1},r_{i}-1,r_{i}+1)\in \L_3, \forall i\in [s-1],\\ 
					&\gamma_{1k}=2^s-(t_k+1)v_k\geq 0, \forall k\in [N],\\
					&\gamma_{2k}=r_s-1-2^{1-s} v_k \geq 0, \forall k\in [N],\\
					&(v_k\sqrt{2^{3-s}(t_k+1)},\gamma_{1k}-\gamma_{2k},\gamma_{1k}+\gamma_{2k})\in \L_3, \forall k\in [N]
				\end{aligned}\right\}.
			\end{equation}
		\end{corollary}
		\begin{proof}
			We choose $L$ in the following way
			\[\left|\phi(z,x)\right|=\left|\frac{2^{s}(x^{1/2^{s}}-1)}{2+(z+1)(x^{1/2^{s}}-1)}\right|=\frac{2^{s}}{\left|\frac{x^{1/2^{s}}+1}{x^{1/2^{s}}-1}+z\right|}\leq\frac{2^{s}}{\left|\left|\frac{x^{1/2^{s}}+1}{x^{1/2^{s}}-1}\right|-|z|\right|}
			\leq\frac{2^{s}}{\frac{M^{1/2^{s}}+1}{M^{1/2^{s}}-1}-\frac{\rho}{2}-\frac{\rho^{-1}}{2}}:=L,\]
			where the first inequality is because of triangle inequality and the second one holds by choosing
			\[\rho=(M^{1/2^{s}}+1)/(M^{1/2^{s}}-1)\in\left(1,\frac{M^{1/2^{s}}+1}{M^{1/2^{s}}-1}+\sqrt{\left(\frac{M^{1/2^{s}}+1}{M^{1/2^{s}}-1}\right)^2-1} \right)\]
			such that 
			\[\frac{|x^{1/2^{s}}+1|}{|x^{1/2^{s}}-1|}\geq\frac{M^{1/2^{s}}+1}{M^{1/2^{s}}-1}>|z|. \]
			Then we have
			\[\frac{L}{\rho^2-1}=\frac{2^{s+1}\rho}{(\rho^2-1)^2}=\frac{2^{s-3}(M^{1/2^{s}}+1)(M^{1/2^{s}}-1)^3}{M^{2/2^{s}}}\leq 2^{s-3}M^{2/2^{s}}.\]\\
			Letting $s=N$, we arrive at
			\[\frac{64 L\rho^{-2N}}{15(\rho^2-1)} \leq \frac{2^{N+3}M^{2/2^{N}}\rho^{-2N}}{15}  \leq \frac{2^{-2N^2-N+3}M\log^{2N}(M)}{15} 
			\leq \epsilon,\]
			where the second inequality is due to  
			\[\frac{M^{1/2^{N}}-1}{M^{1/2^{N}}+1}=\tanh\left(\frac{1}{2}\log\left(M^{1/2^{N}}\right)\right)\leq \frac{\log(M)}{2^{N+1}}. \] 
			Solving the following inequality
			\[2\log(2)N^2+(\log(2)-2\log\log(M))N+\log(15\epsilon/8M)\geq0,\]
			we have
			\[N\geq\left[2\log\log(M)-\log(2)+\sqrt{[\log(2)-2\log\log(M)]^2-8\log(2)\log(15\epsilon/8M)}\right]/4\log(2).\]
			Thus, $N=s=O(\log\log(M)+\sqrt{\log(M)+\log(1/\epsilon)})$.

			%
			
			Finally, the SOC representation of $\hat{H}_{\phi_2,N}$ follows from the fact that the hypograph of $\phi_2(t_k,x)$ is SOC representable for each $k\in [N]$. \QEDA
		\end{proof}
		The generating function $\phi_2(t,x)$ has also been studied in \cite{fawzi2019semidefinite}. We differentiate from their work in the following three aspects: (i) We prove Corollary \ref{cor_phi_2} from a different angle and significantly simplify the proof; (ii) We focus on the SOC approximation, while \cite{fawzi2019semidefinite} focus on semidefinite conic approximation; and (iii) The proof technique here can be applied to other generating functions. 
		
	\subsubsection{Example 3}\label{sec_ex3}
	\begin{example}\label{emp_3}
		Suppose $\phi_3(t,x)=\frac{(x/\hat{x}-1)}{2+(t+1)(x/\hat{x}-1)}$ for $\hat{x}\in \R$. Then we have $\int_{-1}^1 \phi_3(t,x)dt=\log(x)-\log(\hat{x})$ and for any $t\in [-1,1]$, the hypograph of $\phi_3(t,\cdot)$ is SOC representable, i.e.,
		\begin{align}\label{eq_phi_rep_3}
			&\hyp {\phi_3}(t):=\left\{(x,v)\in \R\times \Re: \begin{aligned}
				&r_1=1-(t+1)v\geq 0, \\
				&r_2=(x/\hat{x}-1)-2v\geq 0,\\
				&(v\sqrt{8(t+1)},r_1-r_2,r_1+r_2)\in \L_3
			\end{aligned}
			\right\}.
		\end{align}
		%
		%
		%
		%
	In this example, if we obtain a near-optimal solution $\hat{x}$, we can scale the variable $x$ by $\hat{x}$. The scaling method enables us to focus on the neighborhood around the solution $\hat{x}$ and meanwhile maintains the SOC representability of the hypograph of the scaled generating function.
	\end{example}	
Next, we study the number of variables and SOC constraints needed to achieve the $\epsilon-$approximation accuracy using Example \ref{emp_3}.

	\begin{corollary}\label{cor_phi_3} For the generating function in Example \ref{emp_3}, suppose $\hat{x}$ is a near-optimal solution such that $1-\delta\leq x/\hat{x} \leq 1+\delta$ with $\delta\in(0,1)$ and we choose $\rho=(2-\delta)/\delta$ and $L=((2-\delta)/\delta-\rho/2-\rho^{-1}/2)^{-1}$.
		Then, achieving the $\epsilon-$approximation accuracy requires the number of points $N=O(\log(1-\delta)/\log(\delta)+\log(\epsilon)/\log(\delta))$ such that
		$\hyp {\log}(-\epsilon)\subseteq  \hat{H}_{\phi_3,N} \subseteq \hyp {\log}(\epsilon)$. That is, there should be $O(\log(1-\delta)/\log(\delta)+\log(\epsilon)/\log(\delta))$ number of variables and SOC constraints in the representation of set $\hat{H}_{\phi_3,N}$, which admits
		\begin{equation}
			\hat{H}_{\phi_3,N} = \left\{(x,\nu)\in \R\times\Re: \begin{aligned}
				&\sum_{k\in [N]}w_k v_k\geq \nu-\log(\hat{x}),\\
				&r_{1k}=1-(t_k+1)v_k\geq 0, \forall k\in [N],\\
				&r_{2k}=(x/\hat{x}-1)-2 v_k \geq 0, \forall k\in [N],\\
				&(v_k\sqrt{8(t_k+1)},r_{1k}-r_{2k},r_{1k}+r_{2k})\in \L_3, \forall k\in [N]
			\end{aligned}\right\}.
		\end{equation}
	\end{corollary}
	\begin{proof}
		We derive $L$ for the generating function in Example \ref{emp_3} as follows
		\[\left|\phi(z,x)\right|=\left|\frac{(x/\hat{x}-1)}{2+(z+1)(x/\hat{x}-1)}\right|=\frac{1}{\left|\frac{x/\hat{x}+1}{x/\hat{x}-1}+z\right|}\leq\frac{1}{\left|\left|\frac{x/\hat{x}+1}{x/\hat{x}-1}\right|-|z|\right|}\leq \frac{1}{\frac{2-\delta}{\delta}-\frac{\rho}{2}-\frac{\rho^{-1}}{2}}=L,\]
		where the first inequality is because of triangle inequality and the second one holds by choosing
		\[\rho=(2-\delta)/\delta\in\left(1,\frac{2-\delta}{\delta}+\sqrt{\left(\frac{2-\delta}{\delta}\right)^2-1} \right)\]
		such that 
		\[\frac{\left|x/\hat{x}+1\right|}{\left|x/\hat{x}-1\right|}\geq\frac{2-\delta}{\delta}>|z|.\]
		Then, we get
		\[\frac{L}{\rho^2-1}=\frac{2\rho}{(\rho^2-1)^2}=\frac{\delta^3(2-\delta)}{8(1-\delta)^2}.\]
		Now let
		\[\frac{64 L\rho^{-2N}}{15(\rho^2-1)}=\frac{8\delta^{2N+3}}{15(2-\delta)^{2N-1}(1-\delta)^2}\leq \epsilon,\]
		and we have
		\[N\geq\frac{\log(\delta^3(2-\delta))-\log(15(1-\delta)^2 \epsilon/8)}{2\log((2-\delta)/\delta)}.\]
		Therefore, we can choose $N=O(\log(1-\delta)/\log(\delta)+\log(\epsilon)/\log(\delta))$. 

		%
		The SOC representation of $\hat{H}_{\phi_3,N}$ follows from the fact that the hypograph of $\phi_3(t_k,x)$ is SOC representable for each $k\in [N]$. 
		\QEDA	
	\end{proof}
	For Corollary \ref{cor_phi_3}, we remark that (i) if $\delta$ is close to 1 (i.e., the solution quality of $\hat{x}$ is quite low), then $N=O(\log(1-\delta)/\log(\delta))$, implying that the analysis may not be tight; 
	(ii) if $\delta\rightarrow\epsilon$ (i.e., the solution quality is good), then $N=O(\log(\epsilon)/\log(\delta))$; particularly, if $\delta=\epsilon^{1/t}$ for some positive $t\geq 1$, then we only need $N=O(t)$ points to achieve the $\epsilon$-approximation accuracy. This result, together with Figure \ref{taylor_approx} may explain why Example \ref{emp_3} works so well in our numerical study.

	\subsubsection{Comparisons of Example \ref{emp_1}-Example \ref{emp_3}}
	
	\begin{wrapfigure}{r}{0.50\textwidth}
	\vspace{-30pt}
	\begin{center}
		\includegraphics[width=0.48\textwidth]{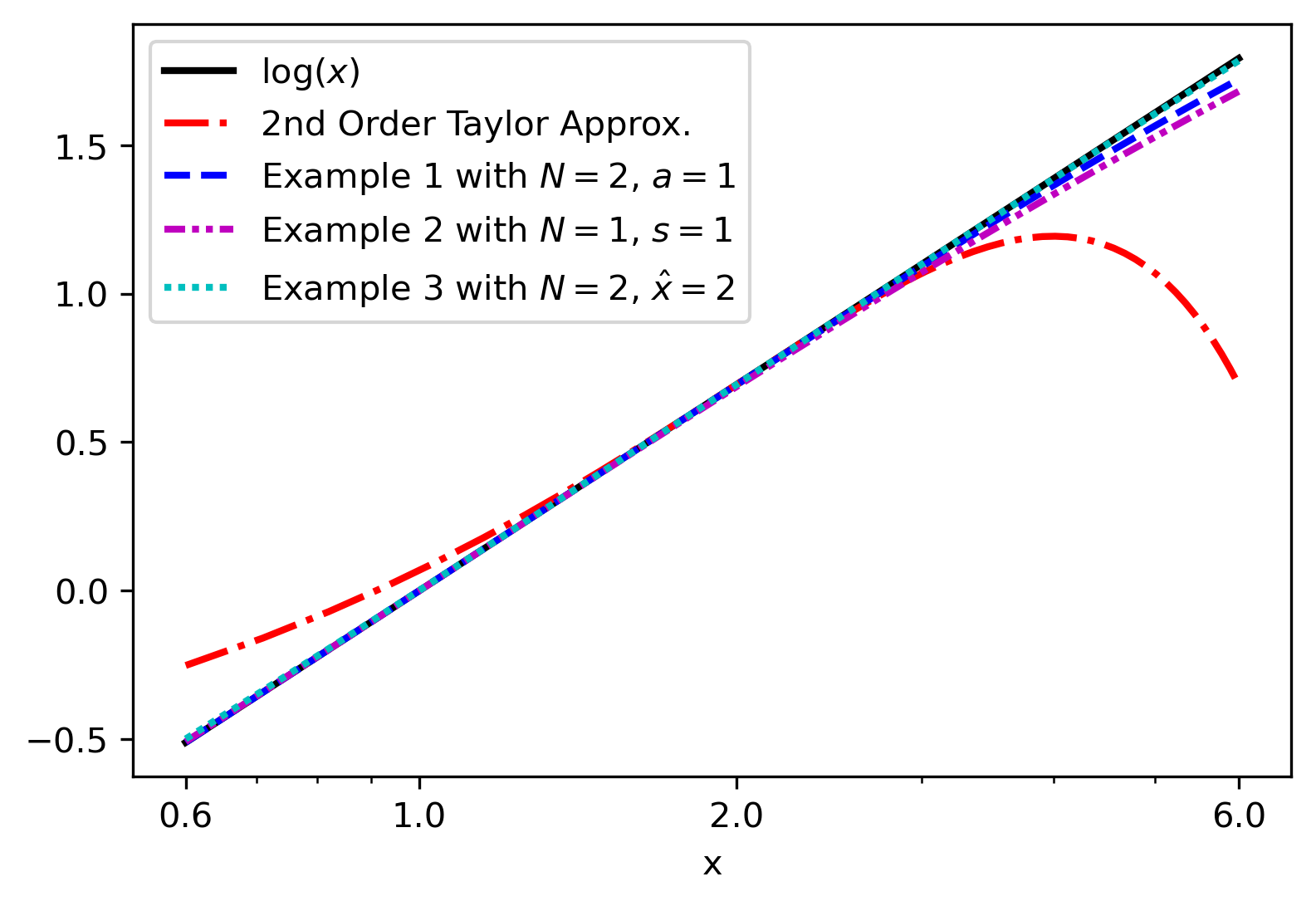}
		\vspace{-20pt}
	\end{center}
	\tiny\caption{A Comparison of the Second-order Taylor Approximation and Example \ref{emp_1}-Example \ref{emp_3} using Two SOC Constraints to Approximate the Logarithm Function in the Domain $[0.6,6.0]$}\label{taylor_approx}
	\vspace{-60pt}
	\end{wrapfigure}

	The numerical comparison of Example \ref{emp_1}-Example \ref{emp_3} versus second-order Taylor approximation to approximate the logarithm function can be found in Figure \ref{taylor_approx}, where we let $\hat{x}=2$ for Example \ref{emp_3}. All three examples have better approximation performance than the second-order Taylor approximation. 
	By applying the scaling method, Example \ref{emp_3} outperforms the other methods, which implies even two-point Gaussian quadrature can be a good approximation of $\log (x)$ as $\log(x)\approx \log(\hat{x})+\phi_3(\sqrt{1/3},x)+\phi_3(-\sqrt{1/3},x)$. In practice, we can run heuristics to obtain a good-quality $\hat{x}$.

	\subsection{Approximating the Exponential Cone}\label{sec_approx_ec_exmps}
	Since the exponential cone in the form of \eqref{eq_ec1} can be viewed as the hypograph of perspective of the logarithm, the SOC approximations in the previous subsection can be directly applied to the exponential cone \eqref{eq_ec1}. This result is summarized below:
	\begin{theorem}\label{prop_approx_exp_cone}
		Suppose that there exists some $\rho>1$ such that
		\[\sup_{x\in [1/M^2,M^2], z\in \Ce}\left\{|\phi(z,x)|: z=\frac{\rho}{2}\exp(i\theta)+
		\frac{\rho^{-1}}{2}\exp(-i\theta):\theta\in [0,2\pi]
		\right\} \leq L<\infty.\]
		Then if the number of points $N=O(\log_{\rho}(L/\epsilon))$ and set
		%
		\[\hat{K}_{\phi,N}^{sc}:=\left\{\bm{x}\in \R^{2}\times \Re: \sum_{k\in [N]}w_k v_k\geq x_3+\textrm{const}\cdot x_2, (x_1/x_2,v_k/x_2)\in \hyp {\phi}(t_k),\forall k\in [N]\right\},\]
		the following approximation result holds:
		$K_{\exp}(-\epsilon)\subseteq \hat{K}_{\phi,N}^{sc}\subseteq K_{\exp}(\epsilon).$
	\end{theorem}
	\begin{proof}
		According to \eqref{eq_log_approx}, for any $\alpha$, we have
		\[ K_{\exp}(\alpha)=\left\{\bm{x}\in \R^2\times \Re: (x_1/x_2,x_3/x_2)\in \hyp {\log}(\alpha)\right\}\]
		and 
		\[\hat{K}_{\phi,N}^{sc}=\left\{\bm{x}\in \R^{2}\times \Re:(x_1/x_2,x_3/x_2)\in \hat{H}_{\phi,N}\right\}.\]
		Thus, according to Theorem \ref{prop_approx_log}, for any $\epsilon >0$, letting $N=O(\log_{\rho}(L/\epsilon))$, we have $K_{\exp}(-\epsilon)\subseteq  \hat{K}_{\phi,N}^{sc} \subseteq K_{\exp}(\epsilon)$ being equivalent to $\hyp {\log}(-\epsilon)\subseteq  \hat{H}_{\phi,N} \subseteq \hyp {\log}(\epsilon)$.

		%
		
		For the SOC approximation, we recall that
		\[\hat{H}_{\phi,N}=\left\{(x,\nu)\in \R\times\Re: \sum_{k\in [N]}w_k v_k\geq \nu+\textrm{const}, (x,v_k)\in \hyp {\phi}(t_k),\forall k\in [N]\right\}.\]
		Hence, we have
		\[\hat{K}_{\phi,N}^{sc}=\left\{\bm{x}\in \R^{2}\times \Re:\sum_{k\in [N]}w_k v_k\geq x_3/x_2+\textrm{const}, (x_1/x_2,v_k)\in \hyp {\phi}(t_k),\forall k\in [N]\right\}.\]
		Letting $v_k:=v_kx_2$, we arrive at the desired formulation. Finally, since 
		set $ \hyp {\phi}(t_k)$ is SOC representable, according to \cite{ben2001lectures}, set $\hat{K}_{\phi,N}^{sc}$ is also SOC representable.
		\QEDA
	\end{proof}

	We conclude this subsection by making the following remarks about Example \ref{emp_1}-Example \ref{emp_3}.
	\begin{subequations}
		\begin{itemize}
			\item For the generating function in Example \ref{emp_1}, if there are $O(M\log(M^3/\sqrt{\epsilon}))$ number of variables and SOC constraints in the representation of set $\hat{K}_{\phi_1,N}^{sc}$, then we have $K_{\exp}(-\epsilon)\subseteq \hat{K}_{\phi_1,N}^{sc}\subseteq K_{\exp}(\epsilon)$, where
			\begin{equation}
				\hat{K}_{\phi_1,N}^{sc} =\left\{\bm{x}\in \R^{2}\times \Re: \begin{aligned}
					&\sum_{k\in [N]}w_k v_k\geq x_3,\\
					&r_{1k}=ax_2-(t_k+1)v_k\geq 0, \forall k\in [N],\\
					&r_{2k}=a(t_k+1)^{a-1}(x_1-x_2)-2^a v_k \geq 0, \forall k\in [N],\\
					&(v_k\sqrt{2^{a+2}(t_k+1)},r_{1k}-r_{2k},r_{1k}+r_{2k})\in \L_3, \forall k\in [N]
				\end{aligned}\right\};
			\end{equation}
			%
			%
			\item For the generating function in Example \ref{emp_2}, if there are $O(\log\log(M)+\sqrt{\log(M)+\log(1/\epsilon)})$ number of variables and SOC constraints in the representation of set $\hat{K}_{\phi_2,N}^{sc}$, 
			then we have $K_{\exp}(-\epsilon)\subseteq \hat{K}_{\phi_2,N}^{sc}\subseteq K_{\exp}(\epsilon)$, where
			\begin{equation}
				\hat{K}_{\phi_2,N}^{sc} =\left\{\bm{x}\in \R^{2}\times \Re: \begin{aligned}
					&\sum_{k\in [N]}w_k v_k\geq x_3,\\
					&r_i\geq 0, \forall i\in [s],\\
					&(2r_1,x_1-x_2,x_1+x_2)\in \L_3, \\
					&(2r_{i+1},r_{i}-x_2,r_{i}+x_2)\in \L_3, \forall i\in [s-1],\\
					&\gamma_{1k}=2^sx_2-(t_k+1)v_k\geq 0, \forall k\in [N],\\
					&\gamma_{2k}=r_s-x_2-2^{1-s} v_k \geq 0, \forall k\in [N],\\
					&(v_k\sqrt{2^{3-s}(t_k+1)},\gamma_{1k}-\gamma_{2k},\gamma_{1k}+\gamma_{2k})\in \L_3, \forall k\in [N]
				\end{aligned}\right\};
			\end{equation}
			%
			%
			
			and
			\item For the generating function in Example \ref{emp_3}, if $\hat{\bm x}$ is a near-optimal solution such that $1-\delta\leq (x_1/x_2)/(\hat{x}_1/\hat{x}_2) \leq 1+\delta$ with $\delta\in(0,1)$ and there are $O(\log(1-\delta)/\log(\delta)+\log(\epsilon)/\log(\delta))$ number of variables and SOC constraints in the representation of set $\hat{K}_{\phi_3,N}^{sc}$, then we have $K_{\exp}(-\epsilon)\subseteq \hat{K}_{\phi_3,N}^{sc}\subseteq K_{\exp}(\epsilon)$, where
			\begin{equation}
				\hat{K}_{\phi_3,N}^{sc} = \left\{\bm{x}\in \R^{2}\times \Re: \begin{aligned}
					&\sum_{k\in [N]}w_k v_k\geq x_3/\hat{x}_2-(x_2/\hat{x}_2)\log(\hat{x}_1/\hat{x}_2),\\
					&r_{1k}=x_2/\hat{x}_2-(t_k+1)v_k\geq 0, \forall k\in [N],\\
					&r_{2k}=(x_1/\hat{x}_1-x_2/\hat{x}_2)-2v_k \geq 0, \forall k\in [N],\\
					&(v_k\sqrt{8(t_k+1)},r_{1k}-r_{2k},r_{1k}+r_{2k})\in \L_3, \forall k\in [N]
				\end{aligned}\right\}.
			\end{equation}   
			%
			%
		\end{itemize}
	\end{subequations}
	
	\section{Approximations Using the Exponential Form \eqref{eq_ec2}}\label{sec_exp}
	
	In this section, we extend the analysis of the approximation result for the exponential function to the exponential cone \eqref{eq_ec2}. We also show that the even-order Taylor approximation of the exponential function is sum-of-squares representable and disprove an SOC representation proposed by \cite{ben2001lectures}. Following the spirit of Example \ref{emp_3} in the previous section, we derive a new shifting approximation scheme using an approximate solution, which overcomes the numerical issues caused by other approximation schemes studied in this section.
	\subsection{Approximating the Exponential Function and the Exponential Cone}\label{sec_exp_lim}
	Given the exponential function $\exp(x)$, it is well known that
	$\lim_{n\rightarrow \infty}(1+x/n)^n =\exp(x)$
	for any $x\in\Re$. This motivates us to approximate $\exp(x)$ by $(1+2^{-N}x)^{2^N} $ for some positive integer $N
	\in \Ze_{++}$. This approximation scheme has been studied in proposition 2.3.7 \citep{ben2001lectures}. However, they did not study the approximation of the exponential cone. This paper fills this gap. We first show the approximation result of the exponential function. 
	\begin{proposition}\label{prop_approx_exp}
		Suppose that $x\in [-\hat{L},\hat{L}]$ for some positive parameter $L\in \Re_+$ and 
		for any $\epsilon\in (0,\hat{L})$, let $N=O(\log(\hat{L})+\log(1/\epsilon))$. 
		Then we have the following approximation result
		\[ \exp(x-\epsilon)\leq \psi_N(x)=(1+2^{-N}x)^{2^N} \leq \exp(x).\]
		Particularly, the epigraph of $\psi_N(x)$ is SOC representable, i.e.,
		\[\epi\psi_N:=
		\left\{(x,v)\in \Re\times\Re: \begin{aligned}
		&r_k\geq 0, \forall k\in [N-1], \\
		&(2r_1,v-1,v+1)\in \L_3, \\
		&(2r_{k+1},r_{k}-1,r_{k}+1)\in \L_3, \forall k\in [N-1],\\
		&r_N=1+2^{-N}x\end{aligned}\right\}.
		\]
	\end{proposition}
	\begin{proof}
		Let $N=\log(4\hat{L}^2/\epsilon)$. Clearly, we have $2^N>2\hat{L}$.  
		Then for $x\in [-\hat{L},\hat{L}]$, setting $y=x/2^N$, we have $|y|\leq1/2$ and
		\[\exp(y-4y^2) \leq 1+y \leq \exp(y).\]
		Consequently,
		\[\exp(2^Ny-2^{N+2}y^2)\leq \left(1+y\right)^{2^N} \leq \exp(2^Ny).\]
		Since $x=2^Ny$, 
		\[\exp(x-x^2/2^{N-2})\leq \left(1+x/2^N\right)^{2^N} \leq \exp(x).\]
		Letting $N=\log(4\hat{L}^2/\epsilon)$, 
		we have the following approximation bound
		\[ \exp(x-\epsilon)\leq \left(1+x/2^N\right)^{2^N} \leq \exp(x), \forall x\in [-\hat{L},\hat{L}].\] 
		%
		
		The SOC representation for the epigraph of $\psi_N(x)$ uses the towers of variables and can also be found in \cite{ben2001lectures}.
		\QEDA
	\end{proof}
	Note that this result is different from proposition 2.3.7 \citep{ben2001lectures}. The latter showed $(1-\epsilon) \exp(x)\leq \psi_N(x)=(1+2^{-N}x)^{2^N} \leq \exp(x)$, which is not applicable to the approximation form of the exponential cone \eqref{eq_ec2}.
	
	Since the exponential cone in the form of \eqref{eq_ec2} can be viewed as the epigraph of perspective of the exponential function, the SOC representation of the approximation in Proposition \ref{prop_approx_exp} can be directly applied to the exponential cone \eqref{eq_ec2}.
	\begin{theorem}\label{prop_approx_exp_cone2}
		Suppose that $x_3\in [-2M\log M, 2M\log M]$ and for any $\epsilon \in (0,2M^2\log(M))$, let $N=O(\log(M^2\log M)+\log(1/\epsilon))$ and SOC representable set $\hat{K}_{\psi,N}^{sc}:=\{\bm{x}\in \R^{2}\times \Re: (x_1/x_2,x_3/x_2)\in \epi(\psi_N) \}$. 
		Then, we have
		$K_{\exp}(0)\subseteq \hat{K}_{\psi,N}^{sc}\subseteq K_{\exp}(\epsilon).$ Particularly, 
		\[\hat{K}_{\psi,N}^{sc}:=\left\{\bm{x}\in \R^{2}\times \Re: \begin{aligned}
		&r_k\geq 0, \forall k\in [N-1], \\
		&(2r_1,x_1-x_2,x_1+x_2)\in \L_3, \\
		&(2r_{k+1},r_{k}-x_2,r_{k}+x_2)\in \L_3, \forall k\in [N-1],\\
		&r_N=x_2+2^{-N}x_3\end{aligned}\right\}.\]
	\end{theorem}
	\begin{proof}
		Since $\hat{K}_{\psi,N}^{sc}:=\{\bm{x}\in \R^{2}\times \Re: (x_1/x_2,x_3/x_2)\in \epi(\psi_N) \}$ and $N=O(\log(M^2\log M)+\log(1/\epsilon)$, the approximation result follows from Proposition \ref{prop_approx_exp} by letting $\hat{L}=2M^2\log(M)$.
		Besides, since 
		set $ \epi(\psi_N) $ is SOC representable, according to \cite{ben2001lectures}, set $\hat{K}_{\psi,N}^{sc}$ is also SOC representable.
		%
		\QEDA
	\end{proof}
	It is worthy of mentioning that the approximation $\hat{K}_{\psi,N}^{sc}$ obtained in Theorem \ref{prop_approx_exp_cone2} is an outer one given that $N\geq\log_{2}(2M^2\log(M))$.

	\subsection{Strengthening Using Even-Order Taylor Expansions}\label{sec_exp_taylor}
	
	The approximation in the previous subsection relies on the fact that $\exp(x)=(\exp(x/2^N))^{2^N}$ and $\exp(x/2^N)\approx 1+x/2^N$. To strengthen it, we approximate $\exp(x/2^N)$ more tightly using its $2s$-order Taylor expansion with $s\geq 1$. Particularly, letting $y=x/2^N$, we have the following approximation results
	\begin{align}
		\exp(y)\approx \hat{\psi}_{N,2s}(y):=\sum_{i\in [0,2s]}\frac{y^i}{i!} \;\mathrm{and}\; \exp(x)\approx [\hat{\psi}_{N,2s}(y)]^{2^N}.
	\end{align}
	
	We first observe that the polynomial function $\hat{\psi}_{N,2s}(y)\geq 0$ for all $y\in \Re$ and thus is sum-of-squares representable.
	\begin{proposition}\label{prop_nonnega} For any $y$, the polynomial function $\hat{\psi}_{N,2s}(y)\geq 0$ and is sum-of-squares representable.
	\end{proposition}
	\begin{proof}
		We will prove that the minimum of $\hat{\psi}_{N,2s}(y)$ is nonnegative for any $y$ using induction
		When $s=1$, we have $\hat{\psi}_{N,2}(y)=1/2+(y+1)^2/2>0$.  Suppose $\hat{\psi}_{N,2t}(y)\geq0$ for any $s\leq t$. Now let $s=t+1$.
		According to induction, we have $d^2\hat{\psi}_{N,2t+2}/{dy^2}=\hat{\psi}_{N,2t}(y)\geq0$ and thus $\hat{\psi}_{N,2t+2}$ is convex. Hence, $\min_{y}\hat{\psi}_{N,2s}(y)$ is an unconstrained convex optimization problem. 
		An optimal solution $y^*$ must satisfy the first order optimality condition ${d\hat{\psi}_{N,2t+2}}(y^*)/{dy}=\hat{\psi}_{N,2t+1}(y^*)=0$. 
		Thus, we have
		\[\min_{y}\hat{\psi}_{N,2s}(y):=\hat{\psi}_{N,2t+2}(y^*)=\hat{\psi}_{N,2t+1}(y^*)+\frac{(y^*)^{2t+2}}{(2t+2)!}\geq0.\]
		Therefore, the function $\hat{\psi}_{N,2s}(y)\geq 0$ for any $y$.
		
		It is known that a nonnegative polynomial function can be written as a sum of squares (see, e.g., \cite{blekherman2012nonnegative}).
		\QEDA
	\end{proof}
	Proposition \ref{prop_nonnega} shows that the function $\hat{\psi}_{N,2s}$ is sum-of-squares representable. However, it is unknown how to represent the function $\hat{\psi}_{N,2s}$ as an SOC program. In \cite{ben2001lectures}, the authors suggested reformulating it as
	\begin{align}\label{eq_fun}
		\hat{\psi}_{N,2s}(y)=\sum_{j\in [0,s]}\frac{\alpha_j}{(2j)!}(\beta_j +y)^{2j}
	\end{align}
	in a hope that $\alpha_j\geq 0$ for all $j\in [0,s]$. We disprove this claim by numerically showing that some components of vector $\bm{\alpha}$ become negative whenever $s\geq 34$. We first show the following efficient way of finding coefficients of function $\hat{\psi}_{N,2s}(y)$.
	\begin{proposition}\label{prop_conic_sos}
		The coefficients in \eqref{eq_fun} can be found recursively as
		\begin{align}\label{eq_coeff}
			\alpha_{s}=1,\beta_s=1, \sum_{j\in [0,k]}\frac{\alpha_{s-j}\beta_{s-j}^{2k-2j}}{(2k-2j)!}=1,
			\sum_{j\in [0,k]}\frac{\alpha_{s-j}\beta_{s-j}^{2k-2j+1}}{(2k-2j+1)!}=1, \forall k \in [0,s-1].
		\end{align}
	\end{proposition}
	\begin{proof}
		See Appendix \ref{sec_proof_prop_conic_sos}.
	\end{proof}

	Below, we illustrate some coefficients in Proposition \ref{prop_conic_sos} and disprove the claim in \cite{ben2001lectures}
	\begin{itemize}
		\item If $s=1$, we can reformulate $\hat{\psi}_{N,2s}(y)$ as
		\begin{align*}
			\hat{\psi}_{N,2}(y)=\frac{1}{2}+\frac{1}{2}(y+1)^{2}
		\end{align*}
		\item If $s=2$, we can reformulate $\hat{\psi}_{N,2s}(y)$ as
		\begin{align*}
			\hat{\psi}_{N,4}(y)=\frac{19}{72}+\frac{1}{4}\left(y+\frac{5}{3}\right)^{2}+\frac{1}{24}(y+1)^{4}
		\end{align*}
		\item If $s=40$, please find $(\bm{\alpha},\bm{\beta})$ in Table \ref{s40_alpha_beta}. We see that $\alpha_0,\ldots,\alpha_6<0$. In fact, according to the recursion in Proposition \ref{prop_conic_sos}, we are unable to represent $\hat{\psi}_{N,2s}(y)$ in the way of \eqref{eq_fun} whenever $s\geq 34$.
	\end{itemize}
	Although we firmly believe that $\hat{\psi}_{N,2s}(y)$ is SOC representable for any $s\in\Ze_{++}$, we are not able to prove it. Therefore, we leave it as an open question to the interested readers.
	
	\begin{table}[htbp]
		\centering
		\caption{The values of $\bm\alpha$ and $\bm\beta$ for $s=40$. Note: $\alpha_{40}=\beta_{40}=1$.}
		\vspace{-5pt}
		\label{s40_alpha_beta}
		\scriptsize\setlength{\tabcolsep}{3.0pt}
		\begin{tabular}{c|rrrrr||c|rrrrr}
			\hline
			\multicolumn{1}{c|}{$k$} & \multicolumn{5}{c||}{$\bm\alpha$} & \multicolumn{1}{c|}{$k$} & \multicolumn{5}{c}{$\bm\beta$}     \\ \hline
			0--4 	& -1.90e-2 & -8.09e-3 & -2.46e-3 & -5.00e-4 & -6.16e-5 & 0--4 & 37.11 & 36.86 & 36.56 & 36.16 & 35.59 \\
			5--9    & -3.88e-6 & -8.98e-8 & 1.17e-10 & 5.27e-9 & 1.25e-8 &	5--9 & 34.70 & 33.01 & 59.20 & 20.27 & 19.22 \\ 
			10--14 & 2.44e-8 & 4.45e-8 & 7.94e-8 & 1.40e-7 & 2.47e-7 &	10--14	& 18.50 & 17.88 & 17.29 & 16.72 & 16.16 \\
			15--19 & 4.36e-7 & 7.69e-7 & 1.36e-6 & 2.39e-6 & 4.23e-6 &	15--19	& 15.59 & 15.02 & 14.45 & 13.89 & 13.32 \\ 
			20--24 & 7.46e-6 & 1.32e-5 & 2.33e-5 & 4.13e-5 & 7.31e-5 &	20--24	& 12.75 & 12.18 & 11.61 & 11.04 & 10.47 \\
			25--29 & 1.29e-4 & 2.30e-4 & 4.08e-4 & 7.24e-4 & 1.29e-3 &	25--29	& 9.90  & 9.32  & 8.75  & 8.18  & 7.60  \\
			30--34 & 2.30e-3 & 4.10e-3 & 7.34e-3 & 1.32e-2 & 2.37e-2 &	30--34	& 7.02  & 6.44  & 5.86  & 5.28  & 4.69  \\
			35--39 & 4.28e-2 & 7.77e-2 & 1.42e-1 & 2.64e-1 & 5.00e-1 &	35--39	& 4.10  & 3.51  & 2.91  & 2.30  & 1.67  \\ \hline
		\end{tabular}
	\end{table}
	
	Next, we derive the strengthened approximation results when we use \eqref{eq_fun} to approximate $\exp(x/2^N)$. The result can be applied to any $s\leq 33$.
	
	\begin{proposition}\label{prop_approx_exp_s} 
		Suppose that $x\in [-\hat{L},\hat{L}]$ for some positive parameter $\hat{L}\in \Re_+$, $s\in [33]$, and for any $\epsilon\in (0,\hat{L})$, we let $N=O(\log(\hat{L}/s)+\log(1/\epsilon)/s)$. Then
		we have the following approximation result 
		\[ \exp(x-\epsilon)\leq \overline{\psi}_{N,2s}(x)=(\hat{\psi}_{N,2s}(x/2^N))^{2^N} \leq \exp(x+\epsilon).\]
		Particularly, when $s=1$, the epigraph of $\overline{\psi}_{N,2}(x)$ admits the following SOC representation:
		\[\epi \overline{\psi}_{N,2}:=
		\left\{(x,v)\in \Re\times\Re: \begin{aligned}
		&r_k\geq 0, \forall k\in [N],\\
		&(2r_1,v-1,v+1)\in \L_3, \\
		&(2r_{k+1},r_{k}-1,r_{k}+1)\in \L_3, \forall k\in [N-1],\\
		&(1+2^{-N}x, r_N-1, r_N)\in \L_3\end{aligned}\right\}.
		\]
	\end{proposition}
	\begin{proof}
	For notational convenience, we let $y=x/2^N$. Suppose that $N\geq \log_2(\hat{L})$. Then we have $|y|\leq 1$. 
	
	Next, we prove the approximation result by discussing whether $y\geq 0$ or not.
	\begin{itemize}
		\item[Case 1. ] If $y\geq0$, according to the error formula of Taylor approximation, we have
		\[	0\leq \exp(y)-\hat{\psi}_{N,2s}(y) \leq \frac{y^{2s+1}}{(2s+1)!}\exp(y).\]
		According to the proof of Proposition \ref{prop_approx_exp}, we have
		\[	\exp\left(y-\frac{y^{2s+1}}{(2s+1)!}-\frac{4y^{4s+2}}{((2s+1)!)^2}\right)\leq \hat{\psi}_{N,2s}(y) \leq \exp(y).\]
		Since $x=2^Ny$, we obtain
		\[	\exp\left(x-\frac{x^{2s+1}}{2^{2Ns}(2s+1)!}-\frac{x^{4s+2}}{2^{N(4s+1)-2}((2s+1)!)^2}\right)\leq \overline{\psi}_{N,2s}(x) \leq \exp(x).\]
		Using the facts that $|x|\leq \hat{L}$, $N\geq 2$, and $(2s+1)!\geq ((2s+1)/e)^{2s+1}$,
		by choosing 
		$$N\geq \frac{(2s+1)\log(e\hat{L}/(2s+1))+\log(1/\epsilon)}{2s\log(2)}$$
		or $N=O(\log(\hat{L}/s)+\log(1/\epsilon)/s)$, we have  
		\[ \exp(x-\epsilon)\leq \overline{\psi}_{N,2s}(x) \leq \exp(x+\epsilon), \forall x\in [0,\hat{L}].\]
		\item[Case 2. ] 	Similarly, if $y\leq0$, according to the error formula of Taylor approximation, we have
		\[	0\geq \exp(y)-\hat{\psi}_{N,2s}(y) \geq \frac{y^{2s+1}}{(2s+1)!}.\]
		As $\exp(y)\geq y^2/3$ for any $y\in[-1,0]$, we have
		\[	\exp\left(y-\frac{3y^{2s-1}}{(2s+1)!}\right)\geq \hat{\psi}_{N,2s}(y) \geq \exp(y).\]
		Since $x=2^Ny$, we obtain
		\[	\exp\left(x-\frac{3x^{2s-1}}{2^{N(2s-2)}(2s+1)!}\right)\geq \overline{\psi}_{N,2s}(x)\geq \exp(x).\]
		Similarly, by choosing $N=O(\log(\hat{L}/s)+\log(1/\epsilon)/s)$, we have 
		\[ \exp(x-\epsilon)\leq \overline{\psi}_{N,2s}(x) \leq \exp(x+\epsilon), \forall x\in [-\hat{L},0].\]
		This proves the approximation result.
	\end{itemize}	
	
	The SOC representation for the epigraph of $\overline{\psi}_{N,2}(x)$ follows from Proposition \ref{prop_approx_exp} and Proposition~\ref {prop_conic_sos}.
	\QEDA
\end{proof}
Note that if the conjecture that $\hat{\psi}_{N,2s}(y)$ 
is SOC representable were true for any $s\in \Ze_{++}$, then one could choose $N=s=O(\sqrt{\log(1/\epsilon)})$ given that $\log(\hat{L})$ is a constant. This matches the best approximation result in Corollary \ref{cor_phi_2} of Example \ref{emp_2}.
In our numerical study, we use $s=1$ (i.e., set $\overline{\psi}_{N,2}(x)$) to effectively solve MIECPs, which unfortunately still has a numerical issue that cannot close the gap even if $N$ is very large. We will use the shifting method in the next subsection to resolve it.

Since the exponential cone in the form of \eqref{eq_ec2} can be viewed as the epigraph of perspective of the exponential function, the approximation result, as well as the SOC representations in Proposition \ref{prop_approx_exp_s}, can be directly applicable to the exponential cone \eqref{eq_ec2}.
%
\begin{theorem}\label{prop_approx_exp_cone2_s}
	Suppose that $x_3\in [-2M\log M, 2M\log M]$, $s\in [33]$, and for $\epsilon \in (0,2M^2\log(M))$, let $N=O(\log(M^2\log(M)/s)+\log(1/\epsilon)/s)$. Then we have conic representable $\hat{K}_{\overline{\psi},N,2s}^{sc}:=\{\bm{x}\in \R^{2}\times \Re: (x_1/x_2,x_3/x_2)\in \epi(\overline{\psi}_{N,2s}) \}$  such that
	$K_{\exp}(-\epsilon)\subseteq \hat{K}_{\overline{\psi},N,2s}^{sc}\subseteq K_{\exp}(\epsilon).$ 
	Particularly, when $s=1$, we have
	\[\hat{K}_{\overline{\psi},N,2}^{sc}:=\left\{\bm{x}\in \R^{2}\times \Re: \begin{aligned}
		&r_k\geq 0, \forall k\in [N-1],\\
		&(2r_1,x_1-x_2,x_1+x_2)\in \L_3, \\
		&(2r_{k+1},r_{k}-x_2,r_{k}+x_2)\in \L_3, \forall k\in [N-1],\\
		&(x_2+2^{-N}x_3, r_N-x_2, r_N)\in \L_3\end{aligned}\right\}.\]
\end{theorem}
\begin{proof}
	Since $\hat{K}_{\overline{\psi},N,2s}^{sc}:=\{\bm{x}\in \R^{2}\times \Re: (x_1/x_2,x_3/x_2)\in \epi(\overline{\psi}_{N,2s}) \}$ and $N=O(\log(M^2\log(M)/s)+\log(1/\epsilon)/s)$, the approximation result follows from Proposition \ref{prop_approx_exp_s}  by letting $\hat{L}=2M^2\log(M)$. Besides, since set $\epi(\overline{\psi}_{N,2s}) $ is SOC representable, according to \cite{ben2001lectures}, set $\hat{K}_{\overline{\psi},N,2s}^{sc}$ is also SOC representable.
	\QEDA
\end{proof}

\subsection{Shifting Using Approximate Solutions}\label{sec_exp_shift}
Similar to Example \ref{emp_3}, we propose to extend the approximation schemes in the previous subsections using the shifting method provided that an approximate solution is available.
To begin with, we summarize the results for applying the shifting method to the approximation proposed in Proposition \ref{prop_approx_exp}. 
\begin{proposition}\label{prop_approx_exp_shift} 
	Suppose $\hat{x}$ is a given approximate solution such that $\left|x-\hat{x}\right| \leq \delta$ with $\delta>0$ and for any $\epsilon\in (0,\delta]$, let $N=O(\log(\delta^2/\epsilon))$. Then we have the following approximation bound
	\[ \exp(x-\epsilon)\leq \exp(\hat{x})\psi_N^{sft}(x)=\exp(\hat{x})(1+2^{-N}(x-\hat{x}))^{2^N} \leq \exp(x).\]
	Particularly, the epigraph of $\psi_N^{sft}(x)$ is SOC representable, i.e.,
	\[\epi \psi_N^{sft}:=
	\left\{(x,v)\in \Re\times\Re: \begin{aligned}
		&r_k\geq 0, \forall k\in [N-1],\\
		&(2r_1,v-1,v+1)\in \L_3, \\
		&(2r_{k+1},r_{k}-1,r_{k}+1)\in \L_3, \forall k\in [N-1],\\
		&r_N=1+2^{-N}(x-\hat{x})\end{aligned}\right\}.
	\]
\end{proposition}
\begin{proof} The proof follows directly from that of Proposition \ref{prop_approx_exp} by letting $y=(x-\hat{x})/2^{N}$.
	\QEDA
\end{proof}

Similarly, since the exponential cone in the form of \eqref{eq_ec2} can be viewed as the epigraph of perspective of the exponential function, the shifting results in Proposition \ref{prop_approx_exp_shift} can be directly applicable to the exponential cone \eqref{eq_ec2}.

\begin{theorem}\label{prop_approx_exp_cone2_shift}
	Suppose $\hat{\bm x}$ is an approximate solution such that $\left|x_3/x_2-\hat{x}_3/\hat{x}_2\right| \leq \delta$ 
	with $\delta>0$ and for any $\epsilon\in (0,\delta]$, let $N=O(\log(\delta^2/\epsilon))$ and set $\hat{K}_{\psi,N}^{sc,sft}:=\{\bm{x}\in \R^{2}\times \Re: (x_1/x_2,x_3/x_2)\in \epi(\psi_N^{sft}) \}$. Then we have
	$K_{\exp}(0)\subseteq \hat{K}_{\psi,N}^{sc,sft}\subseteq K_{\exp}(\epsilon).$ Particularly, 
	\[\hat{K}_{\psi,N}^{sc,sft}:=\left\{\bm{x}\in \R^{2}\times \Re: \begin{aligned}
		&r_k\geq 0, \forall k\in [N-1],\\
		&(2\sqrt{\exp(\hat{x}_3/\hat{x}_2)}r_1,x_1-x_2,x_1+x_2)\in \L_3, \\
		&(2r_{k+1},r_{k}-x_2,r_{k}+x_2)\in \L_3, \forall k\in [N-1],\\
		&r_N=x_2+2^{-N}(x_3-x_2(\hat{x}_3/\hat{x}_2))\end{aligned}\right\}.\]
\end{theorem}
\begin{proof}The proof is similar to that of Theorem \ref{prop_approx_exp_cone2} and thus is omitted.
	\QEDA
\end{proof}

Next, we summarize the results for applying the shifting method to the approximation scheme proposed in Section \ref{sec_exp_taylor}.
%
\begin{proposition}\label{prop_approx_exp_s_shift} 
	Suppose $\hat{x}$ is an approximate solution such that $\left|x-\hat{x}\right| \leq \delta$ with $\delta>0$ and for any $\epsilon\in (0,\delta]$, let $N=O(\log(\delta/s)+\log(1/\epsilon)/s)$. Then we have the following approximation result
	\[ \exp(x-\epsilon)\leq \exp(\hat{x})\overline{\psi}_{N,2s}^{sft}(x)=\exp(\hat{x})(\hat{\psi}_{N,2s}((x-\hat{x})/2^N))^{2^N} \leq \exp(x+\epsilon).\]
	Particularly, when $s=1$, the epigraph of $\overline{\psi}_{N,2}^{sft}(x)$ is SOC representable, i.e.,
	\[\epi(\overline{\psi}_{N,2}^{sft}):=
	\left\{(x,v)\in \Re\times\Re: \begin{aligned}
		&r_k\geq 0, \forall k\in [N-1],\\
		&(2r_1,v-1,v+1)\in \L_3, \\
		&(2r_{k+1},r_{k}-1,r_{k}+1)\in \L_3, \forall k\in [N-1],\\
		&(1+2^{-N}(x-\hat{x}), r_N-1, r_N)\in \L_3\end{aligned}\right\}.\]
\end{proposition}
\begin{proof} The proof follows from Proposition \ref{prop_approx_exp_s} by letting $y=(x-\hat{x})/2^N$.
	\QEDA
\end{proof}

\begin{theorem}\label{prop_approx_exp_cone2_s_shift}
	Suppose $\hat{\bm x}$ is an approximate solution such that $\left|x_3/x_2-\hat{x}_3/\hat{x}_2\right| \leq \delta$ 
	with $\delta>0$ and for any $\epsilon\in (0,\delta]$, let $N=O(\log(\delta/s)+\log(1/\epsilon)/s)$ and set $\hat{K}_{\overline{\psi},N,2s}^{sc,sft}:=\{\bm{x}\in \R^{2}\times \Re: (x_1/x_2,x_3/x_2)\in \epi(\overline{\psi}_{N,2s}^{sft}) \}$  such that
	$K_{\exp}(-\epsilon)\subseteq \hat{K}_{\overline{\psi},N,2s}^{sc,sft}\subseteq K_{\exp}(\epsilon).$ Particularly, when $s=1$, we have
	\[\hat{K}_{\overline{\psi},N,2}^{sc,sft}:=\left\{\bm{x}\in \R^{2}\times \Re: \begin{aligned}
		&r_k\geq 0, \forall k\in [N-1],\\
		&(2r_1\sqrt{\exp(\hat{x}_3/\hat{x}_2)},x_1-x_2,x_1+x_2)\in \L_3, \\
		&(2r_{k+1},r_{k}-x_2,r_{k}+x_2)\in \L_3, \forall k\in [N-1],\\
		&(x_2+2^{-N}(x_3-x_2(\hat{x}_3/\hat{x}_2)), r_N-x_2, r_N)\in \L_3\end{aligned}\right\}.\]
\end{theorem}
\begin{proof} The proof is similar to that of Theorem \ref{prop_approx_exp_cone2_s} and thus is omitted.
	\QEDA
\end{proof}

We remark that (i) the number of variables and SOC constraint (i.e., $N$) depends on the approximation accuracy $\epsilon$ and the quality of a given approximate solution. Thus, a better approximate solution reduces $N$; and (ii) in practice, to solve MIECP, there may exist many heuristic methods, and leveraging their near-optimal solutions can help us find an optimal solution more effectively. This phenomenon has been observed in our numerical study.

\subsection{Comparisons of Methods in Section  \ref{sec_exp}}

\begin{wrapfigure}{r}{0.50\textwidth}
	\vspace{-30pt}
	\begin{center}
		\includegraphics[width=0.48\textwidth]{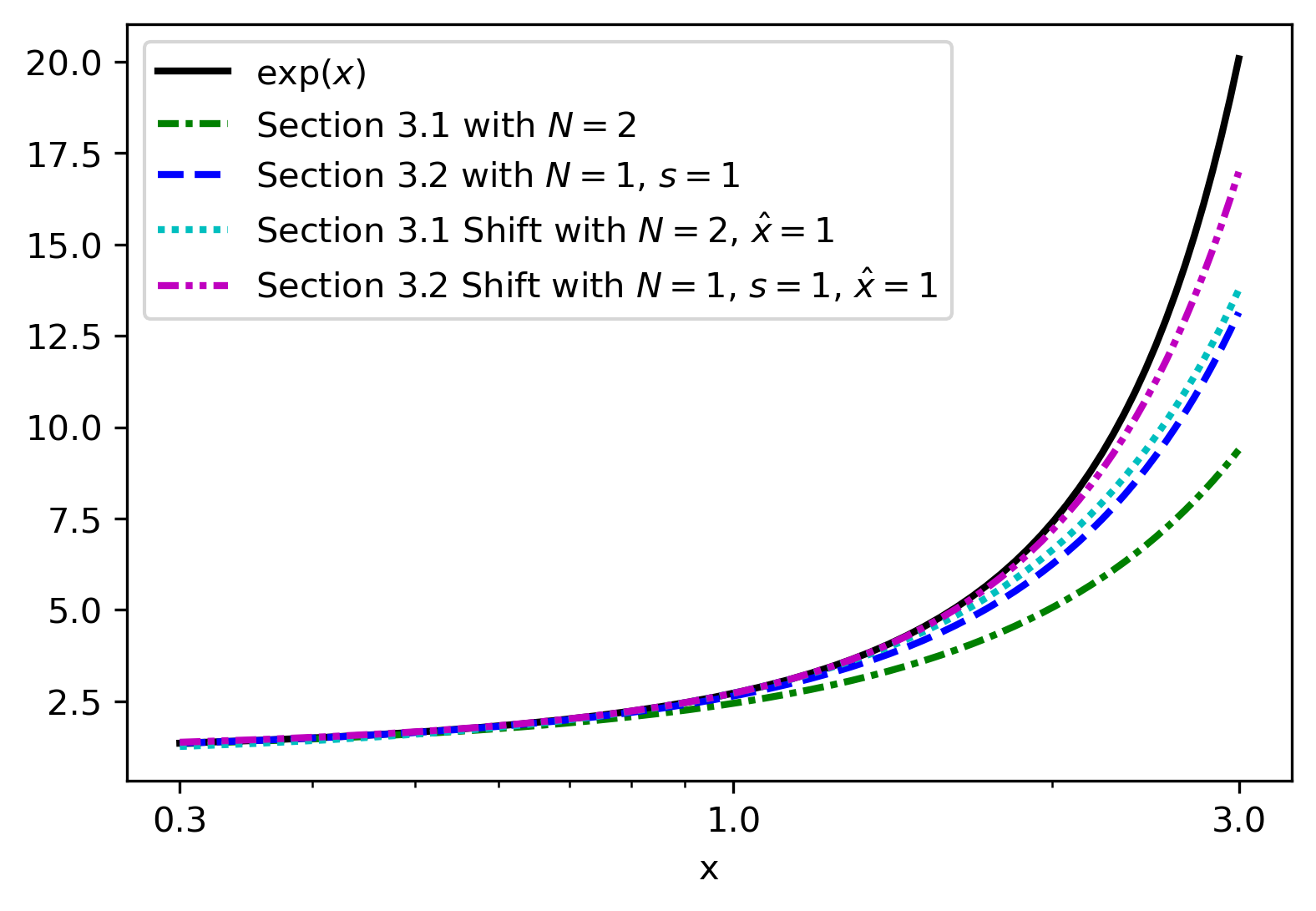}
		\vspace{-20pt}
	\end{center}
	\caption{A Comparison of Methods in Section \ref{sec_exp} using Two SOC Constraints to Approximate the Exponential Function in the Domain $[0.3,3.0]$}\label{fig_sec3_intro}
	\vspace{-50pt}
\end{wrapfigure}

The numerical comparison of Section \ref{sec_exp_lim}, Section \ref{sec_exp_taylor}, Section \ref{sec_exp_lim} Shift, and Section \ref{sec_exp_taylor} Shift methods to approximate the exponential function can be found in Figure \ref{fig_sec3_intro}, where we let $\hat{x}=1$ for the shifting methods. It is seen that the shifting methods can improve both Section \ref{sec_exp_lim} and Section \ref{sec_exp_taylor} methods. In practice, we can run heuristics to obtain a good-quality $\hat{x}$.


\section{Lower Bounds for Polyhedral and SOC Approximations in the Extended Space}\label{sec_bound}

This section focuses on deriving the minimum size (i.e., the minimum number of variables and constraints) of an extended polyhedron or an extended SOC program to approximate the exponential cone within an approximation accuracy.

To begin with, we first study the minimum size of an extended polyhedron $P_{q}$ with $q$ number of variables and  constraints to approximate the hypograph of $\log(x)$ with approximation accuracy $\epsilon$. 

\begin{theorem}\label{prop_poly_lower_bound}
	For any $\epsilon \in (0,1]$, any extended polyhedron $P_{q}$ with $q$ variables (including variables $x,\nu$ defined in \eqref{eq_log_approx}) and constraints such that $\hyp {\log}(-\epsilon)\subseteq \Proj_{(x,\nu)}(P_q )\subseteq \hyp {\log}(\epsilon)$. Then we must have $q=\Omega (\log\log(M)+\log(1/\epsilon))$.
\end{theorem}
\begin{subequations}
	\begin{proof}
		Any extended polyhedron $P_{q}$ with $q$ number of variables and constraints has at most $2^q$ extreme points so is its projection $ \Proj_{(x,\nu)}(P_q )$. Now let us consider the two extreme points of $ \Proj_{(x,\nu)}(P_q )$, denoted by $(\bar{x}_1,\log(\bar{x}_1)+\delta_1),(\bar{x}_2,\log(\bar{x}_2)+\delta_2)$, which are consecutive in $x$-space such that $\bar{x}_1<\bar{x}_2$.  Since $\hyp {\log}(\epsilon)\subseteq \Proj_{(x,\nu)}(P_q )\subseteq \hyp {\log}(-\epsilon)$, we have 
		$$|\delta_1|\leq \epsilon, |\delta_2|\leq \epsilon$$
		Then, the largest distance in the $\nu$-direction between  $ \Proj_{(x,\nu)}(P_q )\cap \{(x,v): x\in [\bar{x}_1,\bar{x}_2]\}$ and set $ \hyp {\log}(\epsilon)\cap \{(x,v): x\in [\bar{x}_1,\bar{x}_2]\}$ can be lower bounded as
		\begin{align}
			d_H=\min_{\delta_1\in [-\epsilon, \epsilon],\delta_2\in [-\epsilon, \epsilon]}\max_{x\in [\bar{x}_1,\bar{x}_2]}\left|\log(x)-\frac{\log(\bar{x}_2)+\delta_2-\log(\bar{x}_1)-\delta_1}{\bar{x}_2-\bar{x}_1}(x-\bar{x}_1)-\log(\bar{x}_1)-\delta_1\right|.\label{eq_hausdoff}
		\end{align}
		In \eqref{eq_hausdoff}, without loss of generality, we can define $x:=x/\bar{x}_1$ and $t=\bar{x}_2/\bar{x}_1>1$. Thus, the optimization problem \eqref{eq_hausdoff} can be simplified as
		\begin{align}
			d_H=\min_{\delta_1\in [-\epsilon, \epsilon],\delta_2\in [-\epsilon, \epsilon]}\max_{x\in [1,t]}\left|\log(x)-\frac{\log(t)+\delta_2-\delta_1}{t-1}(x-1)-\delta_1\right|.\label{eq_hausdoff2}
		\end{align}	
		Note that to achieve the outer minimum in \eqref{eq_hausdoff2}, there are only four valid cases: (a) $\epsilon\geq \delta_1\geq 0\geq \delta_2 \geq -\epsilon$; (b) $\epsilon\geq \delta_2\geq 0\geq \delta_1\geq -\epsilon$; (c) $\epsilon\geq \delta_1,\delta_2\geq 0$; and (d) $0\geq \delta_1, \delta_2\geq -\epsilon$. Please see Figure \ref{LB_four} for an illustration.
		
		\begin{figure}[htbp]
			\centering
			\subfloat[0.48\textwidth][$\epsilon\geq \delta_1\geq 0\geq \delta_2 \geq -\epsilon$]{
				\centering
				\includegraphics[width=0.4\textwidth]{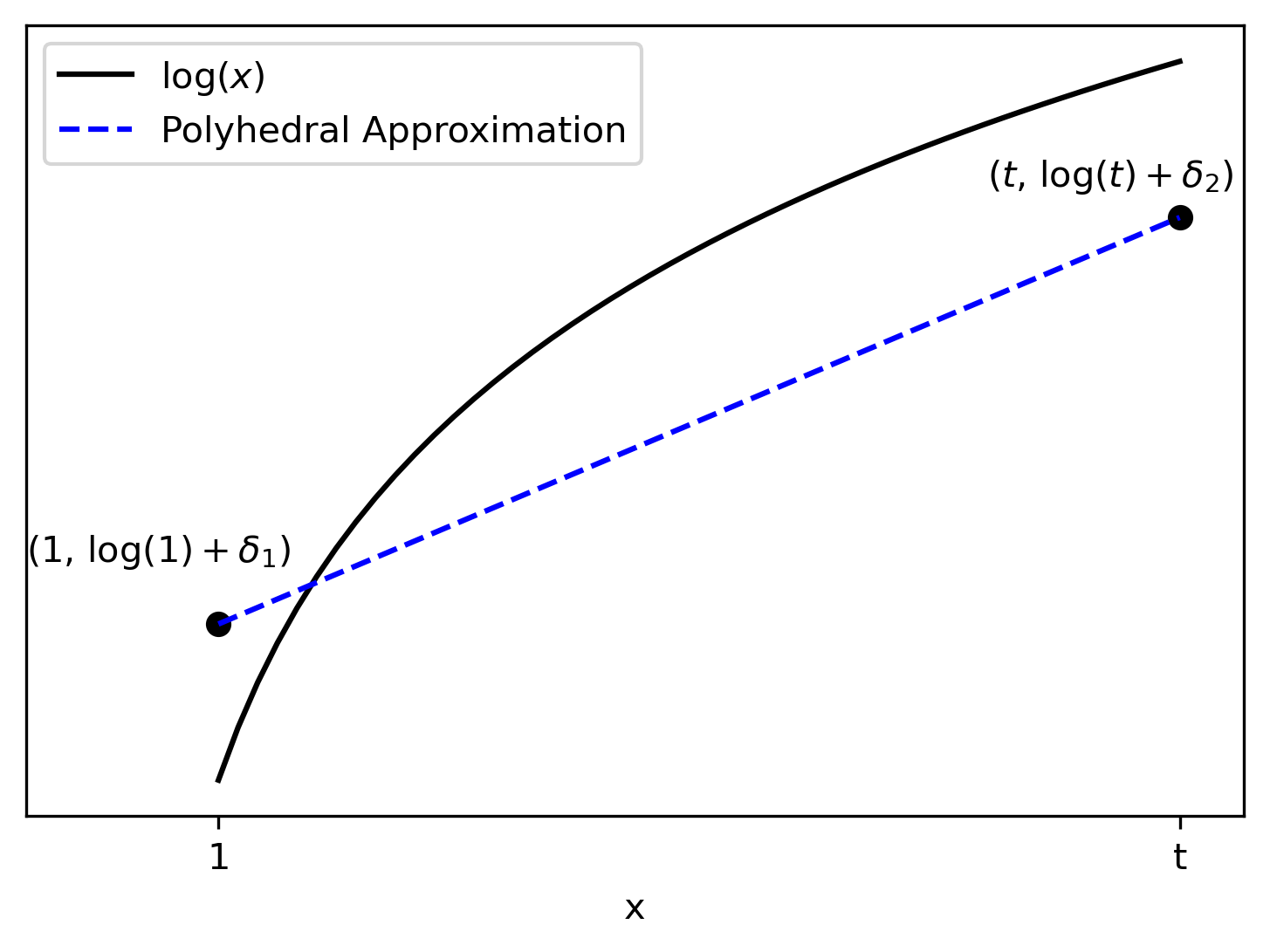} 
				\label{}
			}\hfill
			\subfloat[0.48\textwidth][$\epsilon\geq \delta_2\geq 0\geq \delta_1\geq -\epsilon$]{
				\centering
				\includegraphics[width=0.4\textwidth]{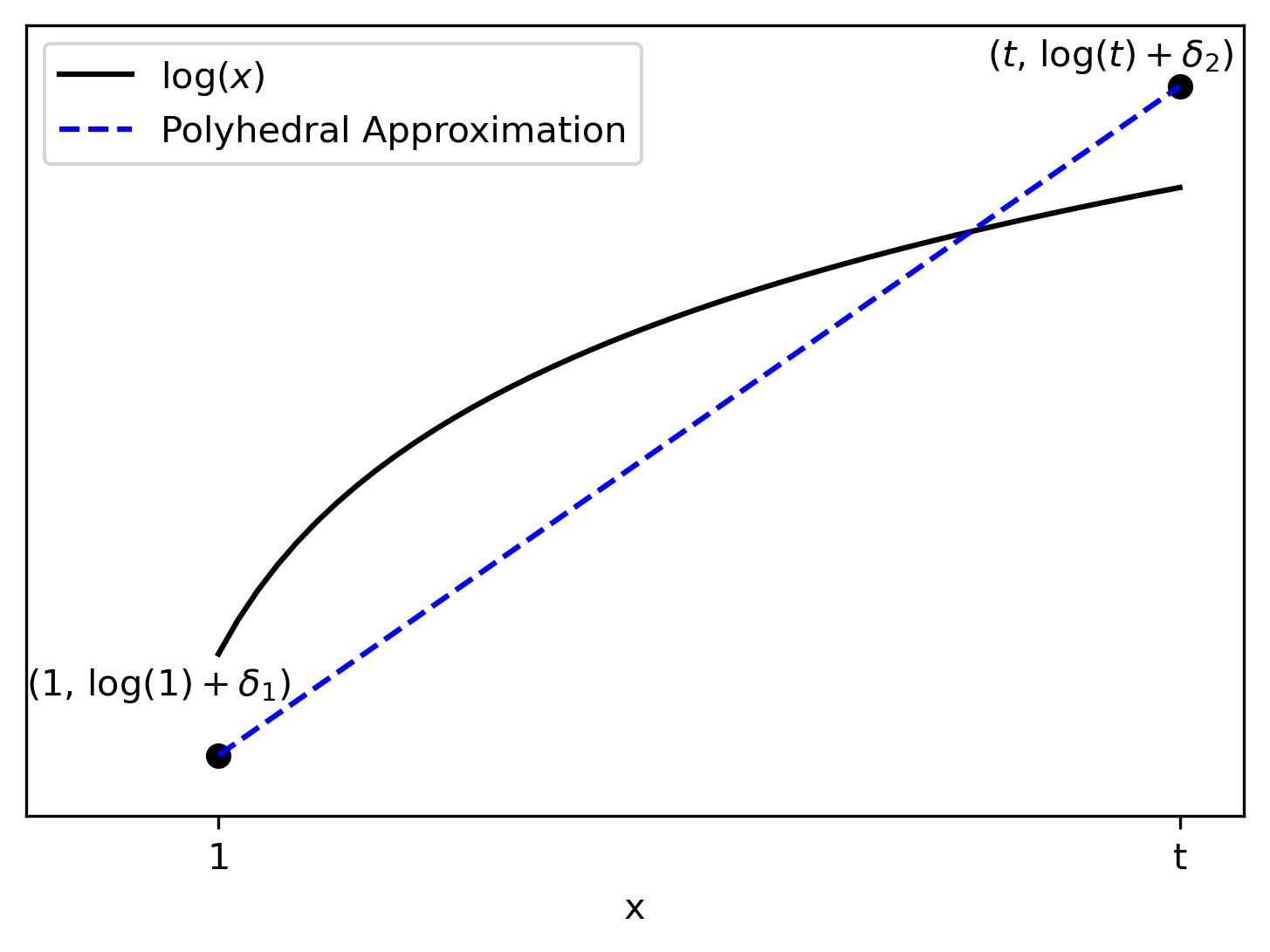} 
				\label{}
			}\hfill
			\subfloat[0.48\textwidth][$\epsilon\geq \delta_1,\delta_2\geq 0$]{
				\centering
				\includegraphics[width=0.4\textwidth]{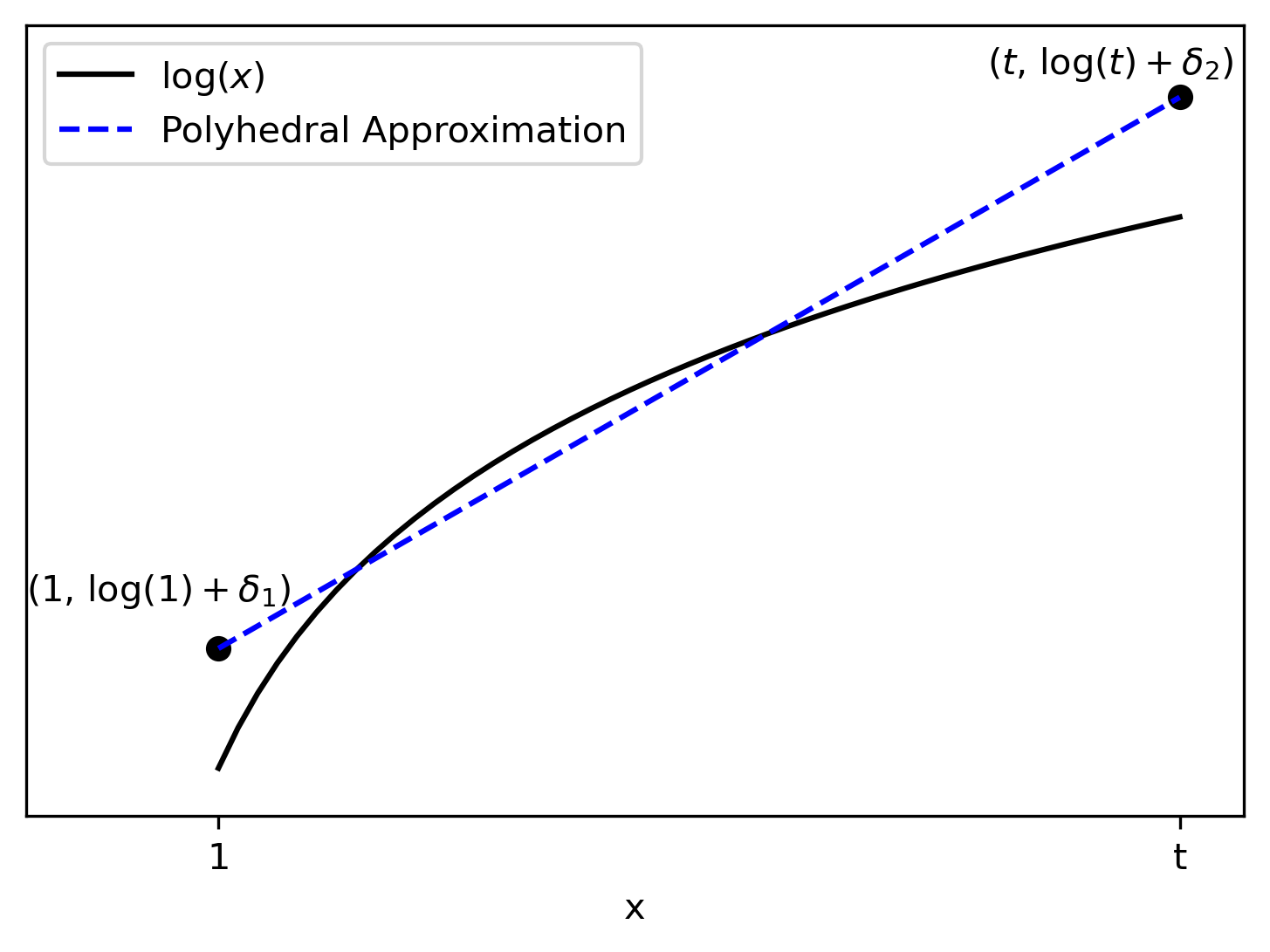} 
				\label{}
			}\hfill
			\subfloat[0.48\textwidth][$0\geq \delta_1, \delta_2\geq -\epsilon$]{
				\centering
				\includegraphics[width=0.4\textwidth]{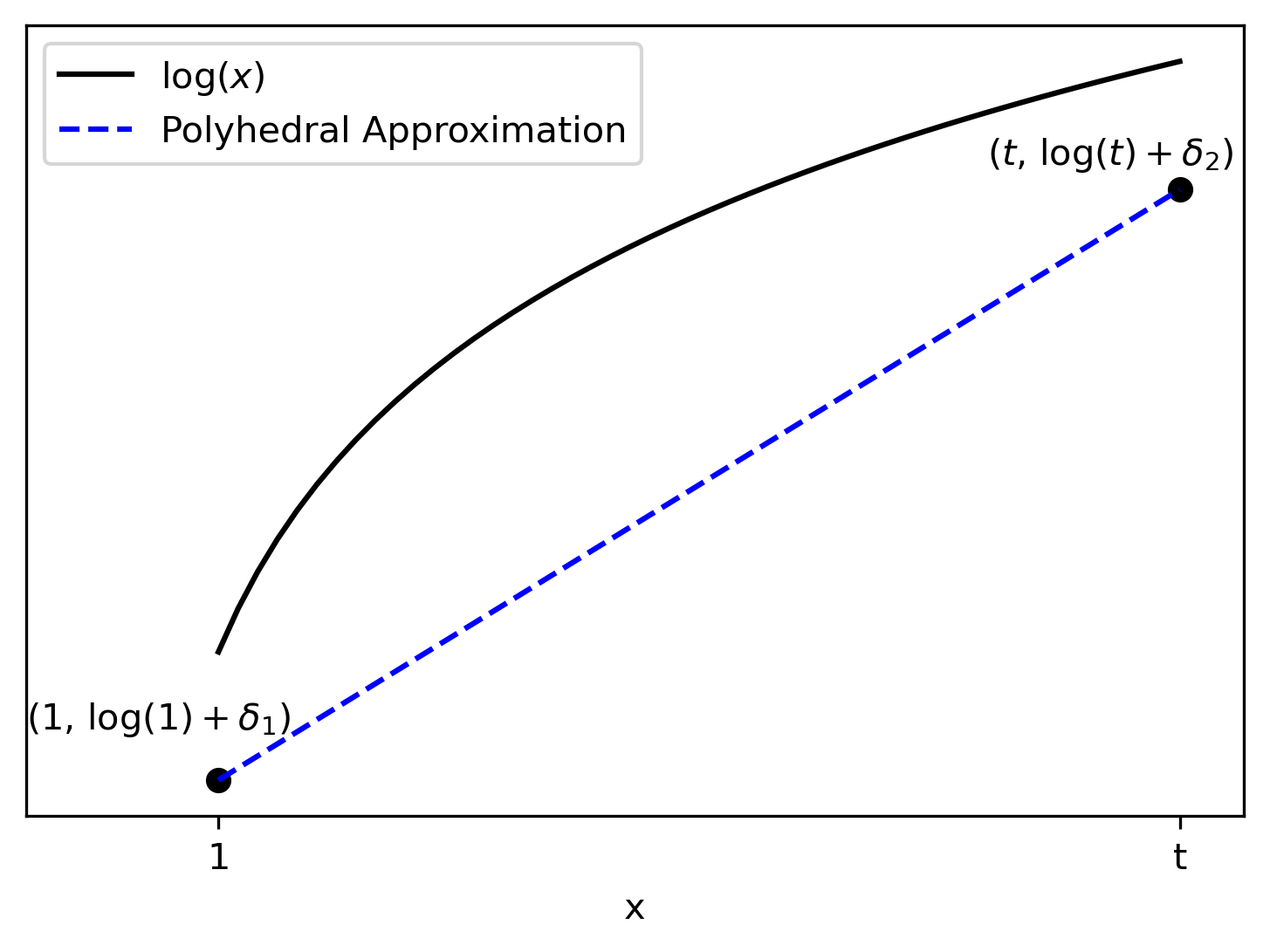} 
				\label{}}
			\caption{Four Cases for the Lower Bound of Polyhedral Approximation}\label{LB_four}
		\end{figure}
		
		\begin{itemize}
			\item[Case (a).] If $\epsilon\geq \delta_1\geq 0\geq \delta_2 \geq -\epsilon$, then \eqref{eq_hausdoff2} can be simplified as
			\begin{align*}
				d_{H1}=\min_{\begin{subarray}{c}
						\epsilon\geq \delta_1\geq 0\geq \delta_2 \geq -\epsilon
				\end{subarray}}\max\left\{\delta_1,-\delta_2,\max_{x\in [1,t]}\log(x)-\frac{\log(t)+\delta_2-\delta_1}{t-1}(x-1)-\delta_1\right\}.
			\end{align*}
			Taking the average of the first and third expressions in the first maximization operator, the value $d_{H1}$ can be lower bounded by
			\begin{align}
				d_{H1}&\geq \min_{\begin{subarray}{c}
						\epsilon\geq \delta_1\geq 0\geq \delta_2 \geq -\epsilon
				\end{subarray}}\frac{1}{2}\max_{x\in [1,t]}\left[\log(x)-\frac{\log(t)+\delta_2-\delta_1}{t-1}(x-1)\right]\notag\\
				&=\frac{1}{2}\max_{x\in [1,t]}\left[\log(x)-\frac{\log(t)}{t-1}(x-1)\right]\notag\\
				&=\frac{1}{2}\left[-\log\left(\frac{\log(t)}{t-1}\right)-1+\frac{\log(t)}{t-1}\right]\label{eq_hausdoff2_h1_1}
			\end{align}
			where the first equality is due to the monotonicity of the objective function with respect to $\delta_1,\delta_2$, and the second one is because of the optimality condition and concavity.
			\item[Case (b).] If $\epsilon\geq \delta_2\geq 0\geq \delta_1 \geq -\epsilon$, then \eqref{eq_hausdoff2} can be simplified as
			\begin{align*}
				d_{H2}=\min_{\begin{subarray}{c}
						\epsilon\geq \delta_2\geq 0\geq \delta_1 \geq -\epsilon
				\end{subarray}}\max\left\{-\delta_1,\delta_2,\max_{x\in [1,t]}\log(x)-\frac{\log(t)+\delta_2-\delta_1}{t-1}(x-1)-\delta_1\right\}.
			\end{align*}
			Taking the average of the second and third expressions in the first maximization operator, the value $d_{H2}$ can be lower bounded by
			\begin{align}
				d_{H2}&\geq \min_{\begin{subarray}{c}
						\epsilon\geq \delta_1\geq 0\geq \delta_2 \geq -\epsilon
				\end{subarray}}\frac{1}{2}\max_{x\in [1,t]}\left[\log(x)-\frac{\log(t)+\delta_2-\delta_1}{t-1}(x-1)-\delta_1+\delta_2\right]\notag\\
				&=\frac{1}{2}\max_{x\in [1,t]}\left[\log(x)-\frac{\log(t)}{t-1}(x-1)\right]\notag\\
				&=\frac{1}{2}\left[-\log\left(\frac{\log(t)}{t-1}\right)-1+\frac{\log(t)}{t-1}\right]\label{eq_hausdoff2_h1_2}
			\end{align}
			where the first equality is due to monotonicity of the objective function with respect to $\delta_1,\delta_2$, and the second one is because of the optimality condition and concavity.
			\item[Case (c).] If $\epsilon\geq \delta_1,\delta_2\geq 0$, then \eqref{eq_hausdoff2} can be simplified as
			\begin{align*}
				d_{H3}=\min_{\begin{subarray}{c}
						\epsilon\geq \delta_1,\delta_2\geq 0
				\end{subarray}}\max\left\{\delta_1,\delta_2,\max_{x\in [1,t]}\log(x)-\frac{\log(t)+\delta_2-\delta_1}{t-1}(x-1)-\delta_1\right\}.
			\end{align*}
			By symmetry, we must have $\delta_1=\delta_2$ at optimality. Thus, we have
			\begin{align*}
				d_{H3}=\min_{\begin{subarray}{c}
						\epsilon\geq \delta_1\geq 0
				\end{subarray}}\max\left\{\delta_1,\max_{x\in [1,t]}\log(x)-\frac{\log(t)}{t-1}(x-1)-\delta_1\right\}.
			\end{align*}
			Taking average of the first and second expressions in the first maximization operator, the value $d_{H3}$ can be lower bounded by
			\begin{align}
				d_{H3}&\geq \frac{1}{2}\max_{x\in [1,t]}\left[\log(x)-\frac{\log(t)}{t-1}(x-1)\right]=\frac{1}{2}\left[-\log\left(\frac{\log(t)}{t-1}\right)-1+\frac{\log(t)}{t-1}\right]\label{eq_hausdoff2_h1_3}
			\end{align}
			where the first equality  is because of the optimality condition and concavity.
			
			\item[Case (d).] If $0\geq \delta_1,\delta_2\geq -\epsilon$, then \eqref{eq_hausdoff2} can be simplified as
			\begin{align*}
				d_{H4}=\min_{\begin{subarray}{c}
						0\geq \delta_1,\delta_2\geq -\epsilon
				\end{subarray}}\max\left\{-\delta_1,-\delta_2,\max_{x\in [1,t]}\log(x)-\frac{\log(t)+\delta_2-\delta_1}{t-1}(x-1)-\delta_1\right\}.
			\end{align*}
			Since the first maximization is monotone non-increasing with respect to $\delta_1,\delta_2$, thus the outer minimization is achieved by $\delta_1=0,\delta_2=0$. Then we have
			\begin{align}
				d_{H4}=\max_{x\in [1,t]}\log(x)-\frac{\log(t)}{t-1}(x-1)=-\log\left(\frac{\log(t)}{t-1}\right)-1+\frac{\log(t)}{t-1}\label{eq_hausdoff2_h1_4}
			\end{align}
			where the second equality is because of the optimality condition and concavity.
		\end{itemize}
		
		Thus, we have 
		\begin{align*}
			d_H=\min_{i\in [4]}\{d_{Hi}\}\geq \frac{1}{2}\left[-\log\left(\frac{\log(t)}{t-1}\right)-1+\frac{\log(t)}{t-1}\right]\geq  \frac{1}{10^4}(t-1)^2		
		\end{align*}
		for any $t>1$ and $d_H\leq \epsilon \leq 1$.
		Hence, to achieve $\hyp {\log}(\epsilon)\subseteq \Proj_{(x,\nu)}(P_q )\subseteq \hyp {\log}(-\epsilon)$, we must have $10^{-4}(t-1)^2\leq \epsilon$, i.e., $t=\bar{x}_2/\bar{x}_1\leq 1+\sqrt{10^4\epsilon}$. 
		
		Since $(\bar{x}_1,\log(\bar{x}_1)+\delta_1),(\bar{x}_2,\log(\bar{x}_2)+\delta_2)$  are two arbitrarily consecutive extreme points, for all the extreme points, we must have
		\[(1+\sqrt{10^4\epsilon})^{2^q}\geq M/(1/M):=M^2\]
		i.e., $q=\Omega (\log\log(M)+\log(1/\epsilon))$.\QEDA
	\end{proof}
\end{subequations}


We notice that the exponential cone can be viewed as the hypograph of the perspective of the logarithm function. Thus, the lower bound result in Theorem \ref{prop_poly_lower_bound} can be directly applied to that of polyhedral approximation of the exponential cone.
\begin{corollary}
	For any $\epsilon\in (0,1]$, given a polyhedral approximation $P_q$ with $q$ additional variables and linear constraints such that $K_{\exp}(-\epsilon)\subseteq P_q \subseteq K_{\exp}(\epsilon)$, we must have $q=\Omega (\log\log(M)+\log(1/\epsilon))$.
\end{corollary}
In \cite{ben2001polyhedral}, the authors showed a lower bound of any polyhedral approximation of an SOC program is $\Omega(\log(1/\epsilon))$. Together with the lower bound result in Theorem \ref{prop_poly_lower_bound}, we have that to obtain an $\epsilon$- SOC approximation of the exponential cone, it requires $\Omega(1+\log\log(M)/\log(1/\epsilon))$ number of variables and SOC constraints. This lower bound result is summarized below.
\begin{corollary}
	For any $\epsilon\in (0,1]$, given an SOC approximation $C_q$ with $q$ number of variables and SOC constraints such that $K_{\exp}(-\epsilon)\subseteq C_q \subseteq K_{\exp}(\epsilon)$, we have \allowbreak$q=\Omega (1+\log\log(M)/\log(1/\epsilon))$.
\end{corollary}
This corollary shows that the minimum number of SOC constraints needed to approximate the exponential cone may be constant, given that $\log(M)$ is constant. Meanwhile, according to Section \ref{sec_approx_ec_exmps}, if there exists a tight approximate solution, then we can also obtain a constant number of variables and SOC constraints using Example \ref{emp_3} to achieve $\epsilon-$approximation accuracy; otherwise, according to the discussion of Example \ref{emp_2} in Section \ref{sec_approx_ec_exmps} or Theorem \ref{prop_approx_exp_cone2_s} with $N=s$, the best we can achieve is $O(\sqrt{\log(1/\epsilon)})$ number of additional variables and second-order constraints. However, we are unable to further improve the lower or upper bounds and leave them to interested readers.

\section{Approximating the Exponential Cone in the Original Space Using Gradient Inequalities }\label{sec_outer}

This section will study the polyhedral outer approximation of the exponential cone in the original space based on gradient inequalities. We study the upper and lower bounds for the approximation errors. 

We first study the outer approximation of the hypograph of $\log(x)$, i.e., $\hyp\log(0)$. Given $N$ points $\{\hat{x}_i\}_{i\in [N]}$ (including two boundary points of $\R$ for the sake of simplicity) in the approximation, the outer approximation, denoted by $\hat{H}_{N}^o$, admits the following form
\begin{align}\label{eq_outer}
	\hat{H}_{N}^o=\left\{(x,\nu)\in \Re\times \Re: \log(\hat{x}_i)+\frac{1}{\hat{x}_i}(x-\hat{x}_i)\geq \nu\right\}.
\end{align}
Since $\log(x)$ is a concave function, we have $\hat{H}_{N}^o\subseteq \hyp\log(0)$. We observe that the approximation error of polyhedral outer approximation depends on the maximum distance between the extreme points and logarithm function. Based on this observation, we can find a tight approximation bound in the $\nu$ direction, and the result is summarized below.  
\begin{theorem}\label{prop_outer_bound}
	For any $\epsilon \in [0,1]$, there exists a polyhedral outer approximation in the original space with $N=\Theta(\log(M)+1/\sqrt{\epsilon})$ points such that $\hyp {\log}(0)\subseteq \hat{H}_{N}^o \subseteq \hyp {\log}(\epsilon)$. 
\end{theorem}
\begin{proof}
	We observe that except the boundary points, there are $N-1$ extreme points in set $\hat{H}_{N}^o$. Now let us consider its two neighboring points $(\bar{x}_1,\log(\bar{x}_1)),(\bar{x}_2,\log(\bar{x}_2))$ which are consecutive in $x$-space and satisfy $\bar{x}_1<\bar{x}_2$. Then, the largest distance on the $\nu$-direction between $ \hat{H}_{N}^o \cap \{(x,\nu): x\in [\bar{x}_1,\bar{x}_2]\}$ and set $ \hyp {\log}(0)\cap \{(x,\nu): x\in [\bar{x}_1,\bar{x}_2]\}$ is
	\begin{align}
		d=\frac{1}{\bar{x}_1}(x^*-\bar{x}_1)+\log(\bar{x}_1)-\log(x^*),\label{eq_distance}
	\end{align}
	where $x^*=\log(\bar{x}_2/\bar{x}_1)/(1/\bar{x}_1-1/\bar{x}_2)$ is the $x$-coordinate of the intersection point of two lines $y=(x-\bar{x}_1)/\bar{x}_1+\log(\bar{x}_1)$ and $(x-\bar{x}_2)/\bar{x}_2+\log(\bar{x}_2).$ Please see Figure \ref{UB_outer} for an illustration.
	
	\begin{wrapfigure}{r}{0.50\textwidth}
		\vspace{-10pt}
		\begin{center}
			\includegraphics[width=0.48\textwidth]{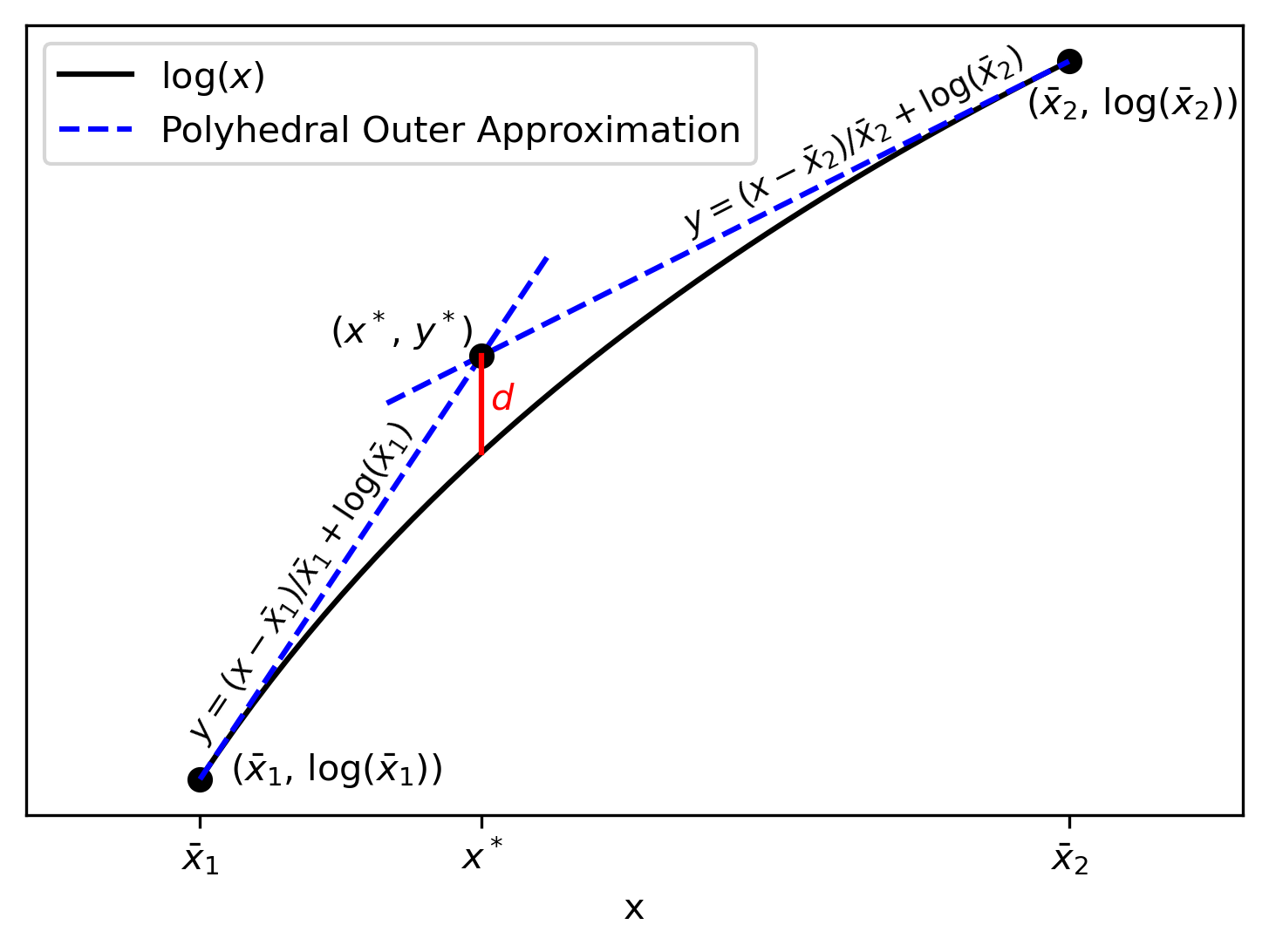} 
			\vspace{-20pt}
		\end{center}
		\caption{Illustration of the Polyhedral Outer Approximation Error}\label{UB_outer}
		\vspace{-10pt}
	\end{wrapfigure}
	
	To simplify \eqref{eq_distance}, we define $t=\bar{x}_2/\bar{x}_1>1$. Thus, the distance formula \eqref{eq_distance} can be simplified as
	$$d=-\log\left(\frac{\log(t)}{t-1}\right)-1+\frac{\log(t)}{t-1}\leq \frac{1}{8}(t-1)^2,$$
	and the inequality holds for any $t>1$.
	Hence, to achieve $\hyp {\log}(-\epsilon)\subseteq \hat{H}_{N} \subseteq \hyp {\log}(0)$, it suffices to let $t=\bar{x}_2/\bar{x}_1\leq 1+\sqrt{8\epsilon}$.
	Since $(\bar{x}_1,\log(\bar{x}_1)),(\bar{x}_2,\log(\bar{x}_2))$ are two arbitrarily consecutive points, let us choose set $\{\bar{x}_i\}_{i\in [N]}$ as $\hat{x}_{i+1}/\hat{x}_{i}=1+\sqrt{8\epsilon}$ for all $i\in[N-1]$. 
	Then we must have
	\[(1+\sqrt{8\epsilon})^{N-1}\geq M/(1/M):=M^2\]
	i.e., $N=O(\log(M)+1/\sqrt{\epsilon})$.
	Following the proof in Theorem \ref{prop_poly_lower_bound} by replacing $2^q$ by $N$, we can obtain that the smallest $N=\Omega(\log(M)+1/\sqrt{\epsilon})$. Thus, we must have $N=\Theta(\log(M)+1/\sqrt{\epsilon})$. \QEDA
\end{proof}
Given that $\log(M)$ is constant, Theorem \ref{prop_outer_bound} shows the best outer approximations in the original space that can achieve $\epsilon$-approximation accuracy require $N=O(1/\sqrt{\epsilon})$. Different from the results in the other sections, $N=O(1/\sqrt{\epsilon})$ can be significantly larger than $N=O(\log(1/\epsilon))$ or $N=O(\sqrt{\log(1/\epsilon)})$. However, in our numerical study, we find that the outer approximations work very well especially having mixed-integer variables, which may be because the solvers are better at solving MILPs compared to MISOCPs.  

Since the exponential cone can be viewed as the hypograph of perspective of the logarithm function, the approximation result also holds for the exponential cone. The upper and lower bounds results in Theorem \ref{prop_outer_bound} can be directly applied to the polyhedral outer approximation of the exponential cone.
\begin{corollary}
	For any $\epsilon \in [0,1]$, there exists a polyhedral outer approximation in the original space with $\Theta(\log(M)+1/\sqrt{\epsilon})$ points such that $K_{\exp}(0)\subseteq \hat{K}_{N}^o \subseteq K_{\exp}(\epsilon)$, where $\hat{K}_{N}^o:=\{\bm{x}\in \R^{2}\times \Re: (x_1/x_2,x_3/x_2)\in \hat{H}_{N}^o \}$.
\end{corollary}

In practice, different from SOC approximations, the proposed outer approximation can be done more efficiently in an iterative way. 
That is, given a solution $\hat{\bm{x}}\in \Re^2\times\Re\setminus K_{\exp}(0)$, i.e., $(\hat{x}_1/\hat{x}_2,\hat{x}_3/\hat{x}_2) \notin \hyp\log(0)$.
Then we can add a gradient inequality to cut it off, i.e.,
\[ \log(\hat{x}_1/\hat{x}_2)+\hat{x}_2/\hat{x}_1(x_1/x_2-\hat{x}_1/\hat{x}_2)\geq x_3/x_2,\]
which is equivalent to
\[ x_2\log(\hat{x}_1/\hat{x}_2)+x_1\hat{x}_2/\hat{x}_1-x_2\geq x_3.\]
We proceed until reaching a stopping criterion (e.g., the relative gap is within $10^{-4}$).

We plan to employ the naive cutting plane method and the branch and cut algorithm to solve the MIECP, where the latter generates a gradient inequality via a delayed cut generating procedure. 
Our numerical study shows that the proposed cutting plane and branch and cut methods based on gradient inequalities can effectively solve MIECPs. 

\section{Numerical Study}\label{sec_numerical}
In this section, we provide numerical illustrations for both SOC approximations and polyhedral approximations by solving three types of MIECPs, i.e., a packing MIECP, a covering MIECP, and a sparse logistic regression problem.
The proposed methods are implemented in Gurobi, and we compare the results to the state-of-art solver MOSEK.
All the instances in this section are executed in Python 3.7 with calls to Gurobi 9.0 and MOSEK 9.2 on a personal computer with 2.3 GHz Intel Core i9 processor and 16G of memory. Codes of the numerical experiments are available at \url{https://github.com/qingye1/Exp_Cone_Approximation}.

\noindent{\textbf{Nomenclature for the Numerical Study:}} The optimality gap, denoted as ``Gap", is computed based on the true optimal solution of an MIECP, which is computed by the absolute difference between the best objective value of an approximation method and the true optimal value of an MIECP divided by the true optimal value. We use dash line ``---" to mark the cases whose Gaps are no larger than $10^{-4}$ and use asterisk ``*" to mark cases that exceed the time limit. We also compute the ratio of the running time of a proposed method to that of MOSEK, denoted as ``Ratio," and use the geometric mean of ratios to evaluate computational efficiency, denoted as ``Geo Mean." We use ``Section \ref{sec_exp_lim}" and ``Section \ref{sec_exp_taylor}" to denote the methods based on the SOC approximation results in Section \ref{sec_exp_lim} and Section \ref{sec_exp_taylor}, and use ``Section \ref{sec_exp_lim} Shift" and ``Section \ref{sec_exp_taylor} Shift" to denote their corresponding shifting methods studied in Section \ref{sec_exp_shift}. We use ``Cutting Plane" to denote naive implementation of cutting plane method without any delayed implementations and ``Branch and Cut" to denote the branch and cut algorithm with delayed constraint generations.
When employing Example \ref{emp_3}, we select the approximate solution from the list $\{2^{-6}, 2^{-5},\ldots, 2^{0}\}$. In ``Example \ref{emp_3} with Best Scale", we construct the best approximate solution for each exponential conic constraint using a feasible solution found by executing Example \ref{emp_3} with 3-tuple $(N,AS, TL)$, where $N\in [4]$, and time limit $TL\in \{2,10,30\}$ 
(to ensure a feasible solution to be found), and an approximate solution $AS$ is chosen from the list $\{2^{-6}, 2^{-5},\ldots, 2^{0}\}$. This procedure is called ``Best Scale." We follow the same procedure to find the best approximate solution for the shifting methods in Section \ref{sec_exp_lim} and Section \ref{sec_exp_taylor}. Note that we report the total time for constructing the best approximate solution using ``Best Scale" and solving an MIECP for methods employing the ``Best Scale" procedure.
We use ``n/a" to denote that an instance cannot be solved by one method. 
We call a method running into ``numerical issues'' if we cannot close the gap by increasing the number of SOC or polyhedral constraints.

	\subsection{Packing MIECP}
Let us consider the following mixed-integer convex packing problem
\begin{align}\label{eq_num_study}
	\min_{\bm{x}\in X}\left\{ \sum_{\ell\in [p]}\exp\left(\sum_{j\in [n]}c_{\ell j}x_j\right):\sum_{j\in [n]}a_{ij}x_j \leq  b_i, \forall i \in [m]\right\},
\end{align}
where $X=\{0,1\}^t\times [0,1]^{n-t}$. Note that the problem \eqref{eq_num_study} can be formulated as the following packing MIECP
\begin{align}\label{eq_num_study_exp_cone}
	\min_{\bm{x}\in X}\left\{\sum_{\ell\in [p]}v_\ell: \left(v_\ell,1, \sum_{j\in [n]}c_{\ell j}x_j\right)\in K_{\exp}(0), \forall \ell\in [p],\sum_{j\in [n]}a_{ij}x_j \leq  b_i, \forall i \in [m]\right\}.
\end{align}

In this subsection, we test the proposed methods by solving the packing MIECP \eqref{eq_num_study_exp_cone}. We conduct two experiments to test the proposed SOC and polyhedral approximations for solving \eqref{eq_num_study_exp_cone}, where we consider pure binary ($t=n$) and mixed-integer ($t\in [n-1]$) experiments and compare our results with MOSEK. In the testing instances, we set $m=100$ and suppose $a_{ij}\sim \mathrm{int}(0, 9)$, $b_i=4n$, $c_{\ell j}\sim -\mathrm{int}(0, 9)/n$ for all $i \in [m], j\in [n], \ell\in [p]$, where $\mathrm{int}(p, q)$ denotes a random integer between $p$ and $q$ including $p$ and $q$. 
The time limit is set to be $3600$ seconds for both experiments. Due to different experimental settings (e.g., the accuracy requirements), the pure continuous ($t=0$) experiment is reported in Appendix \ref{sec_continuous_packing}. 

\noindent\textbf{Experiment 1- Binary Packing MIECP:} 
For the binary experiment, we solve the small-scale instances with $n=20$, $p\in\{5, 7, \ldots, 25\}$ to do sensitive analyses of parameters and learn the patterns 
and then solve the large-scale problem with $n=100$, $p\in\{5, 10, \ldots, 50\}$. Due to the page limit, we do not report these results. The results for large-scale binary packing instances are shown in Table \ref{pack_binary_table_sec2}-Table \ref{pack_binary_table_sec5}, where only the cases with the Gap greater than $10^{-4}$ or running time within the time limit are displayed. Overall, we see that increasing the number of SOC constraints (i.e., $N$) can improve the Gap effectively but also requires a longer running time. In this experiment, two of our best methods (i.e., Example \ref{emp_3} Best Scale and Section \ref{sec_exp_taylor} Shift) on average are $85\%$ shorter than MOSEK, with Gap being less than $10^{-4}$.

The results for the proposed SOC approximation methods are displayed in Tables \ref{pack_binary_table_sec2} and \ref{pack_binary_table_sec3}.
For the binary packing MIECP, Example \ref{emp_1}, Section \ref{sec_exp_lim}, Section \ref{sec_exp_taylor}, and Section \ref{sec_exp_lim} Shift methods cannot solve all the cases with the Gap being no larger than $10^{-4}$ within the time limit, while all the other methods are consistently better than MOSEK. Particularly, it is worthy of mentioning that Example \ref{emp_1} with $N=20$ has numerical issues, and we are not able to improve the solution quality by increasing $N$. On average, the running time of Example \ref{emp_2} and Example \ref{emp_3} can be $60\%$, $69\%$  shorter than that of MOSEK, respectively. 
Both Example \ref{emp_3} with Best Scale and Section \ref{sec_exp_taylor} Shift method apparently are the best among all the methods, since their average running time is only $13\%$ of that of MOSEK. Particularly, Example \ref{emp_3} with Best Scale only needs one SOC constraint, and Section \ref{sec_exp_taylor} Shift method needs two SOC constraints to approximate  each exponential conic constraint. It is worthy of emphasizing that we find the approximate solution for Example \ref{emp_3} with Best Scale and Section \ref{sec_exp_taylor} Shift by executing Example \ref{emp_3} with the approximate solution equal to $2^{-4}$ and 2-second time limit. We also see that Section \ref{sec_exp_lim} and Section \ref{sec_exp_taylor} methods run into numerical issues, while the shifting method improves them significantly. 

\begin{table}[htbp]
	\centering
	\caption{Numerical Results of Section~\ref{sec_log} Methods for Solving the Binary Packing MIECP}
	\label{pack_binary_table_sec2}
	\begin{threeparttable}
		\scriptsize\setlength{\tabcolsep}{1.0pt}
		\begin{tabular}{c|rr|rrr|rrr|rrr|rrr}
			\hline
			$n=100$ & \multicolumn{2}{c|}{MOSEK} & \multicolumn{3}{c|}{Example \ref{emp_1}\tnote{i}} & \multicolumn{3}{c|}{Example \ref{emp_2}\tnote{ii}} & \multicolumn{3}{c|}{Example \ref{emp_3}\tnote{iii}} & \multicolumn{3}{c}{Example \ref{emp_3} Best Scale\tnote{iv}}  \\ \hline
			$p$ & \multicolumn{1}{c}{Gap} & \multicolumn{1}{c|}{Time (s)} & \multicolumn{1}{c}{Gap} & \multicolumn{1}{c}{Time (s)} & \multicolumn{1}{c|}{Ratio} & \multicolumn{1}{c}{Gap} & \multicolumn{1}{c}{Time (s)} & \multicolumn{1}{c|}{Ratio} & \multicolumn{1}{c}{Gap} & \multicolumn{1}{c}{Time (s)} & \multicolumn{1}{c|}{Ratio} & \multicolumn{1}{c}{Gap} & \multicolumn{1}{c}{Time (s)} & \multicolumn{1}{c}{Ratio} \\ \hline
			5   & --- & 2.63    & ---  & 5.34    & 2.02 &  ---  & 0.78  & 0.30 & ---  & 2.15   & 0.81 & ---  & 1.71   & 0.65 \\
			10  & ---  & 3.63    & ---  & 4.85    & 1.34 &  ---  & 1.64  & 0.45 & ---  & 1.69   & 0.47 & ---  & 1.65   & 0.45 \\
			15  & --- & 102.83  & --- & 56.03   & 0.54 	&  ---  & 22.89 & 0.22 & --- & 15.91  & 0.15 & --- & 6.54   & 0.06 \\
			20  & --- & 35.12   & 1.1e-4 & 78.66   & 2.24 &  ---  & 17.23 & 0.49 & --- & 6.90   & 0.20 & --- & 5.59   & 0.16 \\
			25  & --- & 111.18  & --- & 512.36  & 4.61 	&  ---  & 52.48 & 0.47 & --- & 70.73  & 0.64 & ---& 22.21  & 0.20 \\
			30  & --- & 742.96  & --- & 1989.90 & 2.68 &  ---  & 241.59 & 0.33	& --- & 188.23 & 0.25 & --- & 59.96  & 0.08 \\
			35  & --- & 113.84  & --- & 113.64  & 1.00 	&  ---  & 23.44 & 0.21 & --- & 15.87  & 0.14 & --- & 8.04   & 0.07 \\
			40  & --- & 240.44  & --- & 535.56  & 2.23  &  ---  & 154.69 & 0.64 & --- & 90.45  & 0.38 & ---& 23.77  & 0.10 \\
			45  & --- & 306.41  & --- & 436.13  & 1.42 &  ---  & 143.24 & 0.47	& --- & 68.50  & 0.22 & --- & 16.73  & 0.05 \\
			50  & --- & 2524.17 & 8.3e-3 & * & * &  ---  & 1802.24 & 0.71 & --- & 940.52 & 0.37 & --- & 198.66 & 0.08 \\ \hline
			Geo Mean &         &         &         &         & 1.72 &         &         & 0.40 	&         &        & 0.31 &         &        & 0.13  \\ \hline
		\end{tabular}
		\begin{tablenotes}
			\item[i] $N=20, a=1$;
			\item[ii] $N=6, s=4$;
			\item[iii] $N=4$ and the approximation solution is $2^{-4}$;
			\item[iv] $N=1$ and in the Best Scale procedure, we run Example \ref{emp_3} with $(N=1, AS=2^{-4}, TL=2)$.
		\end{tablenotes}
	\end{threeparttable}
\end{table}
\begin{table}[htbp]
	\centering
	\caption{Numerical Results of Section \ref{sec_exp} Methods for Solving the Binary Packing MIECP}
	\label{pack_binary_table_sec3}
	\begin{threeparttable}
		\scriptsize\setlength{\tabcolsep}{1.0pt}
		\begin{tabular}{c|rr|rrr|rrr|rrr|rrr}
			\hline
			$n=100$ & \multicolumn{2}{c|}{MOSEK} & \multicolumn{3}{c|}{Section \ref{sec_exp_lim}\tnote{i}} & \multicolumn{3}{c|}{Section \ref{sec_exp_taylor}\tnote{ii}} & \multicolumn{3}{c|}{Section \ref{sec_exp_lim} Shift\tnote{iii}} & \multicolumn{3}{c}{Section \ref{sec_exp_taylor} Shift\tnote{iv}} \\ \hline
			$p$ & \multicolumn{1}{c}{Gap} & \multicolumn{1}{c|}{Time (s)} & \multicolumn{1}{c}{Gap} & \multicolumn{1}{c}{Time (s)} & \multicolumn{1}{c|}{Ratio} & \multicolumn{1}{c}{Gap} & \multicolumn{1}{c}{Time (s)} & \multicolumn{1}{c|}{Ratio} & \multicolumn{1}{c}{Gap} & \multicolumn{1}{c}{Time (s)} & \multicolumn{1}{c|}{Ratio} & \multicolumn{1}{c}{Gap} & \multicolumn{1}{c}{Time (s)} & \multicolumn{1}{c}{Ratio} \\ \hline
			5   &  --- & 2.63     & 1.4e-3 & 2.80    &1.06  & 2.2e-3 & 2.82    & 1.07 &  ---  & 1.29&	0.49 &  ---  & 2.38   & 0.91  \\
			10  &  ---  & 3.63    & 1.5e-3 & 4.06    & 1.12 &  1.6e-3 & 1.71    & 0.47 &  ---  & 1.05&	0.29 &  ---  & 0.81   & 0.22  \\
			15  &  ---  & 102.83  & 6.9e-3 & 5.26   & 0.05 & 1.6e-3 & 10.06   & 0.10 &  ---  &19.50&	0.19 &  ---  & 8.84   & 0.09  \\
			20  &  ---  & 35.12   & 1.5e-3 & 5.79    & 0.16 &  2.0e-3 & 9.36    & 0.27 &  ---  & 11.32&	0.32 &  ---  & 5.20   & 0.15  \\
			25  &  ---  & 111.18   & 3.8e-3 & 22.94  & 0.21  &  1.7e-3 & 39.80   & 0.36 &  ---  & 769.74&	6.92 &  ---  & 17.37  & 0.16 \\
			30  &  ---  & 742.96   & 5.2e-3 & 64.60 & 0.09 &  1.9e-3 & 125.19  & 0.17 &  6.5e-3 & *&	* &  ---  & 82.70  & 0.11  \\
			35  &  ---  & 113.84  & 1.5e-3 & 22.60  &  0.20&  2.2e-3 & 9.88    & 0.09 &  --- & 440.48&	3.87 &  ---  & 9.60   & 0.08  \\
			40  &  ---  & 240.44  & 2.8e-3 & 85.51 & 0.36 & 2.2e-3 & 961.87  & 4.00 & --- & 1423.26&	5.92 &  ---  & 17.21  & 0.07  \\
			45  &  ---  & 306.41   & 1.0e-2 & 13.43 &  0.04& 2.2e-3 & 2213.90 & 7.23 & 1.1e-4 & 1937.42&	6.32 &  ---  & 20.51  & 0.07 \\
			50  &  --- & 2524.17 & 1.2e-2 & 258.69 &  0.10&  1.8e-2 & * & * & 1.7e-2 & *&	* &  ---  & 271.84 & 0.11  \\ \hline
			Geo Mean &         &        &       &         & 0.19 &        &         & 0.50 &         &         & 1.31 &         &        & 0.13 \\ \hline
		\end{tabular}
		\begin{tablenotes}
			\item[i] $N=15$;
			\item[ii] $N=6,s=1$;
			\item[iii] $N=7$ and using the same approximate solution as Example \ref{emp_3} Best Scale;
			\item[iv] $N=0,s=1$ and using the same approximate solution as Example \ref{emp_3} Best Scale.
		\end{tablenotes}
	\end{threeparttable}
\end{table}

The results for the proposed polyhedral approximation methods are shown in Table \ref{pack_binary_table_sec5}. The two polyhedral approximation methods, Cutting Plane and Branch and Cut, perform quite well. Specifically, the average running time of the Cutting Plane method is around $21\%$ of that of MOSEK, and using a delayed cut generation procedure (i.e., Branch and Cut method) has an average running time of around $18\%$ of that of MOSEK. 

\begin{table}[htbp]
	\centering
	\caption{Numerical Results of Section \ref{sec_outer} Methods for Solving the Binary Packing MIECP}
	\label{pack_binary_table_sec5}
	\begin{threeparttable}
		\scriptsize\setlength{\tabcolsep}{1.0pt}
		\begin{tabular}{c|rr|rrr|rrr}
			\hline
			$n=100$ & \multicolumn{2}{c|}{MOSEK} & \multicolumn{3}{c|}{Cutting Plane} & \multicolumn{3}{c}{Branch and Cut} \\ \hline
			$p$ & \multicolumn{1}{c}{Gap} & \multicolumn{1}{c|}{Time (s)}  & \multicolumn{1}{c}{Gap} & \multicolumn{1}{c}{Time (s)} & \multicolumn{1}{c|}{Ratio} & \multicolumn{1}{c}{Gap} & \multicolumn{1}{c}{Time (s)} & \multicolumn{1}{c}{Ratio}  \\ \hline
			5   &  --- & 2.63     & --- & 1.36   & 0.52 & ---  & 0.96   & 0.37 \\
			10  &  ---  & 3.63     & --- & 1.23   & 0.34 & --- & 1.93   & 0.53 \\
			15  &  ---  & 102.83  & --- & 19.31  & 0.19 & --- & 12.41  & 0.12 \\
			20  &  ---  & 35.12   & --- & 7.38   & 0.21 & ---& 6.06   & 0.17 \\
			25  &  ---  & 111.18    & --- & 23.91  & 0.22 & --- & 42.94  & 0.39 \\
			30  &  ---  & 742.96   & --- & 147.79 & 0.20 & --- & 116.78 & 0.16 \\
			35  &  ---  & 113.84   & --- & 12.86  & 0.11 & --- & 13.33  & 0.12 \\
			40  &  ---  & 240.44   & --- & 38.89  & 0.16 & --- & 23.58  & 0.10 \\
			45  &  ---  & 306.41    & --- & 38.90  & 0.13 & --- & 18.52  & 0.06 \\
			50  &  --- & 2524.17  & --- & 716.95 & 0.28 & --- & 607.11 & 0.24 \\ \hline
			Geo Mean &         &      &   &        & 0.21 &         &        & 0.18 \\ \hline
		\end{tabular}
	\end{threeparttable}
\end{table}

\noindent\textbf{Experiment 2- Mixed-Integer Packing MIECP:} For this experiment, we set $b_i=2n$ for each $i\in [m]$, $n=200$, $p\in\{10, 20, \ldots, 100\}$ and let half of the variables be continuous, and the remaining half be binary (i.e., $t=n/2$). 
The results for large-scale mixed-integer packing cases can be found in Table \ref{pack_mip_table_sec2}-Table \ref{pack_mip_table_sec5}, where only the cases with the Gap greater than $10^{-4}$ or running time within the time limit are displayed. In this experiment, our best method (i.e., Example \ref{emp_3} Best Scale) on average is more than $90\%$ shorter than MOSEK, with Gap being less than $10^{-4}$.

The results for the proposed SOC approximation methods are shown in Tables \ref{pack_mip_table_sec2} and \ref{pack_mip_table_sec3}.
For the mixed-integer packing MIECP, both Section \ref{sec_exp_lim} and Section \ref{sec_exp_taylor} methods cannot solve all the cases with the Gap being no larger than $10^{-4}$ within the time limit, while all the other methods are consistently better than MOSEK. 
The average running time of Example \ref{emp_1}, Example \ref{emp_2}, and Example \ref{emp_3} is $54\%$, $67\%$, $78\%$ shorter than that of MOSEK, respectively. Example \ref{emp_3} with Best Scale is the best among all the methods since its average running time is only $9\%$ of that of MOSEK. Particularly, Example \ref{emp_3} with Best Scale only needs one SOC constraint to approximate each exponential conic constraint. It is worthy of emphasizing that we find the approximate solution for Example \ref{emp_3} with Best Scale, denoted as Best Scale procedure, by executing Example \ref{emp_3} with the approximate solution equal to $2^{-4}$ and 2-second time limit. We also see that both Section \ref{sec_exp_lim} and Section \ref{sec_exp_taylor} methods run into numerical issues, while the shifting method improves them significantly. 

\begin{table}[htbp]
	\centering
	\caption{Numerical Results of Section \ref{sec_log} Methods for Solving the Mixed-Integer Packing MIECP}
	\label{pack_mip_table_sec2}
	\begin{threeparttable}
		\scriptsize\setlength{\tabcolsep}{1.0pt}
		\begin{tabular}{c|rr|rrr|rrr|rrr|rrr}
			\hline
			$n=200$ & \multicolumn{2}{c|}{MOSEK} & \multicolumn{3}{c|}{Example \ref{emp_1}\tnote{i}} & \multicolumn{3}{c|}{Example \ref{emp_2}\tnote{ii}} & \multicolumn{3}{c|}{Example \ref{emp_3}\tnote{iii}} & \multicolumn{3}{c}{Example \ref{emp_3} Best Scale\tnote{iv}}  \\ \hline
			$p$ & \multicolumn{1}{c}{Gap} & \multicolumn{1}{c|}{Time (s)} & \multicolumn{1}{c}{Gap} & \multicolumn{1}{c}{Time (s)} & \multicolumn{1}{c|}{Ratio} & \multicolumn{1}{c}{Gap} & \multicolumn{1}{c}{Time (s)} & \multicolumn{1}{c|}{Ratio} & \multicolumn{1}{c}{Gap} & \multicolumn{1}{c}{Time (s)} & \multicolumn{1}{c|}{Ratio} & \multicolumn{1}{c}{Gap} & \multicolumn{1}{c}{Time (s)} & \multicolumn{1}{c}{Ratio} \\ \hline
			10	&	---	&	2.80	&	---	&	2.41	&	0.86	&	---	&	1.98	&	0.71	&	---	&	0.71	&	0.25	&	---	&	0.87	&	0.31 \\ 	
			20	&	---	&	5.07	&	---	&	1.30	&	0.26	&	---	&	1.19	&	0.23	&	---	&	0.86	&	0.17	&	---	&	1.18	&	0.23 \\	
			30	&	---	&	66.96	&	---	&	19.19	&	0.29	&	---	&	11.51	&	0.17	&	---	&	11.02	&	0.16	&	---	&	5.22	&	0.08 \\	
			40	&	---	&	158.62	&	---	&	60.74	&	0.38	&	---	&	34.48	&	0.22	&	---	&	31.83	&	0.20	&	---	&	10.30	&	0.06 \\	
			50	&	---	&	119.30	&	---	&	57.94	&	0.49	&	---	&	30.80	&	0.26	&	---	&	19.87	&	0.17	&	---	&	6.59	&	0.06 \\	
			60	&	---	&	945.54	&	---	&	209.07	&	0.22 &	---	&	148.00	&	0.16	&	---	&	98.70	&	0.10	&	---	&	32.98	&	0.03 \\	
			70	&	---	&	1774.46	&	---	&	583.00	&	0.33	&	---	&	561.01	&	0.32	&	---	&	292.50	&	0.16	&	---	&	88.63	&	0.05 \\
			80	&	---	&	224.86	&	---	&	81.89	&	0.36 &	---	&	76.54	&	0.34	&	---	&	67.31	&	0.30	&	---	&	19.19	&	0.09 \\
			90	&	---	&	1636.50	&	---	&	778.97	&	0.48	&	---	&	499.44	&	0.31	&	---	&	363.26	&	0.22	&	---	&	83.93	&	0.05 \\
			100	&	---	&	393.98	&	---	&	1062.28	&	2.70	&	---	&	702.11	&	1.78	&	---	&	401.33	&	1.02	&	---	&	86.15	&	0.22 \\	 \hline
			Geo Mean	&		&	&		&		&	0.46 &		&		&	0.33	&		&		&	0.22	&		&		&	0.09\\ \hline
		\end{tabular}
		\begin{tablenotes}
			\item[i] $N=8, a=1$;
			\item[ii] $N=6, s=4$;
			\item[iii] $N=3$ and the approximation solution is $2^{-4}$;
			\item[iv] $N=1$ and in the Best Scale procedure, we run Example \ref{emp_3} with $(N=1, AS=2^{-4}, TL=2)$.
		\end{tablenotes}
	\end{threeparttable}
\end{table}

\begin{table}[htbp]
	\centering
	\caption{Numerical Results of Section \ref{sec_exp} Methods for Solving the Mixed-Integer Packing MIECP}
	\label{pack_mip_table_sec3}
	\begin{threeparttable}
		\scriptsize\setlength{\tabcolsep}{1.0pt}
		\begin{tabular}{c|rr|rrr|rrr|rrr|rrr}
			\hline
			$n=200$ & \multicolumn{2}{c|}{MOSEK}  & \multicolumn{3}{c|}{Section \ref{sec_exp_lim}\tnote{i}} & \multicolumn{3}{c|}{Section \ref{sec_exp_taylor}\tnote{ii}} & \multicolumn{3}{c|}{Section \ref{sec_exp_lim} Shift\tnote{iii}} & \multicolumn{3}{c}{Section \ref{sec_exp_taylor} Shift\tnote{iv}} \\ \hline
			$p$ & \multicolumn{1}{c}{Gap} & \multicolumn{1}{c|}{Time (s)} & \multicolumn{1}{c}{Gap} & \multicolumn{1}{c}{Time (s)} & \multicolumn{1}{c|}{Ratio} & \multicolumn{1}{c}{Gap} & \multicolumn{1}{c}{Time (s)} & \multicolumn{1}{c|}{Ratio} & \multicolumn{1}{c}{Gap} & \multicolumn{1}{c}{Time (s)} & \multicolumn{1}{c|}{Ratio} & \multicolumn{1}{c}{Gap} & \multicolumn{1}{c}{Time (s)} & \multicolumn{1}{c}{Ratio} \\ \hline
			10	&	---	&	2.80 &	2.4e-3	& 2.22	& 0.79	& 1.2e-4 &	1.88 &	0.67 &	---	& 2.48	&	0.89 & ---	& 0.97	& 0.35 \\		
			20	&	---	&	5.07 &	2.3e-3	& 9.69	& 1.91	& --- & 2.26 & 0.44	& ---	& 4.55	&	0.90 & --- & 1.46 &	0.29 \\
			30	&	---	&	66.96 &	2.2e-3	& 31.34	& 0.47	& --- & 8.61 & 0.13	& ---	& 13.65	&	0.20 & --- & 7.03	& 0.10  \\					
			40	&	---	&	158.62 & 2.2e-3	& 53.16	& 0.34	& --- &	21.24 &	0.13 & ---	& 29.98	&	0.19 & --- &	15.22 &	0.10 \\			
			50	&	---	&	119.30 & 2.1e-3	& 59.63	& 0.50	& ---	& 15.00	& 0.13 & --- & 24.24	&	0.20 & --- &	16.01 &	0.13 \\				
			60	&	---	&	945.54 & 2.2e-3	& 459.42 &	0.49 & 2.3e-4	& 48.02	& 0.05	& ---	& 128.16	&	0.14 & ---	& 43.6	& 0.05 \\				
			70	&	---	&	1774.46 & 2.2e-3 & 489.68 &	0.28 & 2.6e-4	& 99.34	& 0.06	& --- & 297.84	&	0.17 & ---	& 81.96	& 0.05	\\				
			80	&	---	&	224.86 & 2.1e-3	& 88.69	& 0.39	& ---	& 37.06	& 0.16 &	---	& 33.04	&	0.15 &	---	& 17.51	& 0.08	\\				
			90	&	---	&	1636.50 & 2.1e-3 & 969.90 &	0.59 &	1.2e-4	& 136.14 &	0.08 & ---	& 2215.28	&	1.35 &	---	& 111.63 & 0.07 \\			
			100	&	---	&	393.98 & 2.1e-3	& 963.85 &	2.45 &	1.4e-4	& 157.30 &	0.40 & ---	& 367.79	&	0.93  &	---	& 115.81 & 0.29 \\ \hline	
			Geo Mean	&		&			&		&		&	0.63	&			&		&	0.16	&	&	&	0.35	&		&		&	0.12 \\ \hline	
		\end{tabular}
		\begin{tablenotes}
			\item[i] $N=20$;
			\item[ii] $N=6,s=1$;
			\item[iii] $N=5$ and using the same approximate solution as Example \ref{emp_3} Best Scale;
			\item[iv] $N=0,s=1$ and using the same approximate solution as Example \ref{emp_3} Best Scale.
		\end{tablenotes}
	\end{threeparttable}
\end{table}

The results for the proposed polyhedral approximation methods are shown in Table \ref{pack_mip_table_sec5}. The polyhedral approximation method, i.e., Branch and Cut, performs well. Specifically, the average running time of the Branch and Cut method is around $27\%$ of that of MOSEK.

\begin{table}[htbp]
	\centering
	\caption{Numerical Results of Section \ref{sec_outer} Methods for Solving the Mixed-Integer Packing MIECP}
	\label{pack_mip_table_sec5}
	\begin{threeparttable}
		\scriptsize\setlength{\tabcolsep}{1.0pt}
		\begin{tabular}{c|rr|rrr}
			\hline
			$n=200$ & \multicolumn{2}{c|}{MOSEK}  & \multicolumn{3}{c}{Branch and Cut} \\ \hline
			$p$ & \multicolumn{1}{c}{Gap} & \multicolumn{1}{c|}{Time (s)} & \multicolumn{1}{c}{Gap} & \multicolumn{1}{c}{Time (s)} & \multicolumn{1}{c}{Ratio} \\ \hline
			10	&	---	&	2.80  & --- &	2.00 &	0.71\\		
			20	&	---	&	5.07  &	---	& 3.02	& 0.60\\
			30	&	---	&	66.96  & --- &	15.87 &	0.24\\					
			40	&	---	&	158.62  &	---	& 29.90	& 0.19\\			
			50	&	---	&	119.30  &	---	& 16.97	& 0.14\\				
			60	&	---	&	945.54 & --- &	159.46	& 0.17\\				
			70	&	---	&	1774.46 & --- &	238.71	& 0.13\\				
			80	&	---	&	224.86 & --- &	45.94 &	0.20\\				
			90	&	---	&	1636.50 & --- &	361.91	& 0.22\\			
			100	&	---	&	393.98 & --- &	323.73	& 0.82\\ \hline	
			Geo Mean	&		&		&		&		&	0.27\\ \hline	
		\end{tabular}
	\end{threeparttable}
\end{table}

\subsection{Covering MIECP}
In this subsection, we consider the following mixed-integer convex covering problem
\begin{align}\label{eq_num_study2}
	\min_{\bm{x}\in X}\left\{\sum_{\ell\in [p]}\left(\sum_{j\in [n]}c_{\ell j}x_j\right)\log\left(\sum_{j\in [n]}c_{\ell j}x_j\right): \sum_{j\in [n]}a_{ij}x_j \geq  b_i, \forall i \in [m]\right\},
\end{align}
where $X=\{0,1\}^t\times [0,1]^{n-t}$. Note that the problem \eqref{eq_num_study2} can be formulated as the following covering MIECP
\begin{align}\label{eq_num_study_exp_cone2}
	\min_{\bm{x}\in X}\left\{\sum_{\ell\in [p]}v_\ell: \left(1, \sum_{j\in [n]}c_{\ell j}x_j,-v_\ell\right)\in K_{\exp}(0), \forall \ell\in [p], \sum_{j\in [n]}a_{ij}x_j \geq  b_i, \forall i \in [m]\right\}.
\end{align}

Similar to the packing MIECP, we conduct two experiments to test the proposed SOC and polyhedral approximations for solving \eqref{eq_num_study_exp_cone2}, where we consider pure binary ($t=n$) and mixed-integer ($t\in [n-1]$) experiments and compare our results with MOSEK.  In our testing instances, we set $m=100$ and suppose $a_{ij}\sim \mathrm{int}(0, 9)$, $b_i=2n$, $c_{\ell j}\sim \mathrm{int}(0, 9)/n$ for all $i \in [m], j\in [n], \ell\in [p]$. 
The time limit is set to be $3600$ seconds for both experiments. 

\noindent\textbf{Experiment 3- Binary Covering MIECP:} 
For the binary experiment, we solve the small-scale instances with $n=30$, $p\in\{5, 10, \ldots, 50\}$ to do sensitive analyses of parameters and learn the patterns and then solve the large-scale instances with $n=50$, $p\in\{5, 10, \ldots, 50\}$. Due to the page limit, we do not report these results. 
The results for large-scale binary covering instances can be found in Table \ref{cover_binary_table_sec2}-Table \ref{cover_binary_table_sec5}, where only the cases with the Gap greater than $10^{-4}$ or running time within the time limit are displayed. In this experiment, our best method (i.e., Branch and Cut), on average, is at least $90\%$ shorter than MOSEK, with Gap being less than $10^{-4}$. 

The results for the proposed SOC approximation methods are shown in Tables \ref{cover_binary_table_sec2} and \ref{cover_binary_table_sec3}.
It is seen that the Section \ref{sec_exp_lim} method cannot solve all the cases with the Gap being no larger than $10^{-4}$ within the time limit, while all the other methods are consistently better than MOSEK. On average, the running time of Example \ref{emp_1}, Example \ref{emp_2}, and Example \ref{emp_3} can be $67\%$, $79\%$, $80\%$ shorter than that of MOSEK, respectively. It is worthy of mentioning that Example \ref{emp_3} with Best Scale is $88\%$ and Section \ref{sec_exp_taylor} Shift method is $89\%$ shorter than MOSEK. Particularly, both Example \ref{emp_3} with Best Scale and Section \ref{sec_exp_taylor} Shift method only need one SOC constraint to approximate each exponential conic constraint. It is worthy of emphasizing that we find the approximate solution for Example \ref{emp_3} with Best Scale and Section \ref{sec_exp_taylor} Shift by executing Example \ref{emp_3} with the approximate solution equal to $2^{-1}$ and 2-second time limit. We also see that the Section \ref{sec_exp_lim} method runs into numerical issues, while the shifting method improves it significantly.

\begin{table}[htbp]
	\centering
	\caption{Numerical Results of Section \ref{sec_log} Methods for Solving the Binary Covering MIECP}
	\label{cover_binary_table_sec2}
	\begin{threeparttable}
		\scriptsize\setlength{\tabcolsep}{1.0pt}
		\begin{tabular}{c|rr|rrr|rrr|rrr|rrr}
			\hline
			$n=50$ & \multicolumn{2}{c|}{MOSEK} & \multicolumn{3}{c|}{Example \ref{emp_1}\tnote{i}} & \multicolumn{3}{c|}{Example \ref{emp_2}\tnote{ii} } & \multicolumn{3}{c|}{Example \ref{emp_3}\tnote{iii}} & \multicolumn{3}{c}{Example \ref{emp_3} Best Scale\tnote{iv}} \\ \hline
			$p$ & \multicolumn{1}{c}{Gap} & \multicolumn{1}{c|}{Time (s)} & \multicolumn{1}{c}{Gap} & \multicolumn{1}{c}{Time (s)} & \multicolumn{1}{c|}{Ratio} & \multicolumn{1}{c}{Gap} & \multicolumn{1}{c}{Time (s)} & \multicolumn{1}{c|}{Ratio} & \multicolumn{1}{c}{Gap} & \multicolumn{1}{c}{Time (s)} & \multicolumn{1}{c|}{Ratio} & \multicolumn{1}{c}{Gap} & \multicolumn{1}{c}{Time (s)} & \multicolumn{1}{c}{Ratio} \\ \hline
			5   & ---  & 0.89   & ---  & 0.24   & 0.27 	& ---  & 0.20   & 0.22 	& ---  & 0.33   & 0.37 & ---  & 1.86  & 2.09 \\			
			10  & ---  & 18.30  & ---  & 5.42   & 0.30 & ---  & 3.79   & 0.21	& ---  & 2.16   & 0.12 &---  & 3.51  & 0.19 \\			
			15  & ---  & 49.04  & ---  & 18.60  & 0.38 & ---  & 9.76   & 0.20 	& ---  & 7.78   & 0.16 & ---  & 5.83  & 0.12 \\			
			20  & ---  & 159.79 & ---  & 39.52  & 0.25 & ---  & 23.44  & 0.15  & ---  & 20.51  & 0.13 & ---  & 10.91 & 0.07 \\			
			25  & ---  & 368.94 & ---  & 59.29  & 0.16 	& ---  & 54.23  & 0.15  & ---  & 30.71  & 0.08 & ---  & 21.49 & 0.06 \\			
			30  & ---  & 720.02 & ---  & 332.12 & 0.46 & ---  & 178.82 & 0.25  & ---  & 156.53 & 0.22 & ---  & 61.50 & 0.09 \\			
			35  & ---  & 313.22 & ---  & 100.84 & 0.32	& ---  & 58.94  & 0.19  & --- & 77.35  & 0.25 & ---  & 21.71 & 0.07 \\			
			40  & ---  & 827.92 & ---  & 376.44 & 0.45 & ---  & 192.63 & 0.23  & ---  & 334.41 & 0.40 & ---  & 74.76 & 0.09 \\			
			45  & ---  & 532.81 & ---  & 216.65 & 0.41 & ---  & 163.42 & 0.31  & ---  & 139.18 & 0.26 & ---  & 33.33 & 0.06 \\			
			50  & ---  & 307.16 & ---  & 145.28 & 0.47 & ---  & 65.50  & 0.21  & ---  & 58.03  & 0.19 & ---  & 25.15 & 0.08 \\ \hline			
			Geo Mean &         &        &         &        & 0.33  &         &        & 0.21 &         &        & 0.20 &         &       & 0.12 \\ \hline
		\end{tabular}
		\begin{tablenotes}
			\item[i] $N=4, a=1$;
			\item[ii] $N=3, s=1$;
			\item[iii] $N=2$ and the approximation solution is $2^{-1}$; 
			\item[iv] $N=1$ and in the Best Scale procedure, we run Example \ref{emp_3} with 3-tuple $(N=1, AS=2^{-1}, TL=2)$.
		\end{tablenotes}
	\end{threeparttable}
\end{table}

\begin{table}[htbp]
	\centering
	\caption{Numerical Results of Section \ref{sec_exp} Methods for Solving the Binary Covering MIECP}
	\label{cover_binary_table_sec3}	
	\begin{threeparttable}
		\scriptsize\setlength{\tabcolsep}{1.0pt}
		\begin{tabular}{c|rr|rrr|rrr|rrr|rrr}
			\hline
			$n=50$ & \multicolumn{2}{c|}{MOSEK} & \multicolumn{3}{c|}{Section \ref{sec_exp_lim}\tnote{i} } & \multicolumn{3}{c|}{Section \ref{sec_exp_taylor}\tnote{ii}} & \multicolumn{3}{c|}{Section \ref{sec_exp_lim} Shift\tnote{iii} } & \multicolumn{3}{c}{Section \ref{sec_exp_taylor} Shift\tnote{iv} } \\ \hline
			$p$ & \multicolumn{1}{c}{Gap} & \multicolumn{1}{c|}{Time (s)}  & \multicolumn{1}{c}{Gap} & \multicolumn{1}{c}{Time (s)} & \multicolumn{1}{c|}{Ratio} & \multicolumn{1}{c}{Gap} & \multicolumn{1}{c}{Time (s)} & \multicolumn{1}{c|}{Ratio} & \multicolumn{1}{c}{Gap} & \multicolumn{1}{c}{Time (s)} & \multicolumn{1}{c|}{Ratio} & \multicolumn{1}{c}{Gap} & \multicolumn{1}{c}{Time (s)} & \multicolumn{1}{c}{Ratio} \\ \hline
			5    & --- & 0.89   & 7.1e-4 & 0.27 & 0.31 & ---  & 6.75    & 7.61 & ---  & 0.48   & 0.54 & ---   & 0.48  & 0.55  \\
			10   & ---  & 18.30   & 8.6e-4  & 9.43 & 0.52 & ---  & 5.29    & 0.29 & ---  & 7.59   & 0.41 & ---  & 3.39  & 0.19  \\
			15   & ---  & 49.04  & 9.4e-4  & 13.32 & 0.27 & ---  & 12.42   & 0.25 & ---  & 12.75  & 0.26 & ---  & 6.68  & 0.14  \\
			20   & ---  & 159.79  & 1.0e-3  & 32.51 & 0.20 & ---  & 31.19   & 0.20 & --- & 26.63  & 0.17 & ---  & 8.14  & 0.05  \\
			25   & ---  & 368.94  & 1.0e-3  & 1167.74 & 3.17 & ---  & 1096.47 & 2.97 & ---  & 44.59  & 0.12 & ---  & 22.54 & 0.06  \\
			30   & ---  & 720.02  & 1.0e-3 & 2750.53 & 3.82 & ---  & 163.64  & 0.23 & ---  & 126.02 & 0.18 & ---  & 76.89 & 0.11  \\
			35   & ---  & 313.22  & 9.2e-4  & 1661.78 & 5.31 & ---  & 106.81  & 0.34 & ---  & 61.83  & 0.20 & ---  & 31.21 & 0.10  \\
			40   & ---  & 827.92  & 9.6e-4 & 2595.93 & 3.14 & ---  & 536.93  & 0.65 & ---  & 164.26 & 0.20 & --- & 66.03 & 0.08  \\
			45   & ---  & 532.81  & 9.6e-4  & 1638.89 & 3.08 & ---  & 2047.01 & 3.84 & ---  & 164.99 & 0.31 & ---  & 43.12 & 0.08  \\
			50   &  ---  & 307.16  & 9.2e-4 & 1126.08 & 3.67 & ---  & 1366.49 & 4.45 & ---  & 893.44 & 2.91 & ---  & 37.14 & 0.12  \\ \hline
			Geo Mean &         &        &         &         & 1.35 &         &         & 0.88 &         &        & 0.31 &         &       & 0.11 \\ \hline
		\end{tabular}
		\begin{tablenotes}
			\item[i] $N=9$;
			\item[ii] $N=6,s=1$;
			\item[iii] $N=6$ and using the same approximate solution as Example \ref{emp_3} Best Scale;  
			\item[iv] $N=0,s=1$ and using the same approximate solution as Example \ref{emp_3} Best Scale.
		\end{tablenotes}
	\end{threeparttable}
\end{table}

The results for the proposed polyhedral approximation methods are shown in Table \ref{cover_binary_table_sec5}. The two polyhedral approximation methods, i.e., Cutting Plane and Branch and Cut, perform very well, and the Branch and Cut method is the best among all the methods. Specifically, the average running time of the Cutting Plane method is around $7\%$ of that of MOSEK, and using delayed cut generation procedure (i.e., Branch and Cut method) has an average running time of around $6\%$ of that of MOSEK.

\begin{table}[htbp]
	\centering
	\caption{Numerical Results of Section \ref{sec_outer} Methods for Solving the Binary Covering MIECP}
	\label{cover_binary_table_sec5}	
	\begin{threeparttable}
		\scriptsize\setlength{\tabcolsep}{1.0pt}
		\begin{tabular}{c|rr|rrr|rrr}
			\hline
			$n=50$ & \multicolumn{2}{c|}{MOSEK} & \multicolumn{3}{c|}{Cutting Plane} & \multicolumn{3}{c}{Branch and Cut} \\ \hline
			$p$ & \multicolumn{1}{c}{Gap} & \multicolumn{1}{c|}{Time (s)}  & \multicolumn{1}{c}{Gap} & \multicolumn{1}{c}{Time (s)} & \multicolumn{1}{c|}{Ratio} & \multicolumn{1}{c}{Gap} & \multicolumn{1}{c}{Time (s)} & \multicolumn{1}{c}{Ratio}  \\ \hline
			5    & --- & 0.89   & ---  & 0.14  & 0.16 & ---     & 0.37  & 0.42 \\
			10   & ---  & 18.30   & ---  & 2.96  & 0.16 & ---  & 2.32  & 0.13 \\
			15   & ---  & 49.04   & --- & 3.48  & 0.07 & ---  & 4.67  & 0.10 \\
			20   & ---  & 159.79  & --- & 7.40  & 0.05 & ---     & 5.76  & 0.04 \\
			25   & ---  & 368.94  & ---  & 30.03 & 0.08 & ---  & 10.23 & 0.03 \\
			30   & ---  & 720.02  & --- & 58.17 & 0.08 & ---  & 27.26 & 0.04 \\
			35   & ---  & 313.22  & --- & 12.68 & 0.04 & --- & 13.28 & 0.04 \\
			40   & ---  & 827.92  & --- & 56.95 & 0.07 & ---  & 36.91 & 0.04 \\
			45   & ---  & 532.81  & --- & 35.53 & 0.07 & ---  & 22.45 & 0.04 \\
			50   &  ---  & 307.16  & --- & 9.55  & 0.03 & ---     & 21.22 & 0.07 \\ \hline
			Geo Mean &         &        &         &       & 0.07 &         &       & 0.06 \\ \hline
		\end{tabular}
	\end{threeparttable}
\end{table}

\noindent\textbf{Experiment 4- Mixed-Integer Covering MIECP:} 
For the mixed-integer experiment, we set $n=200$, $p\in\{10, 20, \ldots, 100\}$ and let half of the variables be continuous, and the remaining half be binary (i.e., $t=n/2$). 
The results for the mixed-integer covering instances can be found in Table \ref{cover_mip_table_sec2}-Table \ref{cover_mip_table_sec5}, where only the cases with the Gap greater than $10^{-4}$ or running time within the time limit are displayed. In this experiment, two of our best methods (i.e., Example \ref{emp_3} Best Scale and Section \ref{sec_exp_taylor} Shift) on average are around $89\%$ shorter than MOSEK, with Gap being less than $10^{-4}$. 

The results for the proposed SOC approximation methods are shown in Tables \ref{cover_mip_table_sec2} and \ref{cover_mip_table_sec3}.
For mixed-integer covering MIECP, Section \ref{sec_exp_lim} and Section \ref{sec_exp_taylor} 
methods cannot solve all the cases with the Gap being no larger than $10^{-4}$ within the time limit, while all the other methods are consistently better than MOSEK. On average, the running time of Example \ref{emp_1}, Example \ref{emp_2}, and Example \ref{emp_3} can be $79\%$, $82\%$, $87\%$ shorter than that of MOSEK, respectively. Both Example \ref{emp_3} with Best Scale and Section \ref{sec_exp_taylor} Shift methods are the best among all the methods since their average running time is only $11\%$ of that of MOSEK. Particularly, both Example \ref{emp_3} with Best Scale and Section \ref{sec_exp_taylor} Shift only need one SOC constraint to approximate  each exponential conic constraint. It is worthy of emphasizing that we find the approximate solution for Example \ref{emp_3} with Best Scale and Section \ref{sec_exp_taylor} Shift method by executing Example \ref{emp_3} with the approximate solution equal to $2^{-1}$ and 2-second time limit. We also see that Section \ref{sec_exp_lim} and Section \ref{sec_exp_taylor} methods run into numerical issues, while the shifting method improves them significantly. 

\begin{table}[htbp]
	\centering
	\caption{Numerical Results of Section \ref{sec_log} Methods for Solving the Mixed-Integer Covering MIECP}
	\label{cover_mip_table_sec2}
	\begin{threeparttable}
		\scriptsize\setlength{\tabcolsep}{1.0pt}
		\begin{tabular}{c|rr|rrr|rrr|rrr|rrr}
			\hline
			$n=200$ & \multicolumn{2}{c|}{MOSEK} & \multicolumn{3}{c|}{Example \ref{emp_1}\tnote{i}} & \multicolumn{3}{c|}{Example \ref{emp_2}\tnote{ii} } & \multicolumn{3}{c|}{Example \ref{emp_3}\tnote{iii}} & \multicolumn{3}{c}{Example \ref{emp_3} Best Scale\tnote{iv}}  \\ \hline
			$p$ & \multicolumn{1}{c}{Gap} & \multicolumn{1}{c|}{Time (s)} & \multicolumn{1}{c}{Gap} & \multicolumn{1}{c}{Time (s)} & \multicolumn{1}{c|}{Ratio} & \multicolumn{1}{c}{Gap} & \multicolumn{1}{c}{Time (s)} & \multicolumn{1}{c|}{Ratio} & \multicolumn{1}{c}{Gap} & \multicolumn{1}{c}{Time (s)} & \multicolumn{1}{c|}{Ratio} & \multicolumn{1}{c}{Gap} & \multicolumn{1}{c}{Time (s)} & \multicolumn{1}{c}{Ratio} \\ \hline
			10 & --- & 3.40	& --- &	1.10 & 0.32	& --- & 0.79 &	0.23  & --- & 0.71 & 0.21 & --- & 1.72 & 0.50	\\
			20 & --- & 8.38	& --- &	1.56 & 0.19	& --- & 3.64 &	0.43  & --- & 1.17 & 0.14 & --- & 2.77 & 0.33	\\
			30 & --- & 74.58 & --- & 10.33 & 0.14 &	---	& 8.85 & 0.12  & ---	& 7.24 & 0.10 &	---	& 6.00 & 0.08 \\
			40 & --- & 130.38 &	---	& 20.33	& 0.16	& --- & 14.47 & 0.11 & --- & 12.29 & 0.09 & ---	& 8.58	& 0.07	\\
			50 & --- & 198.82 &	---	& 26.63	& 0.13	& --- & 19.20 & 0.10 & --- & 15.73 & 0.08 & ---	& 11.99	& 0.06	\\
			60	& --- &	79.42 &	---	& 19.87	& 0.25	&	---	& 12.40 & 0.16 & --- & 9.84 & 0.12 & --- & 7.67	& 0.10	\\
			70	& --- &	113.44 & --- & 21.54 & 0.19	& --- & 15.18 & 0.13 & --- & 12.45 & 0.11 & ---	& 9.01	& 0.08	\\
			80	& --- &	1515.42	& --- &	183.89 & 0.12 & --- &	161.97 & 0.11 &	---	& 105.55 & 0.07 & --- &	60.74 &	0.04 \\
			90	& --- &	438.67 & --- & 82.62 & 0.19	& --- & 54.26 & 0.12 & --- & 62.93 & 0.14 & --- & 27.16	& 0.06	\\
			100	& --- &	357.46 & --- & 293.19 &	0.82 & --- & 304.25 &	0.85 &	---	& 191.73 & 0.54	& --- &	110.71	& 0.31	\\ \hline
			Geo Mean & & & & & 0.21	& & & 0.18 & & & 0.13 & & & 0.11 \\ \hline
		\end{tabular}
		\begin{tablenotes}
			\item[i] $N=3, a=1$;
			\item[ii] $N=3, s=1$;
			\item[iii] $N=2$ and the approximation solution is $2^{-1}$; 
			\item[iv] $N=1$ and in the Best Scale procedure, we run Example \ref{emp_3} with 3-tuple $(N=1, AS=2^{-1}, TL=2)$. 
		\end{tablenotes}
	\end{threeparttable}
\end{table}

\begin{table}[htbp]
	\centering
	\caption{Numerical Results of Section \ref{sec_exp} Methods for Solving the Mixed-Integer Covering MIECP}
	\label{cover_mip_table_sec3}	
	\begin{threeparttable}
		\scriptsize\setlength{\tabcolsep}{1.0pt}
		\begin{tabular}{c|rr|rrr|rrr|rrr|rrr}
			\hline
			$n=200$ & \multicolumn{2}{c|}{MOSEK} & \multicolumn{3}{c|}{Section \ref{sec_exp_lim}\tnote{i} } & \multicolumn{3}{c|}{Section \ref{sec_exp_taylor}\tnote{ii}} & \multicolumn{3}{c|}{Section \ref{sec_exp_lim} Shift\tnote{iii} } & \multicolumn{3}{c}{Section \ref{sec_exp_taylor} Shift\tnote{iv} } \\ \hline
			$p$ & \multicolumn{1}{c}{Gap} & \multicolumn{1}{c|}{Time (s)} & \multicolumn{1}{c}{Gap} & \multicolumn{1}{c}{Time (s)} & \multicolumn{1}{c|}{Ratio} & \multicolumn{1}{c}{Gap} & \multicolumn{1}{c}{Time (s)} & \multicolumn{1}{c|}{Ratio} & \multicolumn{1}{c}{Gap} & \multicolumn{1}{c}{Time (s)} & \multicolumn{1}{c|}{Ratio} & \multicolumn{1}{c}{Gap} & \multicolumn{1}{c}{Time (s)} & \multicolumn{1}{c}{Ratio} \\ \hline
			10	& --- &	3.40 & --- & 13.93 & 4.09 &	---	& 1.79 & 0.52 &	---	& 1.98	&	0.58 &	---	& 1.63 & 0.48 \\
			20	& --- &	8.38 & --- & 36.36 & 4.34 &	---	& 4.92 & 0.59 &	---	& 3.12	&	0.37 &	---	& 3.54 & 0.42 \\
			30	& --- &	74.58 &	---	& 242.64 & 3.25 & --- & 12.97 &	0.17 & --- &  9.41	&	0.13 &	---	& 6.69 & 0.09 \\
			40	& --- &	130.38 & 4.9e-4 & 376.26 & 2.89 & --- & 24.50	& 0.19 & --- & 	11.64	&	0.09 & --- &	8.54 & 0.07 \\
			50	& --- &	198.82 & 2.1e-3 & 400.18 &	2.01 & --- & 30.46	& 0.15 & --- & 	18.31	&	0.09 & --- &	12.10 &	0.06 \\
			60	& --- &	79.42 & --- & 204.16 &	2.57 & --- & 21.50	& 0.27 & --- & 	13.52	&	0.17 & --- & 7.13	& 0.09 \\
			70	& --- &	113.44 	& --- & 305.38 & 2.69 &	---	& 197.95 & 1.74	& --- &	13.44	& 0.12 & --- & 8.83 &	0.08 \\
			80	& --- &	1515.42	&	1.7e-3 & * & * & 1.5e-3	& *	& *	& --- &	87.86	&	0.06  & --- & 51.44 & 0.03 \\
			90	& --- &	438.67 	& 1.5e-4 & 1825.29 & 4.16 &	---	& 1633.22 &	3.72 & --- & 47.07	&	0.11 & --- &	22.49 &	0.05 \\
			100	& --- &	357.46  & 3.0e-3 & * &	* &	2.0e-3 & * & * & --- & 	144.77	&	0.41 & --- &	85.60 &	0.24 \\ \hline
			Geo Mean	& &		& & & 3.15 & & & 0.48 & & &	0.16 & & & 0.11 \\ \hline
		\end{tabular}
		\begin{tablenotes}
			\item[i] $N=12$;
			\item[ii] $N=6,s=1$;
			\item[iii] $N=3$ and using the same approximate solution as Example \ref{emp_3} Best Scale;  
			\item[iv] $N=0,s=1$ and using the same approximate solution as Example \ref{emp_3} Best Scale. 
		\end{tablenotes}
	\end{threeparttable}
\end{table}

The results for the proposed polyhedral approximation methods are shown in Table \ref{cover_mip_table_sec5}. The polyhedral approximation method, i.e., Branch and Cut, performs very well. Specifically, the average running time of the Branch and Cut method is around $12\%$ of that of MOSEK. 

\begin{table}[htbp]
	\centering
	\caption{Numerical Results of Section \ref{sec_outer} Methods for Solving the Mixed-Integer Covering MIECP}
	\label{cover_mip_table_sec5}	
	\begin{threeparttable}
		\scriptsize\setlength{\tabcolsep}{1.0pt}
		\begin{tabular}{c|rr|rrr}
			\hline
			$n=200$ & \multicolumn{2}{c|}{MOSEK} & \multicolumn{3}{c}{Branch and Cut} \\ \hline
			$p$ & \multicolumn{1}{c}{Gap} & \multicolumn{1}{c|}{Time (s)} & \multicolumn{1}{c}{Gap} & \multicolumn{1}{c}{Time (s)} & \multicolumn{1}{c}{Ratio} \\ \hline
			10	& --- &	3.40 & --- & 1.76 & 0.52\\
			20	& --- &	8.38 & --- &	1.68 & 0.20\\
			30	& --- &	74.58 &	---	& 7.45 & 0.10\\
			40	& --- &	130.38 & --- &	17.42 &	0.13\\
			50	& --- &	198.82 & --- & 13.22 &	0.07\\
			60	& --- &	79.42 & --- &	5.30 & 0.07\\
			70	& --- &	113.44 & --- & 8.02 & 0.07\\
			80	& --- &	1515.42	& --- & 137.53 & 0.09\\
			90	& --- &	438.67 & --- & 24.57 & 0.06\\
			100	& --- &	357.46 & --- & 111.92 & 0.31\\ \hline
			Geo Mean	& &	  & & &	0.12\\ \hline
		\end{tabular}
	\end{threeparttable}
\end{table}

\subsection{Sparse Logistic Regression (SLR)}
In this subsection, we consider the following sparse logistic regression (SLR)
\begin{align}\label{eq_num_study3}\footnotesize
	\min_{\scriptsize\begin{array}{c}
			\bm{\theta}\in \Re^d,\\\bm{z}\in \{0,1\}^{ d}
	\end{array}}\left\{\sum_{i\in [n]}\left[-y_i \log(h_{\bm{\theta}}(\bfx_i))-(1-y_i) \log(1-h_{\bm{\theta}}(\bfx_i))\right]+ \lambda \|\bm{\theta}\|_1: \begin{array}{c}
		\sum_{ j\in [d]} z_{ j}=k,\\
		\bm{\theta}-M_{\theta}\bfz \leq 0,\\
		-\bm{\theta}-M_{\theta}\bfz \leq 0
	\end{array}
	\right\},
\end{align}
where there are $n$ data points $\{(\bm{x}_i,y_i)\}_{ i\in[n]}\subseteq \Re^d\times \{0,1\}$,  $h_{\bm{\theta}}(\bfx_{ i})=[1+\exp(-\bm{\theta}^{\top}\bfx_{ i})]^{-1}$ denotes the sigmoid function, and $M_{\theta}$ is the big-M coefficient. Note that  we can choose $M_{\theta}:=n\log(2)/\lambda$. 
Indeed, observe that the optimal value of SLR \eqref{eq_num_study3} must be less than or equal to $n\log(2)$ since $\bm{\theta}=\bm{0}$ is a feasible solution with objective value equal to $n\log(2)$. Thus, there exists an optimal solution such that 
$$\sum_{i\in [n]}\left[-y_i \log(h_{\bm{\theta}}(\bfx_i))-(1-y_i) \log(1-h_{\bm{\theta}}(\bfx_i))\right]+ \lambda \|\bm{\theta}\|_1 \leq n\log(2).$$ 
Since $\sum_{i\in [n]}\left[-y_i \log(h_{\bm{\theta}}(\bfx_i))-(1-y_i) \log(1-h_{\bm{\theta}}(\bfx_i))\right] \geq 0$, we must have $\lambda\|\bm{\theta}\|_1 \leq n\log(2)$. Therefore, we can upper bound $\|\bm{\theta}\|_1$ as  $\|\bm{\theta}\|_1 \leq n\log(2)/\lambda=M_{\theta}$, i.e., $|\theta_j|\leq M_{\theta}$ is valid for each $j\in [d]$.

Note that SLR \eqref{eq_num_study3} can be formulated as the following MIECP
\begin{align}\label{eq_num_study_exp_cone3}\footnotesize
	\min_{\bm{\theta}\in \Re^d, \bm{z}\in \{0,1\}^{ d},\bm{p}_1,\bm{p}_2} \left\{\sum_{i\in [n]} t_i+\lambda \|\bm{\theta}\|_1: \begin{array}{c}	
		\sum_{ j\in [d]} z_{ j}=k, \bm{\theta}-M_{\theta}\bfz \leq 0, -\bm{\theta}-M_{\theta}\bfz \leq 0,\\
		p_{i1}+p_{i2}=1, \forall i \in [n],\\
		\left(p_{i1}, 1,(1-2y_i)\bm{\theta}^{\top}\bfx_i-t_i\right)\in K_{\exp}(0), \forall i \in [n],\\
		\left(p_{i2}, 1,-t_{i}\right)\in K_{\exp}(0), \forall i \in [n]
	\end{array}
	\right\}.
\end{align}

In this subsection, we conduct two experiments to test the proposed SOC and polyhedral approximations for solving \eqref{eq_num_study_exp_cone3}, where we consider moderate-scale and large-scale cases and compare our results with MOSEK. In our testing instances, we set $k=20$, $\lambda=0.01$.
Overall, the proposed approximation methods can effectively solve SLR using real datasets and outperform MOSEK. 

\noindent\textbf{Experiment 5- Moderate-Scale SLR Cases:} In this experiment, we solve SLR with the UCI student mathematics performance dataset \citep{cortez2008using}
using MOSEK, Example \ref{emp_3} with Best Scale, and Branch and Cut methods. These three are identified as the best ones in the previous experiments. In particular, in Example \ref{emp_3} with Best Scale, we use the obtained approximate solution to update the $M_{\theta}$ value; and to avoid numerical issues, when implementing the Branch and Cut algorithm, we first run the gradient descent method to solve the continuous relaxation of SLR and then adding all the gradient inequalities into the root node of the branch and bound tree. This dataset contains $31$ variables and a binary experimental class, i.e., the positive class for the grade above the median and the negative class for otherwise. We also use the linear kernel to increase features to ${ d = 496}$, and we consider the number of data points being $n\in\{10 20, \ldots, 100\}$. 

\begin{wraptable}{r}{0.45\textwidth}
	\vspace{-15pt}
	\centering
	\caption{Numerical Results of Moderate-Scale SLR}
	\vspace{-5pt}
	\begin{threeparttable}
		\scriptsize\setlength{\tabcolsep}{1.0pt}
		\begin{tabular}{c|rr|rrr|rrr}
			\hline
			$d = 496$ & \multicolumn{2}{c|}{MOSEK} & \multicolumn{3}{c|}{Example \ref{emp_3} Best Scale\tnote{i}} & \multicolumn{3}{c}{Branch and Cut} \\ \hline
			$n$ & \multicolumn{1}{c}{Gap} & \multicolumn{1}{c|}{Time (s)} & \multicolumn{1}{c}{Gap} & \multicolumn{1}{c}{Time (s)} & \multicolumn{1}{c|}{Ratio} & \multicolumn{1}{c}{Gap} & \multicolumn{1}{c}{Time (s)} & \multicolumn{1}{c}{Ratio} \\ \hline
			10 & --- & 14.26 & --- & 0.35 & 0.02 & --- & 10.26 & 0.72  \\
			20 & --- & 21.63 & --- & 1.04 & 0.05 & --- & 35.42 & 1.64  \\
			30 & --- & 1.90 & --- & 1.84 & 0.97 & --- & 60.54 & 31.88 \\
			40 & --- & 338.50 & --- & 1.75 & 0.01 & --- & 65.66  & 0.19  \\
			50 & --- & 9.86 & --- & 4.88 & 0.49 & --- & 117.96 & 11.97 \\
			60 & --- & 9.05 & --- & 5.90 & 0.65 & --- & 98.98  & 10.94 \\
			70 & --- & 791.01 & --- & 8.63 & 0.01 & n/a & n/a & n/a \\
			80 & --- & 954.47 & --- & 4.55 & 0.00 & n/a & n/a & n/a \\
			90 & --- & 306.48 & --- & 17.19 & 0.06 & n/a & n/a & n/a \\
			100 & --- & 407.08 & --- & 12.20 & 0.03 & n/a & n/a & n/a \\ \hline
			Geo Mean &     &        &     &       & 0.05 &     &        & 3.14 \\ \hline
		\end{tabular}
		\begin{tablenotes}
			\item[i] $N=4$ and in the Best Scale procedure, we first run Example \ref{emp_3} with 3-tuple $(N=4, AS=2^{-2}, TL=10)$. 
		\end{tablenotes}
	\end{threeparttable}
	\label{slr_small_table_studentm}
	\vspace{-20pt}
\end{wraptable}

Table \ref{slr_small_table_studentm} summarizes the results for moderate-scale SLR instances, where only the cases with the Gap greater than $10^{-4}$ or running time within the time limit are displayed. Both MOSEK and Example \ref{emp_3} with Best Scale can solve all the cases with the Gap being less than $10^{-4}$. Notably, the polyhedral approximation method, i.e., Branch and Cut with the first-order method, can only solve the cases when $n\leq 60$ within the time limit of $3600$ seconds, and it takes a much longer time than the other two methods. This may be because the big-M coefficient causes numerical difficulty for the MILP solver. Example \ref{emp_3} with Best Scale is the best among all the methods since its average running time is only $5\%$ of that of MOSEK. Particularly, Example \ref{emp_3} with Best Scale needs four SOC constraints to approximate each exponential conic constraint. It is worthy of emphasizing that we find the approximate solution for Example \ref{emp_3} with Best Scale by executing Example \ref{emp_3} with the approximate solution equal to $2^{-2}$ and a 10-second time limit. In summary, in this experiment, our best method (i.e., Example \ref{emp_3} Best Scale) on average is $95\%$ shorter than MOSEK, with Gap being less than $10^{-4}$.

\begin{wraptable}{r}{0.35\textwidth}
	\vspace{-15pt}
	\centering
	\caption{Numerical Results of Large-Scale SLR}
	\vspace{-5pt}
	\label{slr_large_table_oral}
	\begin{threeparttable}
		\scriptsize\setlength{\tabcolsep}{1.0pt}
		\begin{tabular}{c|r|rr}
			\hline
			$d=1024$ & \multicolumn{1}{c|}{MOSEK} & \multicolumn{2}{c}{Example \ref{emp_3} Best Scale\tnote{i}} \\ \hline
			$n$ & \multicolumn{1}{c|}{obj.val} & \multicolumn{1}{c}{obj.val} & \multicolumn{1}{c}{obj.impr (\%)} \\ \hline
			100 & 0.4675 & 0.4272 & 8.61 \\
			200 & 0.5303 & 0.5127 & 3.31 \\ 
			300 & 0.5899 & 0.5315 & 9.89 \\ 
			400 & 0.5575 & 0.5456 & 2.13 \\ 
			500 & 0.5512 & 0.5466 & 0.83 \\ 
			600 & 0.5438 & 0.5408 & 0.54 \\
			700 & 0.5625 & 0.5581 & 0.79 \\
			800 & 0.6014 & 0.5593 & 7.01 \\
			900 & 0.5974 & 0.5684 & 4.85 \\
			1000 & 0.5673 & 0.5624 & 0.87 \\ \hline
		\end{tabular}
		\begin{tablenotes}
			\item[i] $N=1$ and in the Best Scale procedure, we first run Example \ref{emp_3} with 3-tuple $(N=2, AS=2^{-2}, TL=30)$. 
		\end{tablenotes}
	\end{threeparttable}
	\vspace{-25pt}
\end{wraptable}

\noindent\textbf{Experiment 6- Large-Scale SLR Cases:} In this experiment, we solve the UCI oral toxicity dataset \citep{ballabio2019integrated} 
with the time limit of $600$ seconds using MOSEK and Example \ref{emp_3} with Best Scale methods since Experiment 5 shows that the Branch and Cut method does not work well for SLR cases. Similarly, in Example \ref{emp_3} with Best Scale, we use the obtained approximate solution to update the $M_{\theta}$ value. This dataset contains $d = 1024$ binary variables and one binary experimental class, i.e., the positive class for very toxic and the negative class for not very toxic. We randomly select $500$ data points from each class to form a new dataset with $1000$ data points and consider $n\in\{100, 200, \ldots, 1000\}$.  Our solution time limit is set to be $600$ seconds since machine learning problems often require finding a good-quality solution within a short amount of time.

Table \ref{slr_large_table_oral} summarizes the results for large-scale SLR cases, where the objective value (denoted as ``obj.val") by substituting the obtained solution into the original SLR objective function in \eqref{eq_num_study3} and the relative gap (denoted as ``obj.impr") between the obj.val of MOSEK and the obj.val of  Example \ref{emp_3} with Best Scale are displayed. With the time limit of $600$ seconds, both methods cannot be solved to optimality; however, Example \ref{emp_3} with Best Scale outputs better objective values than MOSEK. Particularly, Example \ref{emp_3} with Best Scale only needs one SOC constraint to approximate each exponential conic constraint. It is worthy of mentioning that we find the approximate solution for Example \ref{emp_3} with Best Scale by executing Example \ref{emp_3} with the approximate solution equal to $2^{-2}$ and a 30-second time limit. The relative gaps between the obj.vals of MOSEK and Example \ref{emp_3} Best Scale for different cases vary from $0.54\%$ to $9.89\%$. Overall, in this experiment, our best method (i.e., Example \ref{emp_3} Best Scale) consistently outperforms MOSEK on solving large-scale SLR instances by providing better quality solutions.


\section{Conclusion}\label{sec_conclusion}
This paper studies the approximation schemes of the mixed-integer exponential conic programs (MIECPs). We generalize and extend the existing second-order conic approximation scheme and propose new scaling and shifting methods. We also prove approximation accuracies and derive lower bounds of approximation results. We study the polyhedral outer approximation of the exponential cones in the original space based on gradient inequalities. Our numerical study shows that the scaling, shifting, and polyhedral outer approximation methods work very well and can consistently outperform MOSEK with 5-20 times speed-ups. We are working on developing valid inequalities for MIECPs by exploring submodularity and studying disjunctive cuts.

\bibliography{Reference}
\newpage
\titleformat{\section}{\large\bfseries}{\appendixname~\thesection .}{0.5em}{}
\begin{appendices}

	\section{Proofs}\label{sec_proof}
	\subsection{Proof of Proposition \ref{prop_conic_sos}}\label{sec_proof_prop_conic_sos}
	\textbf{Proposition \ref{prop_conic_sos}} \textit{The coefficients in \eqref{eq_fun} can be found recursively as}
	\begin{align}\tag{\ref{eq_coeff}}
		\alpha_{s}=1,\beta_s=1, \sum_{j\in [0,k]}\frac{\alpha_{s-j}\beta_{s-j}^{2k-2j}}{(2k-2j)!}=1,
		\sum_{j\in [0,k]}\frac{\alpha_{s-j}\beta_{s-j}^{2k-2j+1}}{(2k-2j+1)!}=1, \forall k \in [0,s-1].
	\end{align}
	\begin{proof}
		We observe that
		\[(\beta_k+y)^{2k}=\sum_{i\in [0,2k]}{2k \choose i}\beta_{k}^{i}y^{2k-i}, \forall k\in [0,s],\]
		which allows us to find the coefficients for each $y^k, k \in [0, 2s]$ 
		\[ \hat{\psi}_{N,2s}(y)=\sum_{j\in [0,s]}\frac{\alpha_j}{(2j)!}(\beta_j +y)^{2j}.\]
		For $ k \in [0, 2s]$, the coefficient for $y^k$ is 
		\[\begin{cases}
			\sum_{i\in [0,s-k/2]}\frac{\alpha_{s-i}\beta_{s-i}^{2s-2i-k} }{(2s-2i-k)!k!}&\textrm{\;if $k$ is even,\;} \\
			\sum_{i\in [0,s-(k+1)/2]} \frac{\alpha_{s-i}\beta_{s-i}^{2s-2i-k} }{(2s-2i-k)!k!}&\textrm{\;if $k$ is odd.}
		\end{cases}\]
		In the expression of $\hat{\psi}_{N,2s}(y)=\sum_{i\in [0,2s]}y^i/i!$, we also know that the coefficient for $y^k$ is $1/k!$ for each $k \in [0, 2s]$. Since
		\[\hat{\psi}_{N,2s}(y)=\sum_{i\in [0,2s]}\frac{y^i}{i!}=\sum_{j\in [0,s]}\frac{\alpha_j}{(2j)!}(\beta_j +y)^{2j},\]
		for each $k \in [0, 2s]$, we can obtain
		\[\frac{1}{k!}=\begin{cases}
			\sum_{i\in [0,s-k/2]}\frac{\alpha_{s-i}\beta_{s-i}^{2s-2i-k} }{(2s-2i-k)!k!}&\textrm{\;if $k$ is even,\;} \\
			\sum_{i\in [0,s-(k+1)/2]} \frac{\alpha_{s-i}\beta_{s-i}^{2s-2i-k} }{(2s-2i-k)!k!}&\textrm{\;if $k$ is odd.}
		\end{cases} \]
		Solving these equations, we arrive at the conclusion.
		\QEDA
	\end{proof}
	
	\section{Small-Scale Binary Packing MIECP Results}\label{sec_small}
	
	Figure \ref{fig_ex1} shows the Gap and running time for Example \ref{emp_1} with $N\in\{10, 15, 20\}$ and $a\in\{0.90, 0.95, \ldots, 1.10\}$. As shown in Figures \ref{fig_ex1_gap_k10}-\ref{fig_ex1_gap_k20}, for a given $a$, the Gap can be improved in general by increasing $N$. We see that the gap is greater than $10^{-4}$ for all $a$ when $N=10$. When $N=15$, using Example \ref{emp_1} with $a\in\{1.00, 1.05, 1.10\}$ can solve several cases with the gap being no larger than $10^{-4}$. When $N=20$, Example \ref{emp_1} with $a\in\{1.00, 1.05, 1.10\}$ can solve all the cases with the Gap being within $10^{-4}$. We notice that when $N$ increases, the improvement of Gap when $a\in\{0.90, 0.95\}$ is not significant compared to the other $a$ values. In particular, since Gurobi uses $10^{-4}$ as the default Gap, the improvement due to increasing $N$ might not be evident; see $a=1.05, 1.10$ in Figures \ref{fig_ex1_gap_k15} and \ref{fig_ex1_gap_k20}, for instance. As shown in Figures \ref{fig_ex1_time_k10}-\ref{fig_ex1_time_k20}, Gurobi spends a longer time on solving the problem with larger $p$ or $N$ values in general, since a larger $p$ or $N$ value means more SOC constraints. Given the same setting of $p$ and $N$, the running time using Example \ref{emp_1} for different $a$ values is nearly the same. On average, Gurobi spends around $0.5, 1.0, 1.5$ seconds on solving packing MIECP \eqref{eq_num_study_exp_cone} using Example \ref{emp_1} with $N=10, 15, 20$, respectively. Hence, overall, we see that choosing $a=1$ and $N=20$ might be desirable when solving larger instances.
	\begin{figure}[htbp]
		\centering
		\subfloat[0.32\textwidth][Gap for $N=10$]{
			\includegraphics[width=0.32\textwidth]{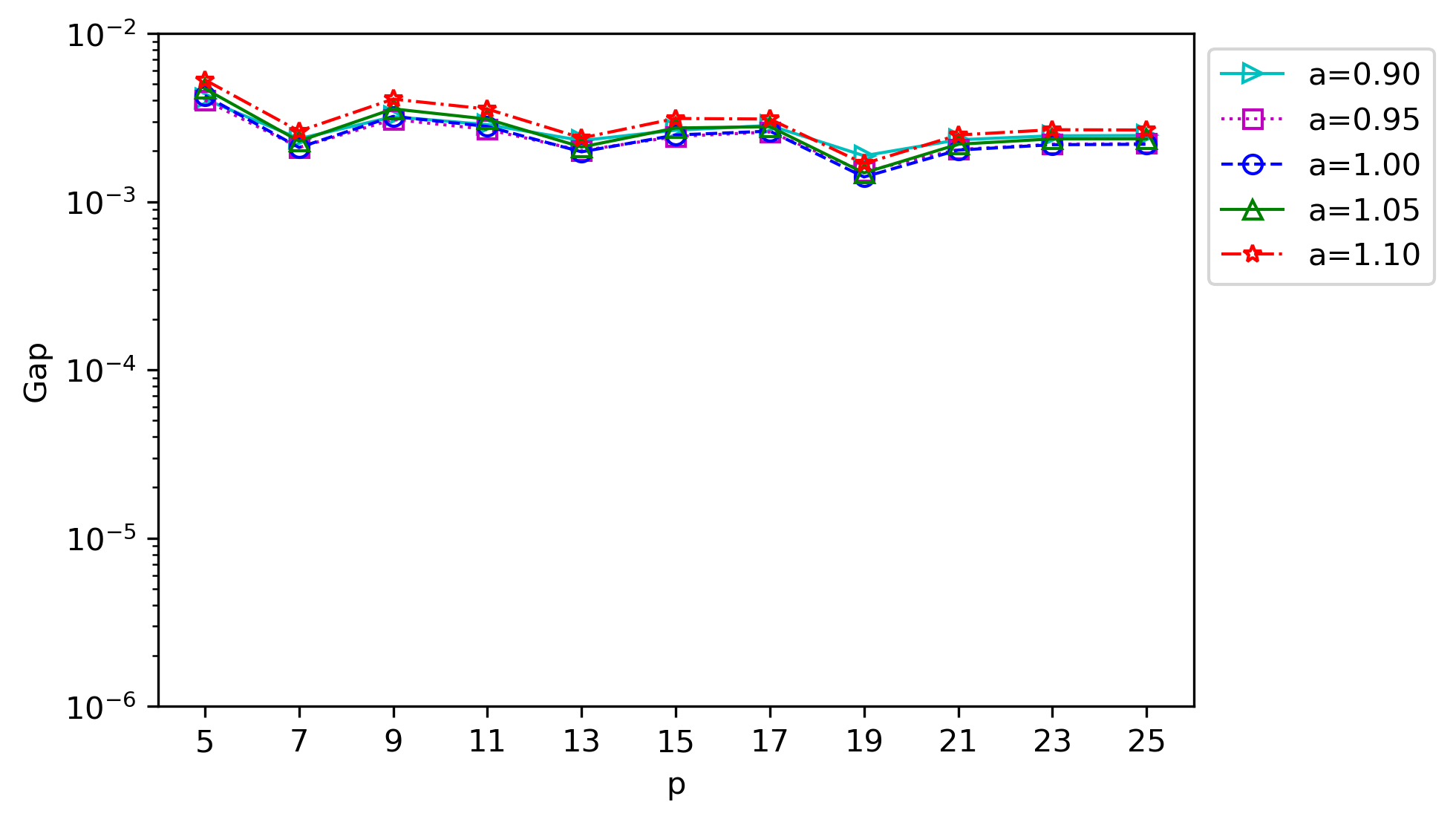} 
			\label{fig_ex1_gap_k10} 
		}\hfill
		\subfloat[0.32\textwidth][Gap for $N=15$]{
			\includegraphics[width=0.32\textwidth]{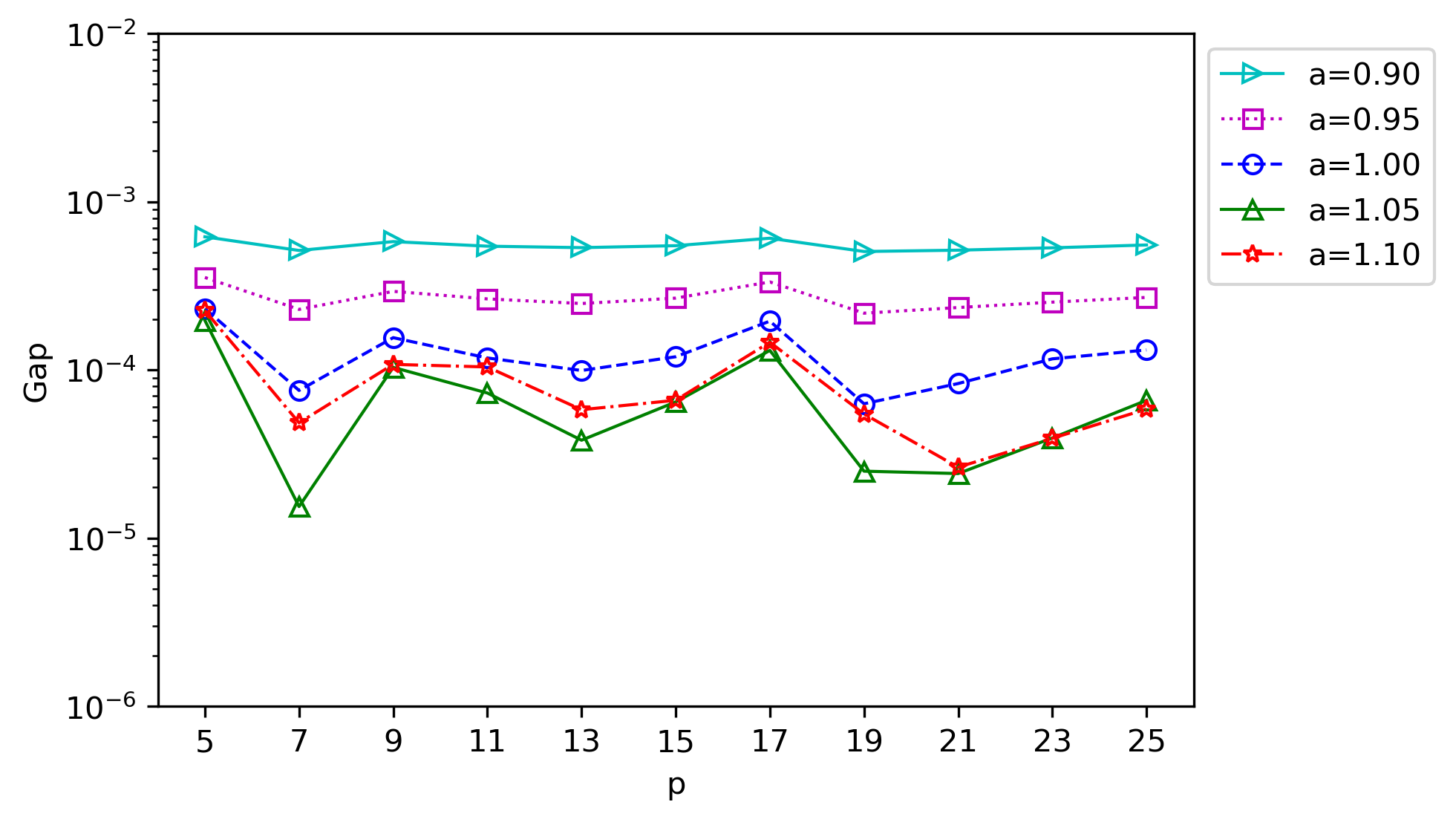} 
			\label{fig_ex1_gap_k15} 
		}\hfill
		\subfloat[0.32\textwidth][Gap for $N=20$]{
			\includegraphics[width=0.32\textwidth]{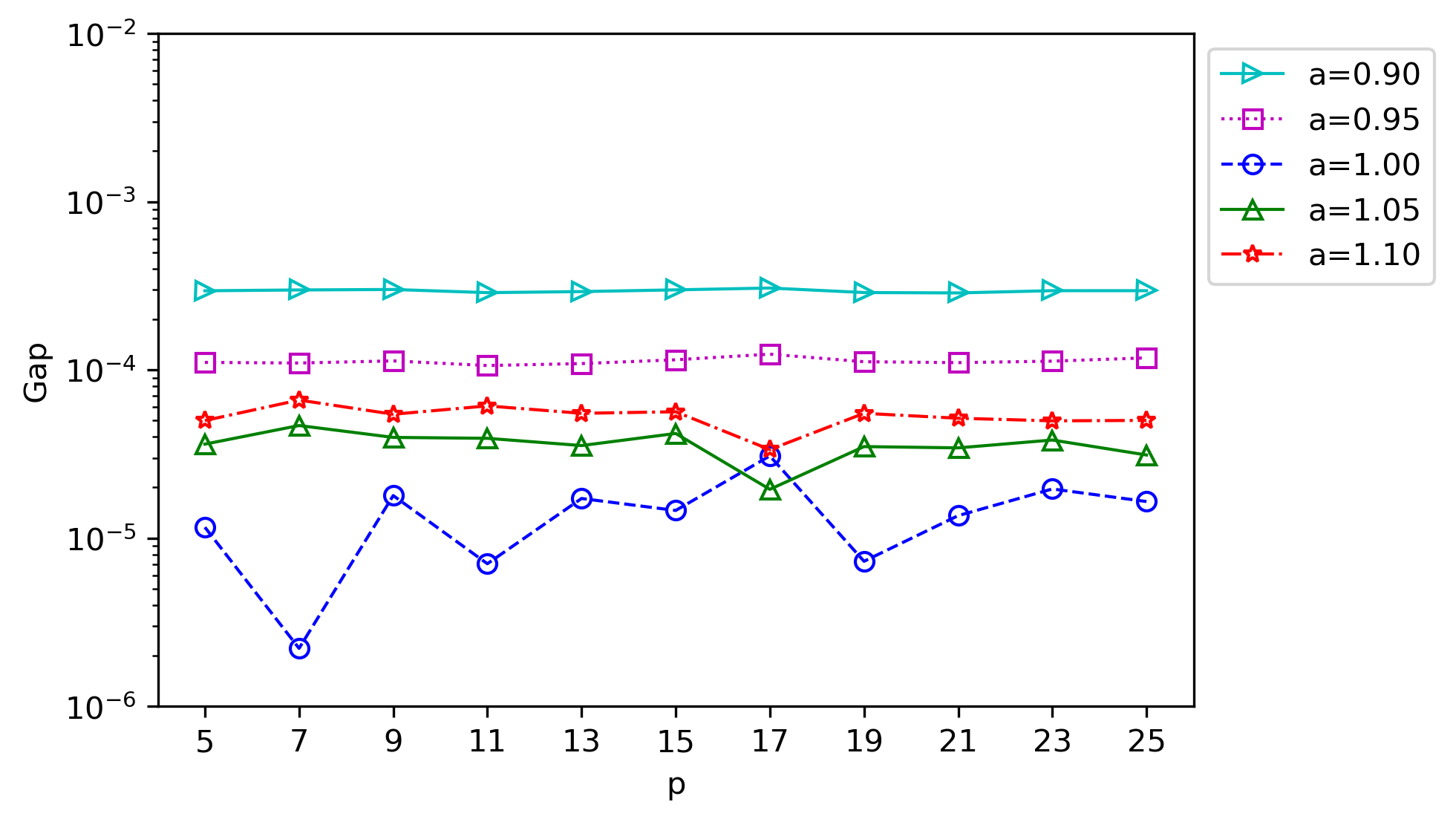} 
			\label{fig_ex1_gap_k20} 
		}\hfill
		\subfloat[0.32\textwidth][Running Time for $N=10$]{
			\includegraphics[width=0.32\textwidth]{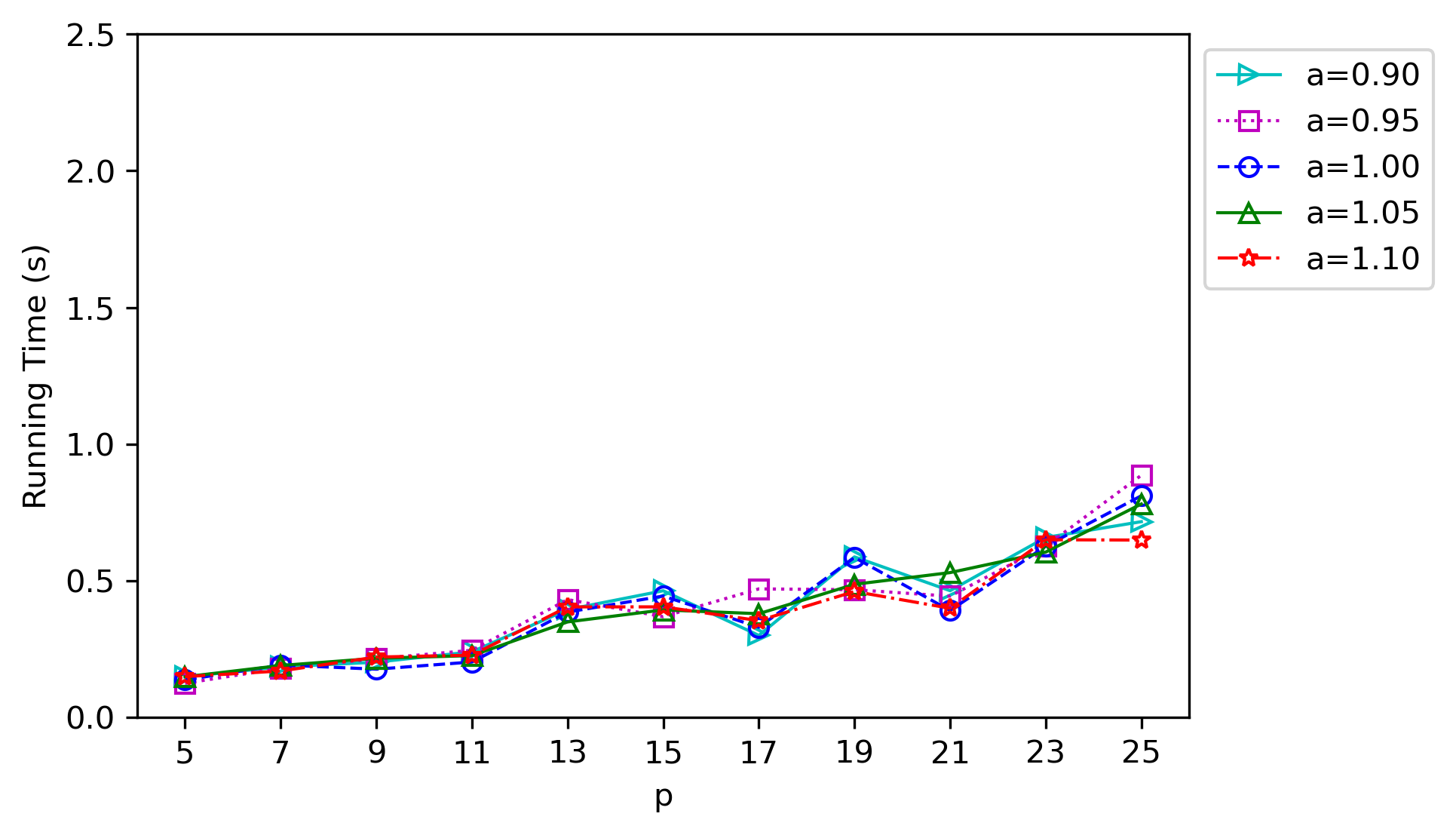} 
			\label{fig_ex1_time_k10} 
		}\hfill
		\subfloat[0.32\textwidth][Running Time for $N=15$]{
			\includegraphics[width=0.32\textwidth]{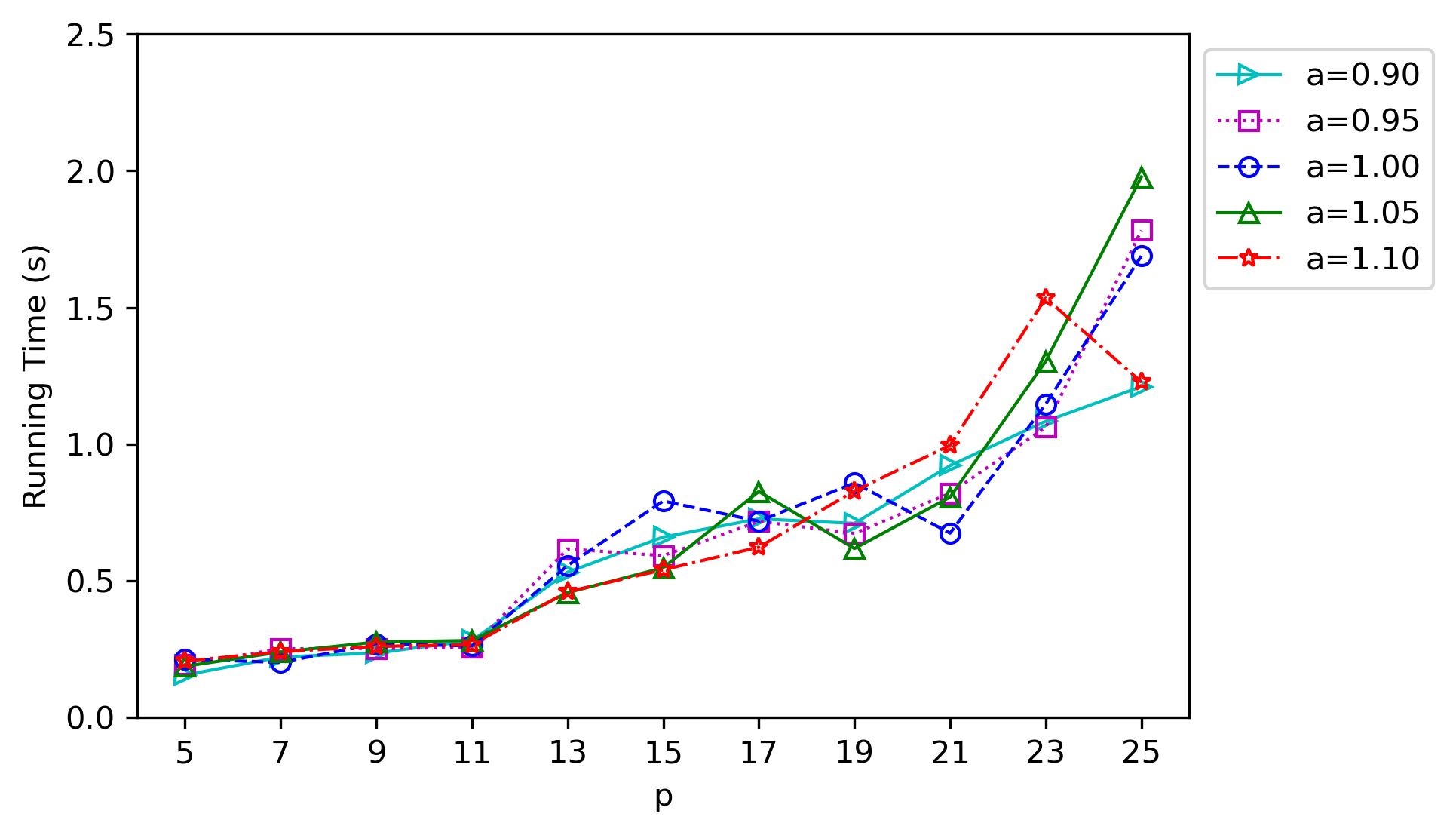} 
			\label{fig_ex1_time_k15} 
		}\hfill
		\subfloat[0.32\textwidth][Running Time for $N=20$]{
			\centering
			\includegraphics[width=0.32\textwidth]{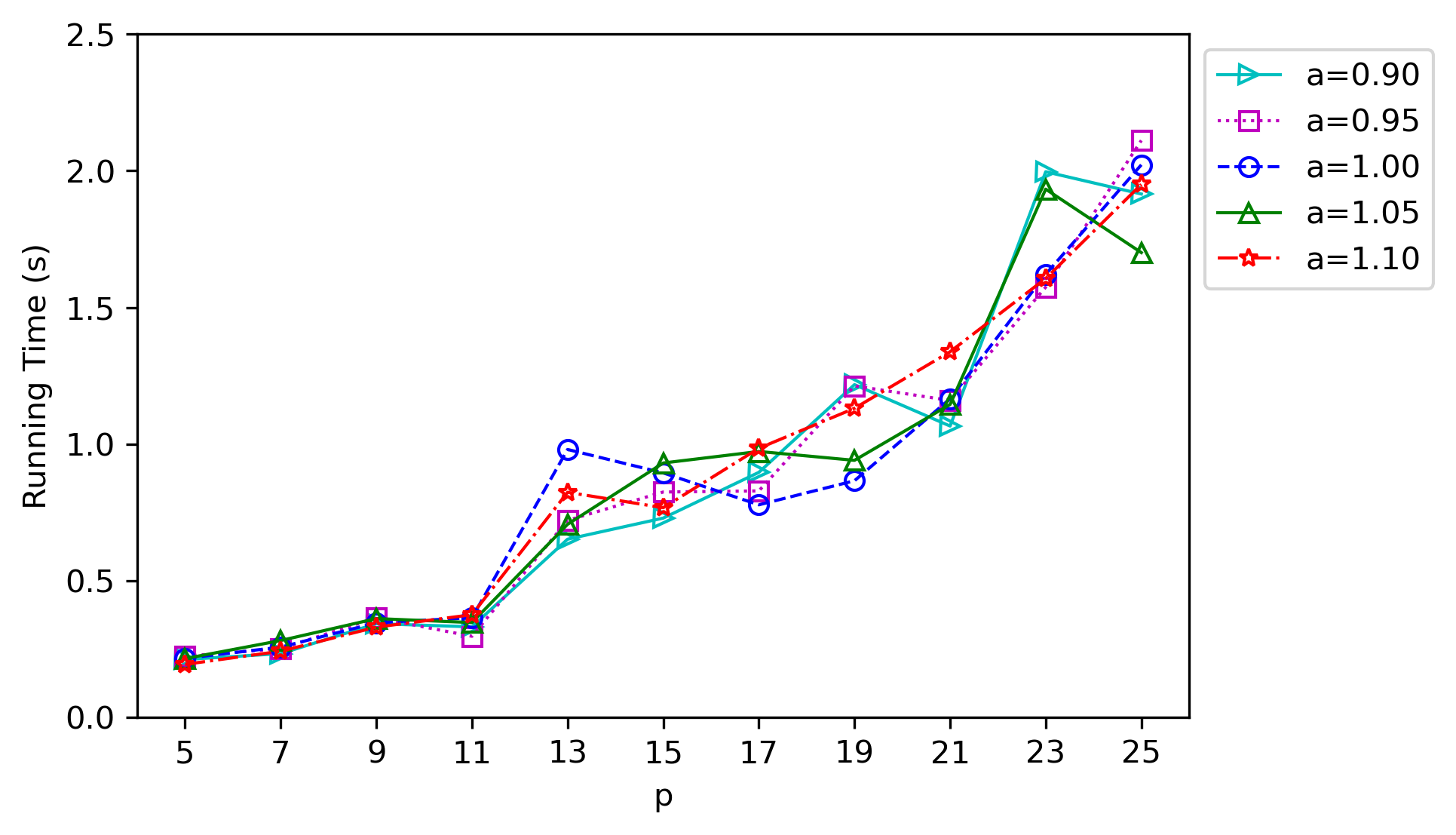} 
			\label{fig_ex1_time_k20}}
		\caption{Numerical Illustration of Example \ref{emp_1}: Small-Scale Binary Packing MIECP}\label{fig_ex1}
	\end{figure}
	
	Figure \ref{fig_ex2} shows the Gap and running time for Example \ref{emp_2} with $N\in\{1, 2, 3\}$ and $s\in\{1, 2, 3, 4, 5\}$. As shown in Figures \ref{fig_ex2_gap_k1}-\ref{fig_ex2_gap_k3}, for each $s$, the Gap can be improved in general by increasing $N$. We see that the gap is greater than $10^{-4}$ for all the $s$ when $N=1$. When $N=2$, Example \ref{emp_2} with $s\in\{4, 5\}$ can solve all the cases with Gap being no larger than $10^{-4}$. When $N=3$, Example \ref{emp_2} with $s\in\{3, 4, 5\}$ can solve all the cases with Gap within $10^{-4}$. With the same $N$, Example \ref{emp_2} using a larger $s$ value can solve the cases with a better Gap. However, due to the default setting of Gurobi, the improvement due to increasing $N$ or $s$ might be unpredictable; see $s=4, 5$ in Figures \ref{fig_ex2_gap_k2}, \ref{fig_ex2_gap_k3} for instance. 
	As shown in Figures \ref{fig_ex2_time_k1}-\ref{fig_ex2_time_k3}, Gurobi spends a longer time solving the instances with larger $p$ or $N$ values in general. For instance, with fixed values $p$ and $N$, using Example \ref{emp_2} with larger $s$ values tends to spend more time on solving the case. The running time of Example \ref{emp_2} with a larger $s$ grows as $p$ increases for the same $N$. In particular, we observe that using Example \ref{emp_2} with both settings $(N, s)=(2, 4)$ and $(N, s)=(3, 3)$ can solve all the cases within $10^{-4}$ Gap and their running time is similar. Overall, Example \ref{emp_2} spends a similar amount of time on solving the cases to a certain level of Gap regardless of the changes of $N$ and $s$. On average, Example \ref{emp_2} spends around $0.2, 0.3, 0.4$ seconds on solving packing MIECP \eqref{eq_num_study_exp_cone} with $N=1, 2, 3$, respectively. Hence, we recommend that choosing $s=4$ might be desirable when solving larger instances.
	\begin{figure}[htbp]
		\centering
		\subfloat[0.32\textwidth][Gap for $N=1$]{
			\includegraphics[width=0.32\textwidth]{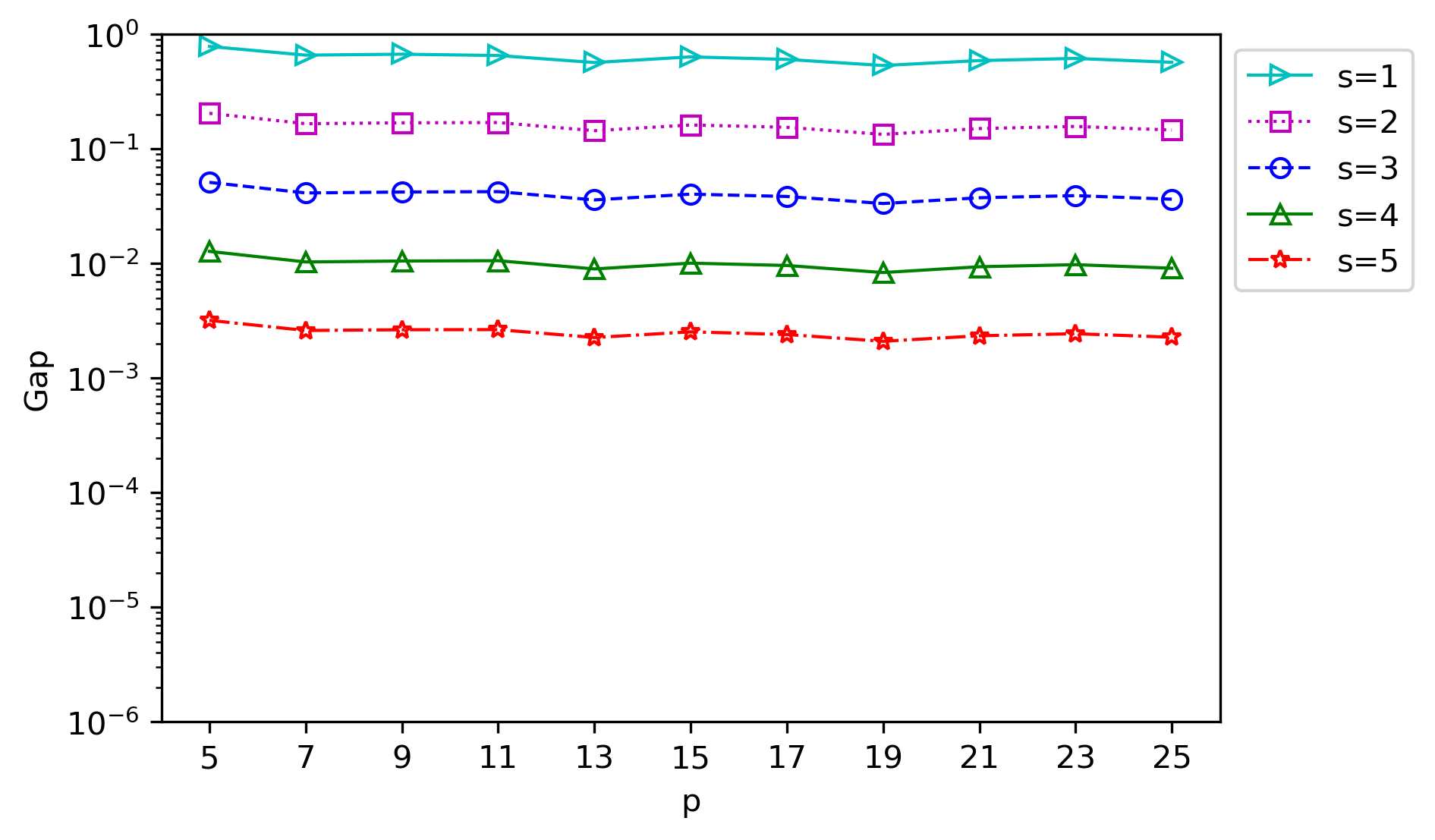} 
			\label{fig_ex2_gap_k1} 
		}\hfill
		\subfloat[0.32\textwidth][Gap for $N=2$]{
			\includegraphics[width=0.32\textwidth]{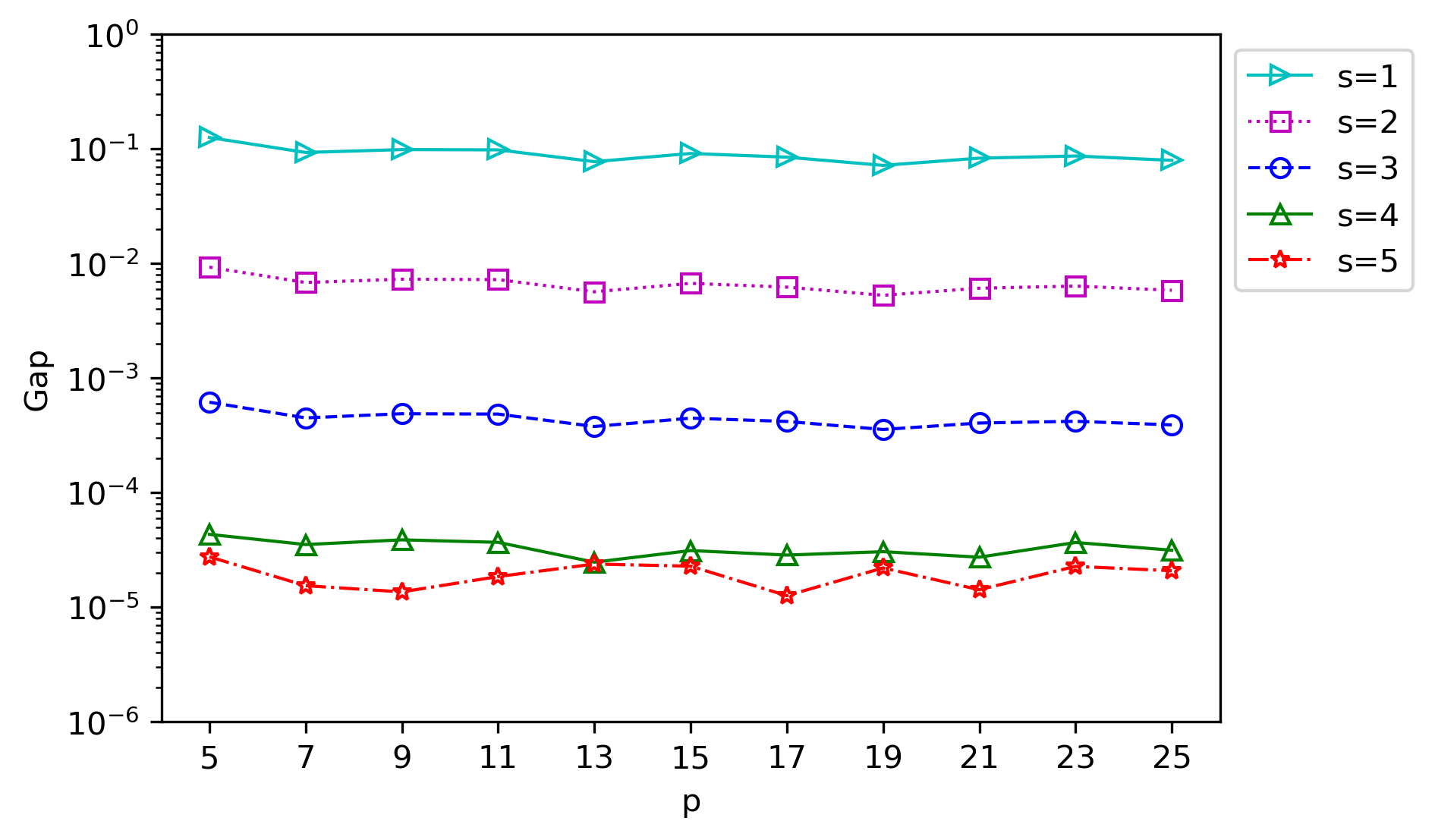} 
			\label{fig_ex2_gap_k2} 
		}\hfill
		\subfloat[0.32\textwidth][Gap for $N=3$]{
			\includegraphics[width=0.32\textwidth]{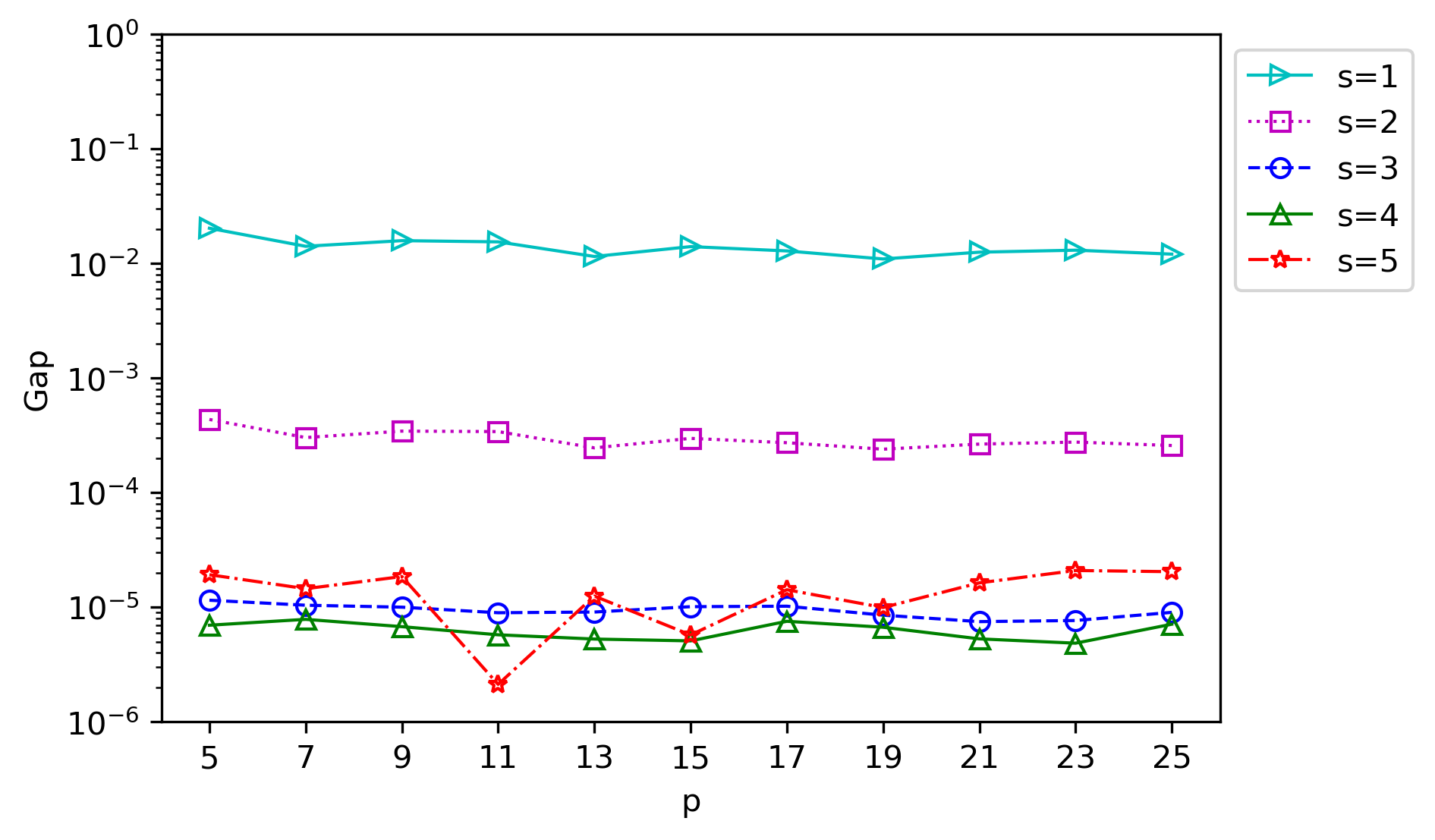} 
			\label{fig_ex2_gap_k3} 
		}\hfill
		\subfloat[0.32\textwidth][Running Time for $N=1$]{
			\includegraphics[width=0.32\textwidth]{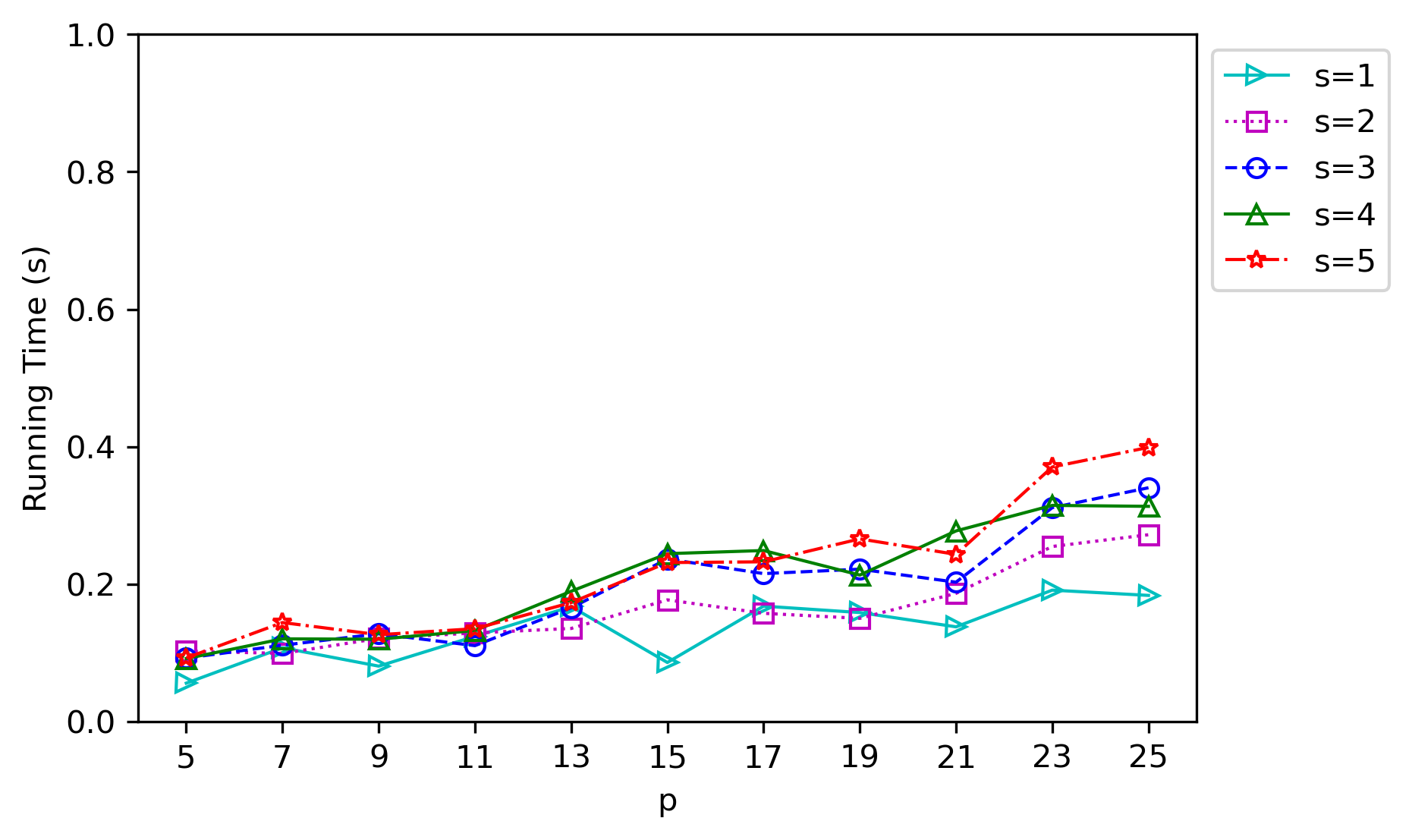} 
			\label{fig_ex2_time_k1} 
		}\hfill
		\subfloat[0.32\textwidth][Running Time for $N=2$]{
			\includegraphics[width=0.32\textwidth]{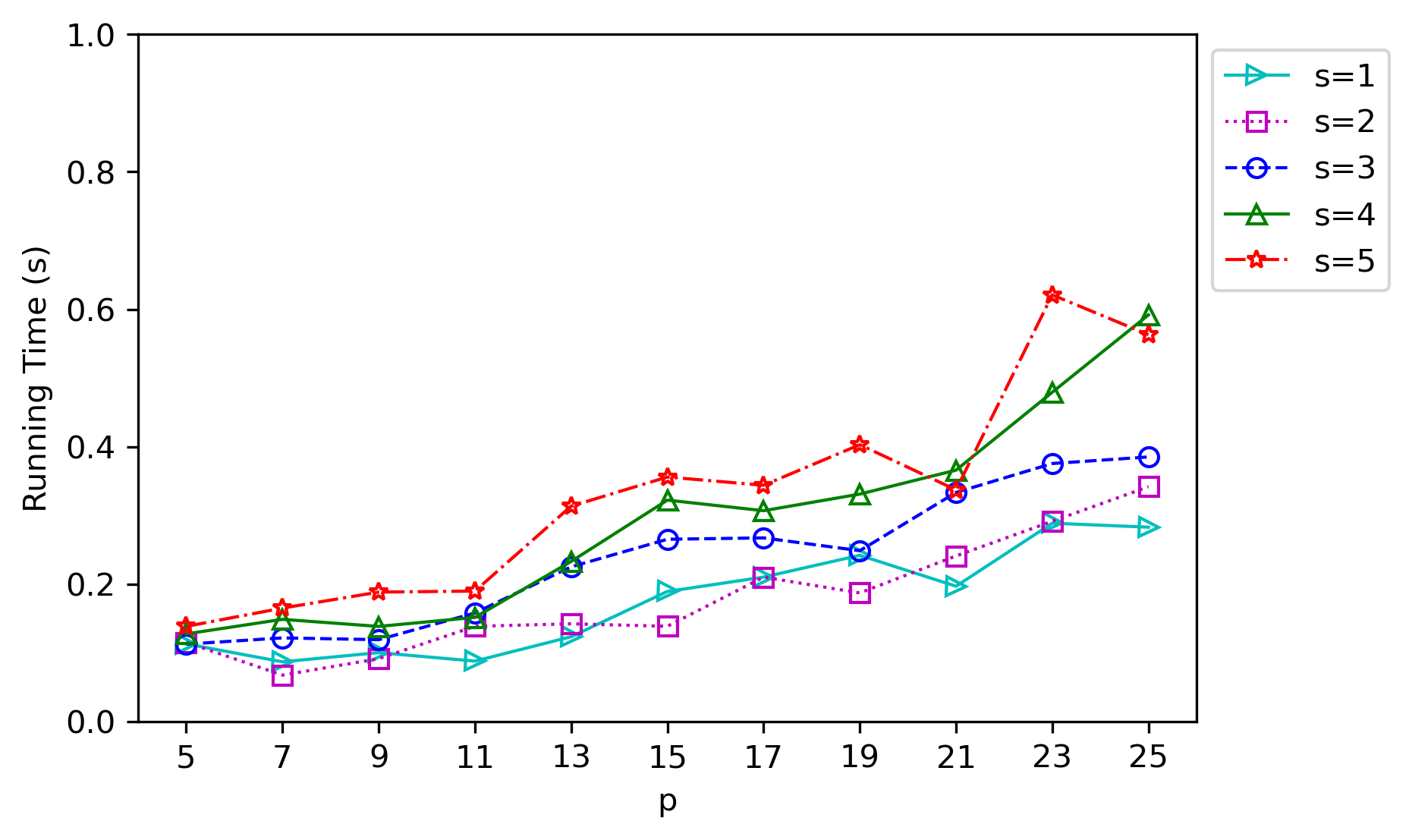} 
			\label{fig_ex2_time_k2} 
		}\hfill
		\subfloat[0.32\textwidth][Running Time for $N=3$]{
			\centering
			\includegraphics[width=0.32\textwidth]{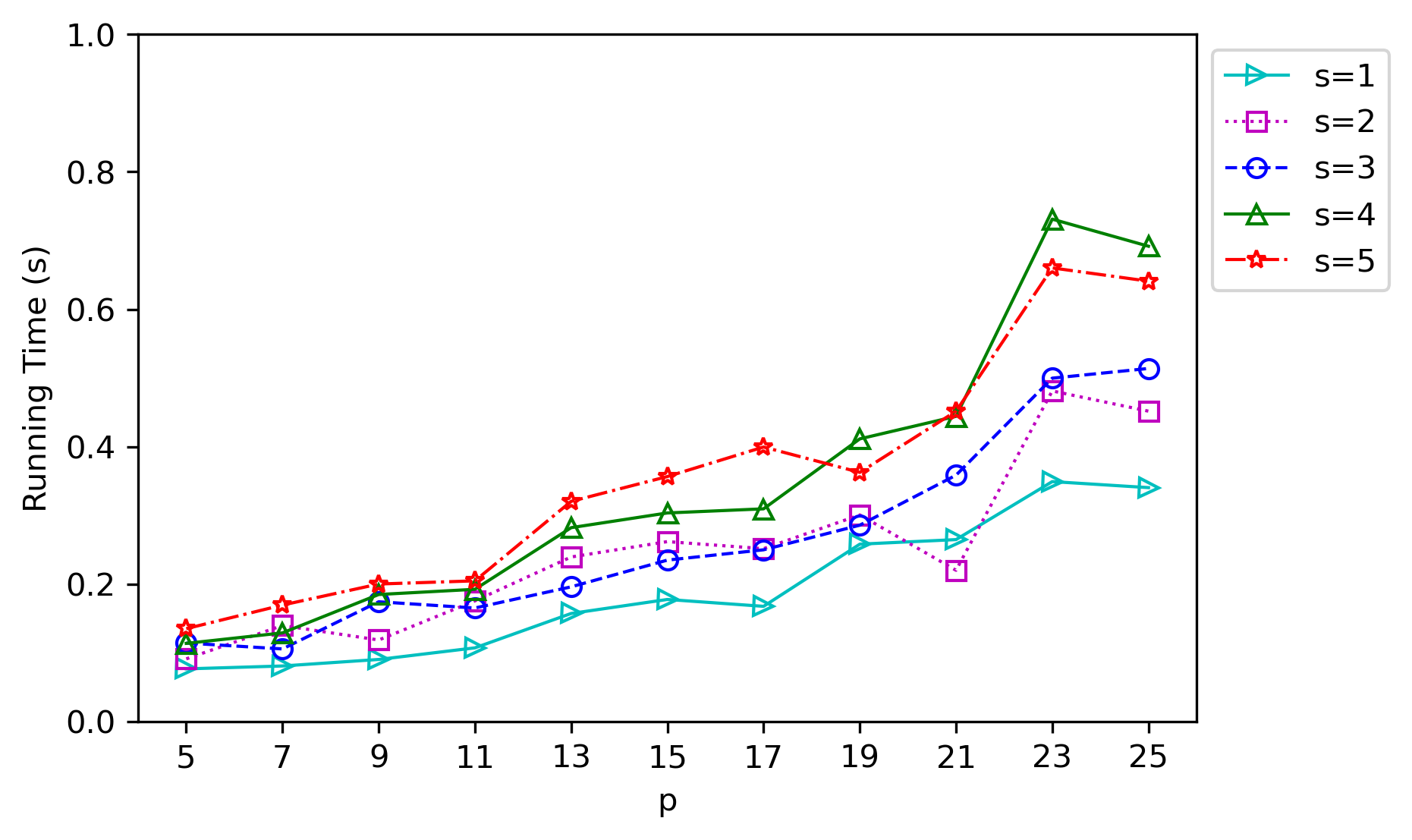} 
			\label{fig_ex2_time_k3}}
		\caption{Numerical Illustration of Example \ref{emp_2}: Small-Scale Binary Packing MIECP}\label{fig_ex2}
	\end{figure}
	
	Figure \ref{fig_ex3} shows the Gap and running time for Example \ref{emp_3} with $N\in\{3, 5, 7\}$ and the approximate solution being selected from the list $\{2^{-6}, 2^{-4},2^{-2}, 2^{0}, \mathrm{Best \;Scale}\}$. Here, ``$\mathrm{Best \;Scale}$" means that we first solve each case using Example \ref{emp_3} with $N=1$ and the approximate solution equal to $2^{-4}$ with a 2-second time limit to find a feasible solution, and then use the found solution to construct the approximate solution for each exponential conic constraint. As shown in Figures \ref{fig_ex3_gap_k3}-\ref{fig_ex3_gap_k7}, for each approximate solution, the Gap can be improved in general by increasing $N$. Only Best Scale can solve all the cases with Gap being no larger than $10^{-4}$ when $N=3$. When $N=5$, Example \ref{emp_3} with the approximate solution being selected from $\{2^{-4}, \mathrm{Best \;Scale}\}$ can solve all the cases with Gap being less than $10^{-4}$. When $N=7$, Example \ref{emp_3} with the approximate solution being selected from $\{2^{-6}, 2^{-4}, \mathrm{Best \;Scale}\}$ can solve all the cases with Gap within $10^{-4}$. We notice that when increasing $N$, the improvement of Gap for the approximate solution being selected from $\{2^{-2}, 2^{0}\}$ are not significant compared to other approximate solutions. Given the same $N$, Example \ref{emp_3} with the approximate solution equal to $2^{-4}$ can solve the case with a better Gap compared to other approximate solutions except for Best Scale. Remarkably, the Best Scale can solve all the cases with Gap being no larger than $10^{-4}$  for $N\in\{3, 5, 7\}$ and its gap is around $10^{-7}$, which outperforms all the other methods based on Example \ref{emp_3} method. In particular, due to the default setting of Gurobi, the improvement due to increasing $N$ might not be obvious; see the cases with the approximate solution equal to $2^{-4}, \mathrm{Best \;Scale}$ in Figures \ref{fig_ex3_gap_k5}, \ref{fig_ex3_gap_k7} for instance. 
	As shown in Figures \ref{fig_ex3_time_k3}-\ref{fig_ex3_time_k7}, Gurobi spends a longer time solving the cases with larger $p$ or $N$ values in general. Given the same setting of $p$ and $N$, the Best Scale method takes a longer time than other methods based on Example \ref{emp_3} since it needs to solve an additional model to generate a feasible solution. The running time for other methods is similar. 
	However, the Best Scale method still works the best since it requires a much smaller number of SOC constraints to achieve the desired accuracy. Therefore, we suggest using $N=1$ or $2$ for the best case and slightly larger $N$ if we predetermine the approximate solution to be equal to $2^{-4}$. 
	\begin{figure}[htbp]
		\centering
		\subfloat[0.32\textwidth][Gap for $N=3$]{
			\includegraphics[width=0.32\textwidth]{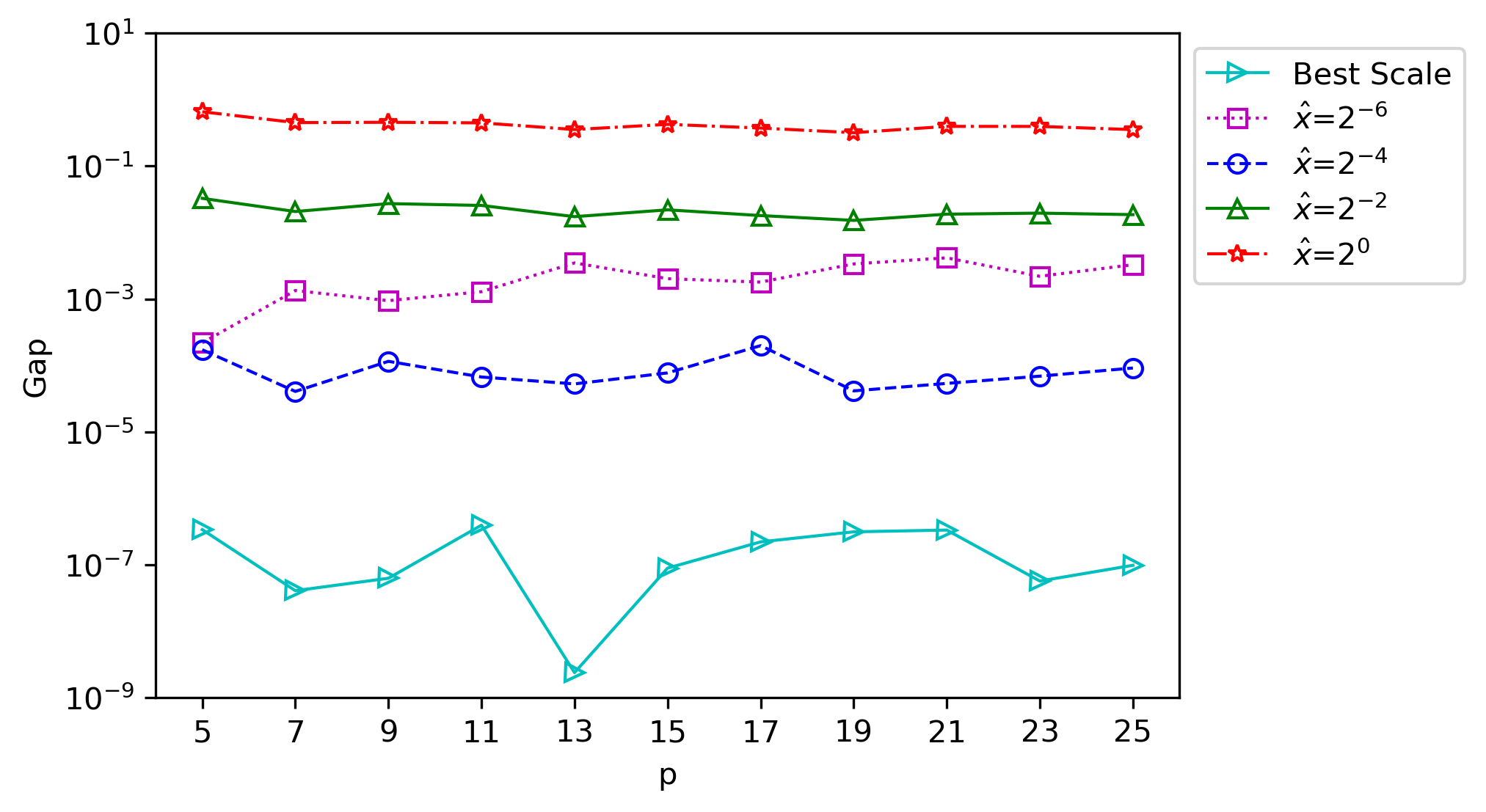} 
			\label{fig_ex3_gap_k3} 
		}\hfill
		\subfloat[0.32\textwidth][Gap for $N=5$]{
			\includegraphics[width=0.32\textwidth]{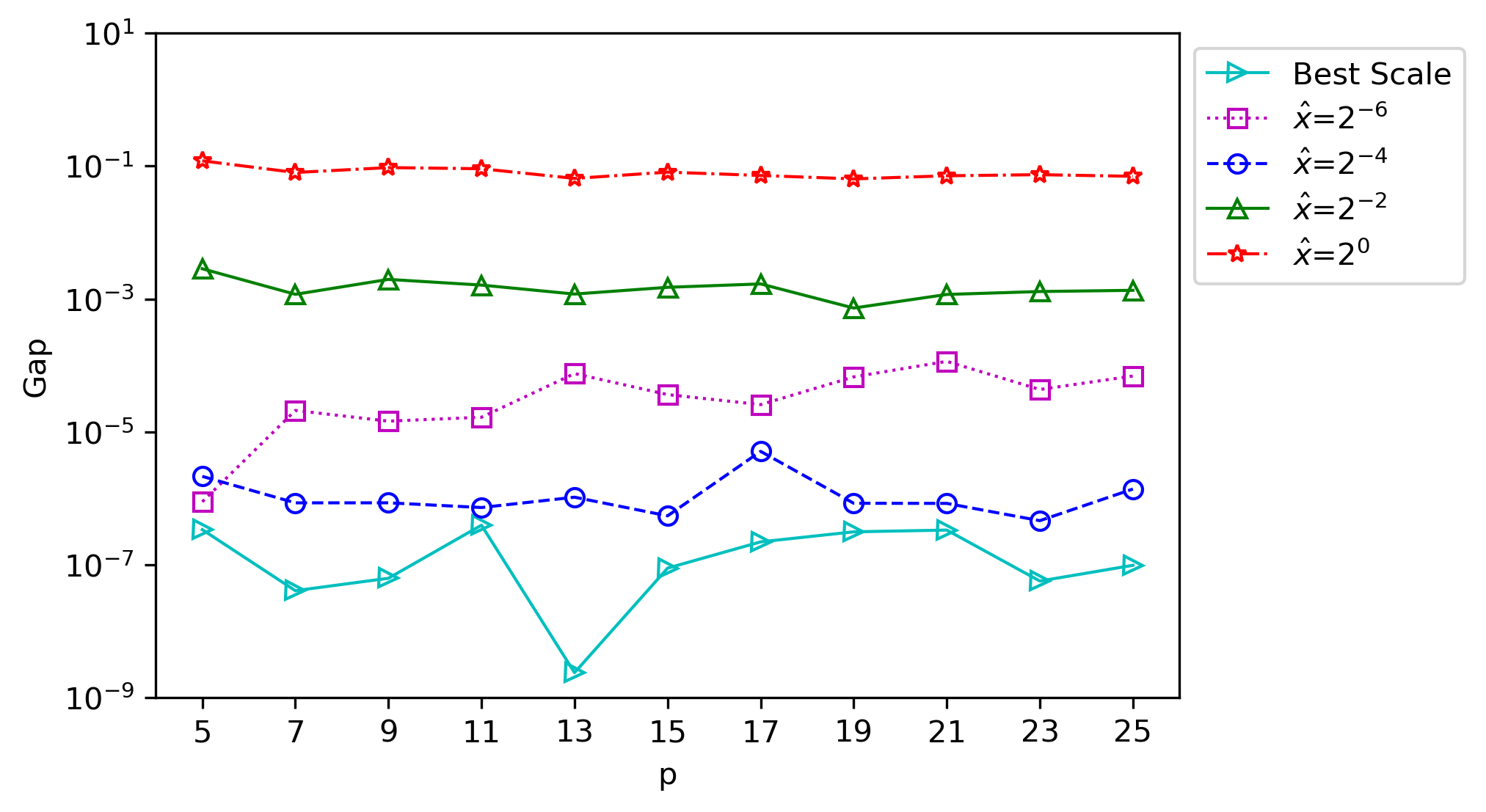} 
			\label{fig_ex3_gap_k5} 
		}\hfill
		\subfloat[0.32\textwidth][Gap for $N=7$]{
			\includegraphics[width=0.32\textwidth]{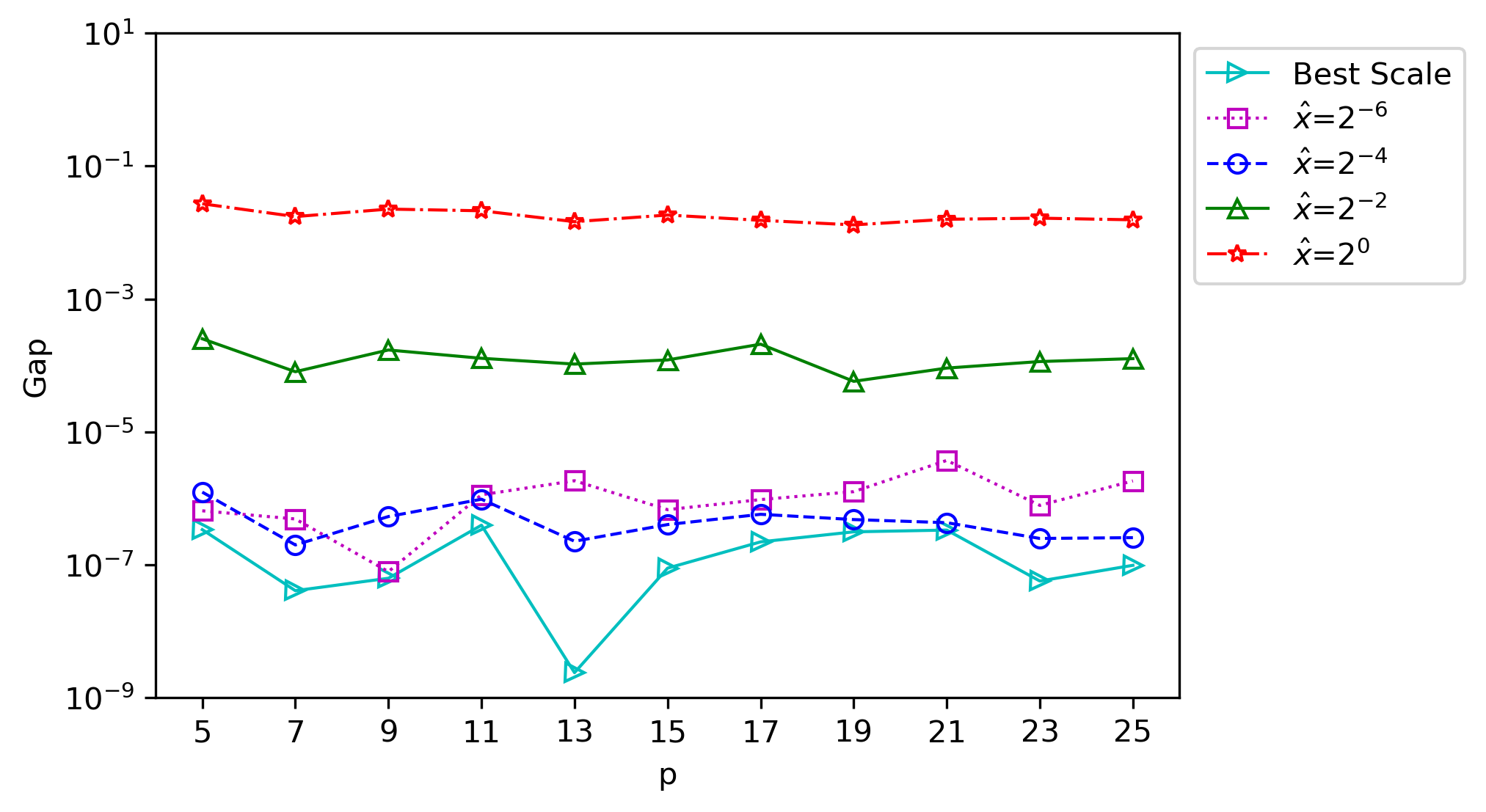} 
			\label{fig_ex3_gap_k7} 
		}\hfill
		\subfloat[0.32\textwidth][Running Time for $N=3$]{
			\includegraphics[width=0.32\textwidth]{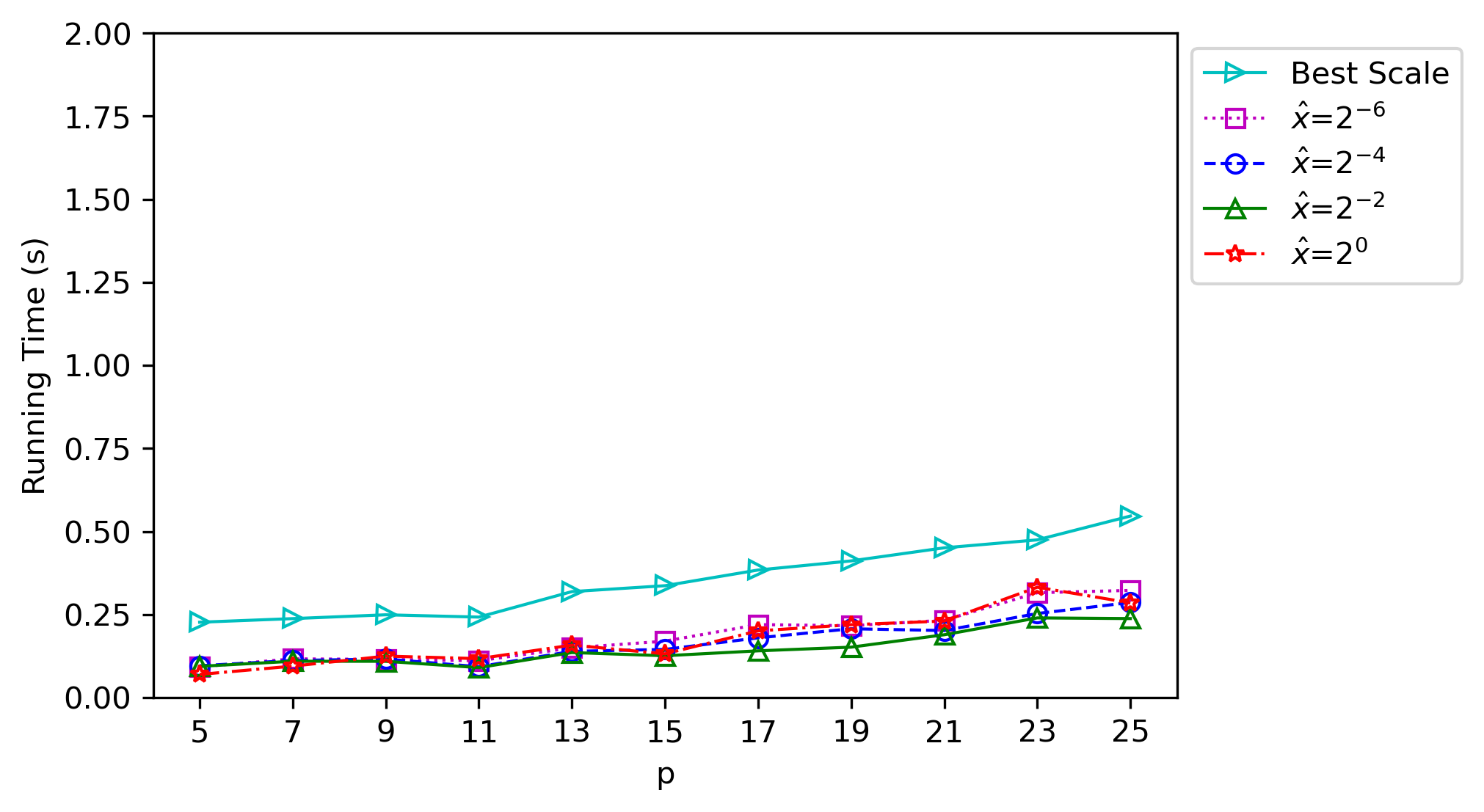} 
			\label{fig_ex3_time_k3} 
		}\hfill
		\subfloat[0.32\textwidth][Running Time for $N=5$]{
			\includegraphics[width=0.32\textwidth]{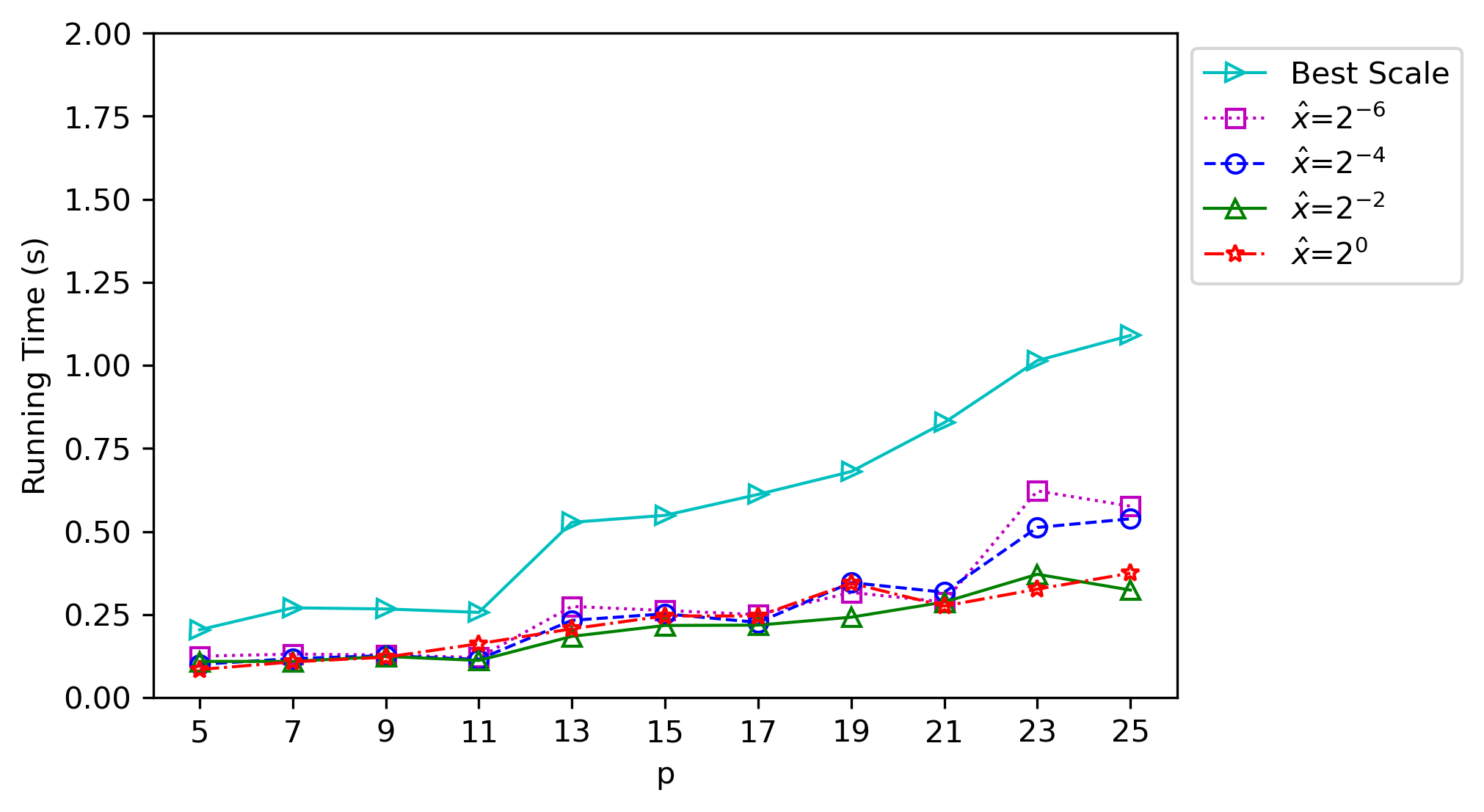} 
			\label{fig_ex3_time_k5} 
		}\hfill
		\subfloat[0.32\textwidth][Running Time for $N=7$]{
			\centering
			\includegraphics[width=0.32\textwidth]{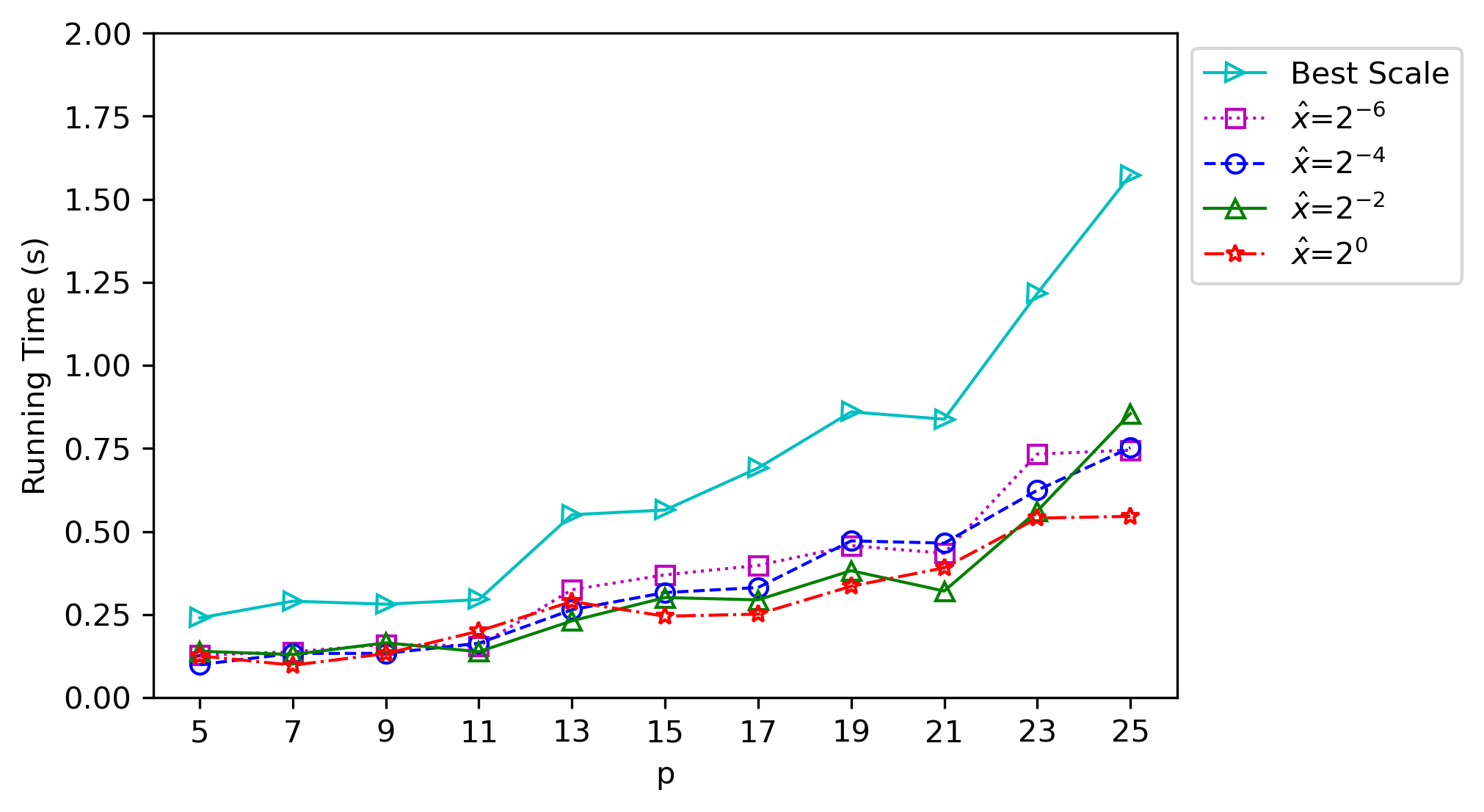} 
			\label{fig_ex3_time_k7}}
		\caption{Numerical Illustration of Example \ref{emp_3}: Small-Scale Binary Packing MIECP}\label{fig_ex3}
	\end{figure}
	
	Figure \ref{fig_sec3} illustrates the Gap and running time for SOC approximations in Section \ref{sec_exp} with $N\in\{1, 3, 5\}$. We use ``Section \ref{sec_exp_lim} Shift" and ``Section \ref{sec_exp_taylor} Shift" to denote the shifting methods based on the results in Section \ref{sec_exp_lim}  and Section \ref{sec_exp_taylor}, respectively. To find a proper approximation solution, we first solve each case using Example \ref{emp_3} with $N=1$ and the approximate solution $2^{-4}$ with 2-second time limit to find a feasible solution, and then use this solution to construct the approximate term for each exponential conic constraint. As shown in Figures \ref{fig_sec3_gap_N1}-\ref{fig_sec3_gap_N5}, the Gap for each method can be improved in general by increasing $N$, i.e., increasing the number of SOC constraints. However, only can the Section \ref{sec_exp_taylor} Shift method solve all the cases with Gap being less than $10^{-4}$ when $N=1$ and its Gap is around $10^{-7}$. When $N\in\{3, 5\}$, both Section \ref{sec_exp_lim} Shift and Section \ref{sec_exp_taylor} Shift methods can solve all the cases with Gap being less than $10^{-4}$. Both Section \ref{sec_exp_lim} and Section \ref{sec_exp_taylor} methods have the Gap being greater than $10^{-4}$ for all $N\in\{1, 3, 5\}$, which is probably due to the numerical issues. We see that the shifting method can remarkably improve the Section \ref{sec_exp_lim} and Section \ref{sec_exp_taylor} ones. We also observe that when increasing $N$, the Gap for the Section \ref{sec_exp_taylor} method decreases more compared to the Section \ref{sec_exp_lim} method. The Section \ref{sec_exp_taylor} method can solve the cases with a better Gap than that of the Section \ref{sec_exp_lim} method when $N\in\{3, 5\}$. 
	As shown in Figures \ref{fig_sec3_time_N1}-\ref{fig_sec3_time_N5}, Gurobi spends a longer time solving the cases with larger $p$ or $N$ values in general. Overall, Section \ref{sec_exp_taylor} method spends a longer time on solving the cases than Section \ref{sec_exp_lim} method, and Section \ref{sec_exp_taylor} Shift method spends a longer time than Section \ref{sec_exp_lim} Shift method. Although the shifting method increases the running time, we still recommend this method since it can solve the cases with the desirable Gap using fewer points than those without shifting. Besides, no-shifting approaches (i.e., Section \ref{sec_exp_lim} and Section \ref{sec_exp_taylor}) may run into numerical issues.
	On average, the Section \ref{sec_exp_lim} and Section \ref{sec_exp_taylor} methods spend around $0.1, 0.2, 0.3$ seconds, and the Section \ref{sec_exp_lim} Shift and Section \ref{sec_exp_taylor} Shift methods spend around $0.3, 0.4, 0.5$ seconds on solving problem \eqref{eq_num_study_exp_cone} with $N=1, 3, 5$, respectively. Overall, we may only need $N=1$ or $2$ for Section \ref{sec_exp_lim} Shift and Section \ref{sec_exp_taylor} Shift methods when solving larger instances.
	\begin{figure}[htbp]
		\centering
		\subfloat[0.32\textwidth][Gap for $N=1$]{
			\includegraphics[width=0.32\textwidth]{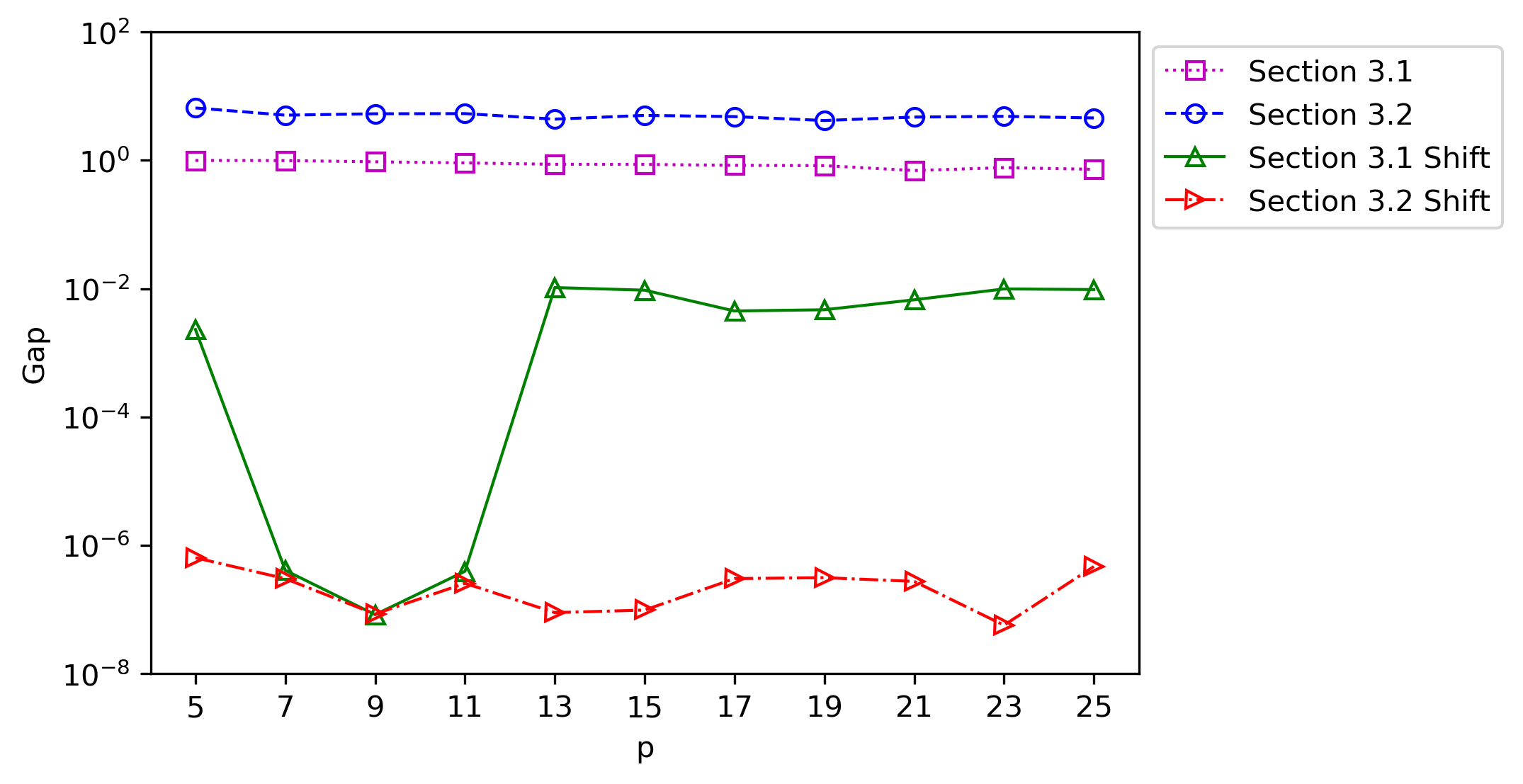} 
			\label{fig_sec3_gap_N1} 
		}\hfill
		\subfloat[0.32\textwidth][Gap for $N=3$]{
			\includegraphics[width=0.32\textwidth]{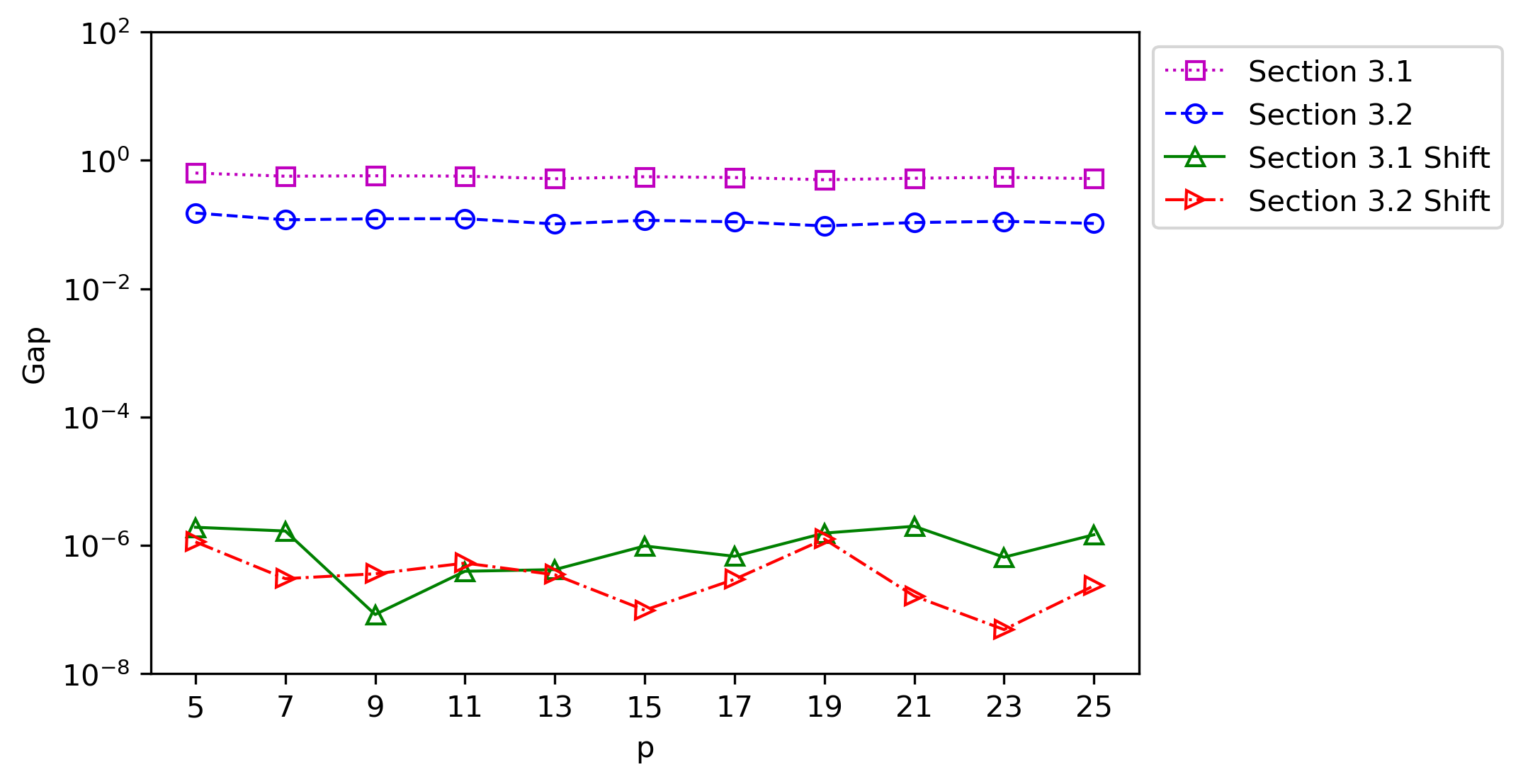} 
			\label{fig_sec3_gap_N3} 
		}\hfill
		\subfloat[0.32\textwidth][Gap for $N=5$]{
			\includegraphics[width=0.32\textwidth]{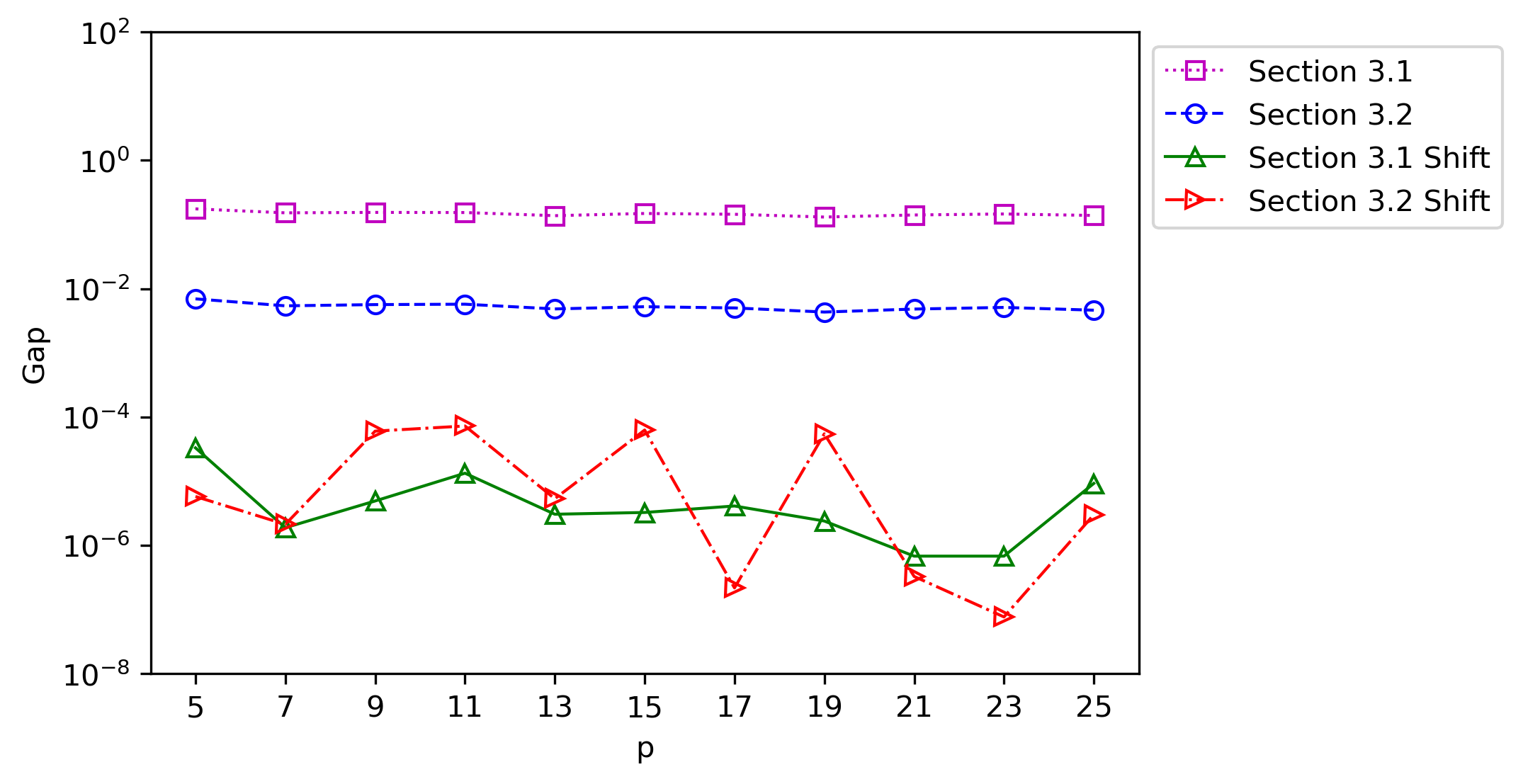} 
			\label{fig_sec3_gap_N5} 
		}\hfill
		\subfloat[0.32\textwidth][Running Time for $N=1$]{
			\includegraphics[width=0.32\textwidth]{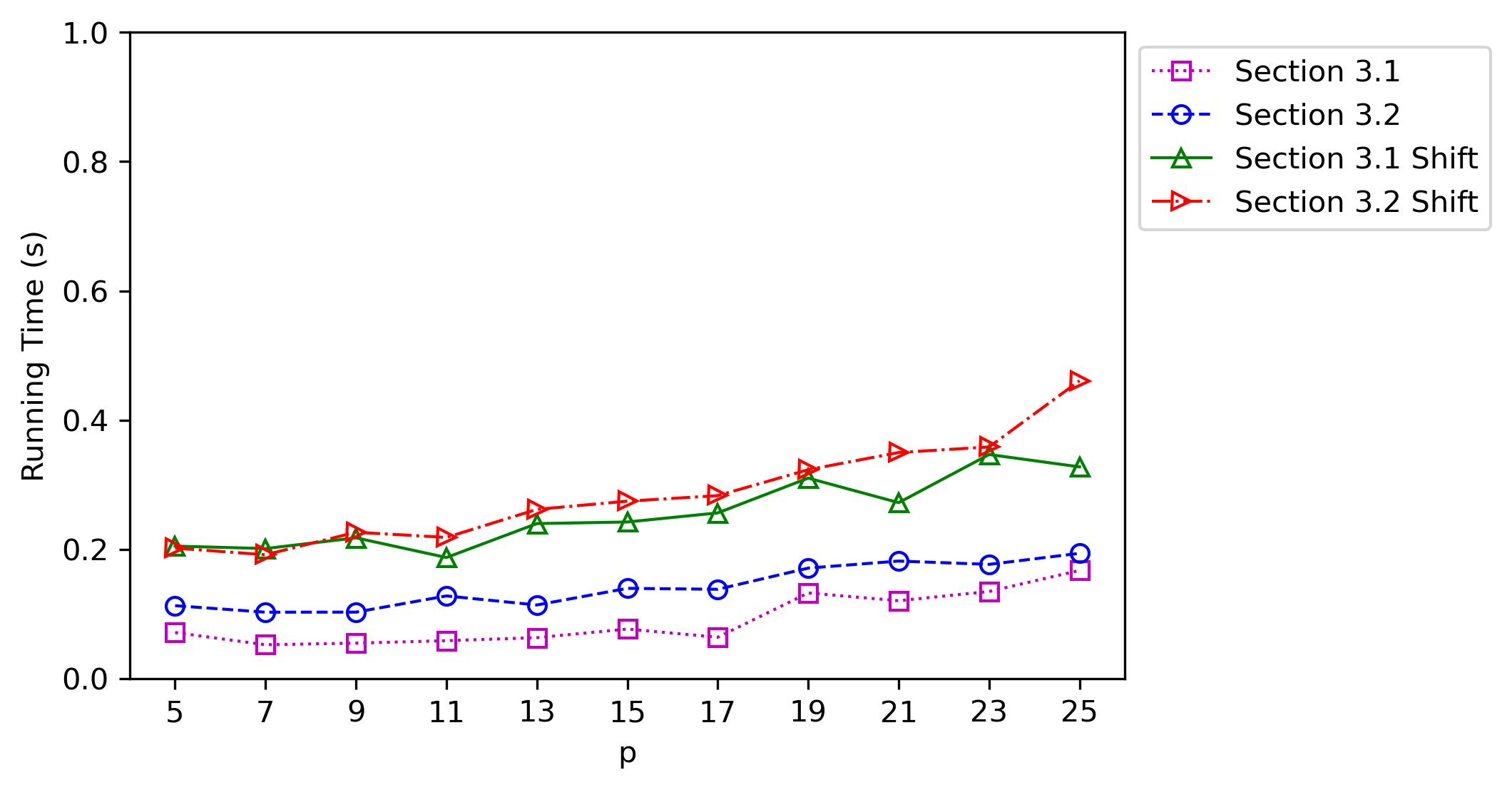} 
			\label{fig_sec3_time_N1} 
		}\hfill
		\subfloat[0.32\textwidth][Running Time for $N=3$]{
			\includegraphics[width=0.32\textwidth]{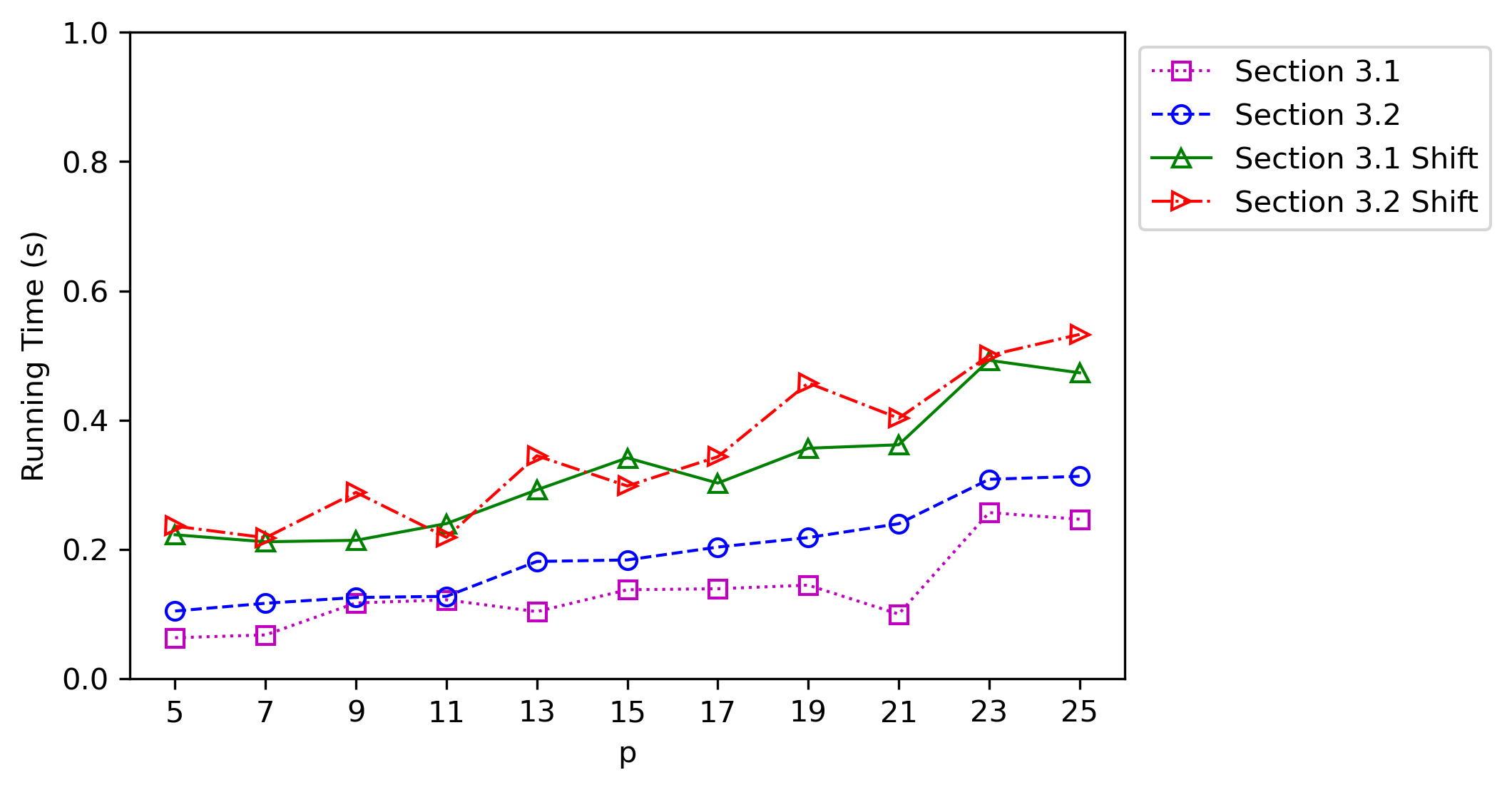} 
			\label{fig_sec3_time_N3} 
		}\hfill
		\subfloat[0.32\textwidth][Running Time for $N=5$]{
			\centering
			\includegraphics[width=0.32\textwidth]{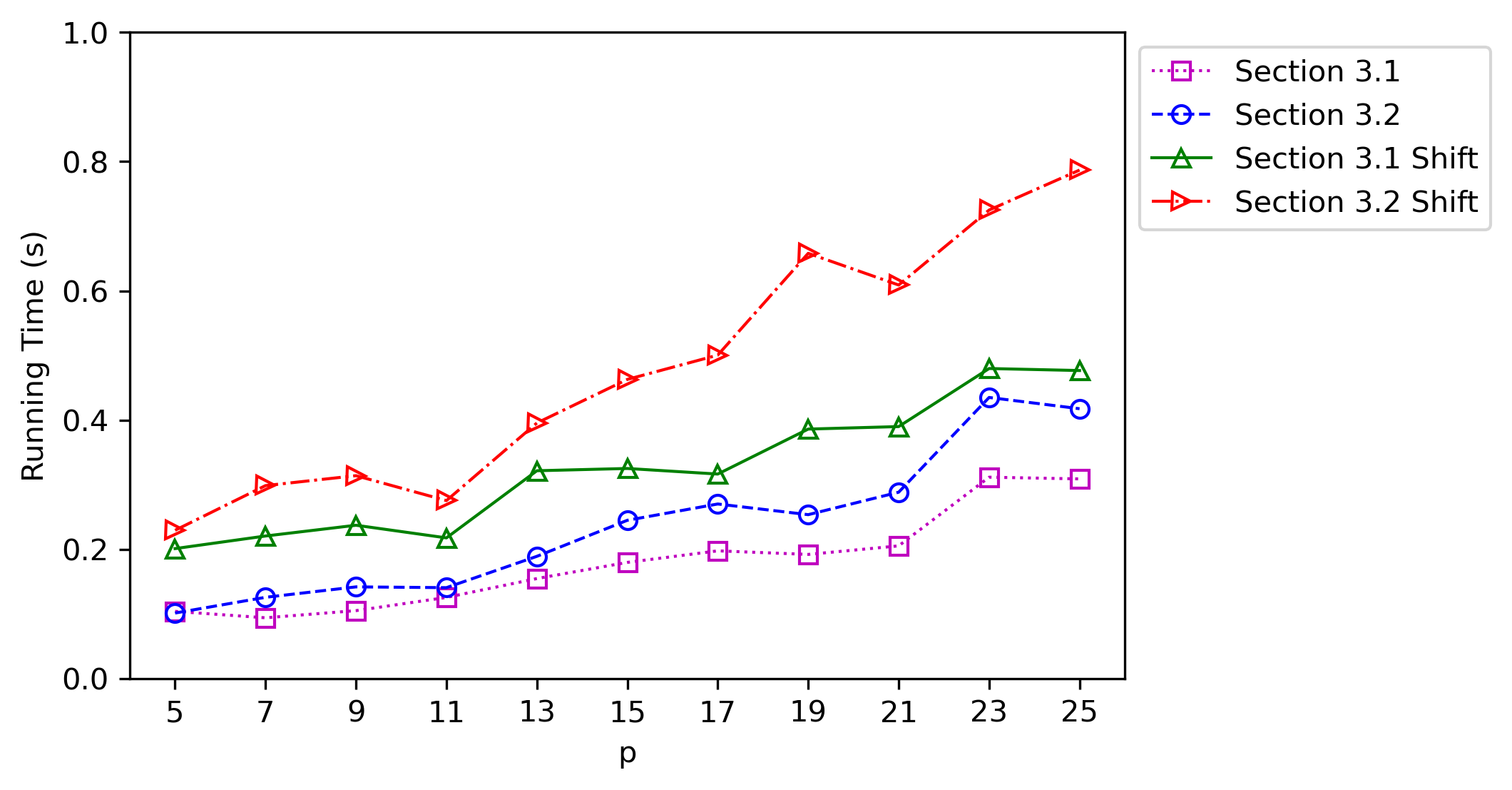} 
			\label{fig_sec3_time_N5}}
		\caption{Numerical Illustration of Section \ref{sec_exp} Methods: Small-Scale Binary Packing MIECP}\label{fig_sec3}
		\end{figure}

		\section{Continuous Packing ECP Results}\label{sec_continuous_packing}
		
		In this experiment, we test the proposed methods by solving the continuous packing ECP \eqref{eq_num_study_exp_cone}, where we consider small-scale and large-scale cases and compare our results with MOSEK. In our testing instances, we let $t=0$ and $a_{ij}\sim \mathrm{int}(0, 9)$, $b_i=4n$, $c_{\ell j}\sim -\mathrm{int}(0, 9)/n$ for all $i \in [m], j\in [n], \ell\in [p]$, where $\mathrm{int}(p, q)$ denotes a random integer between $p$ and $q$ including $p$ and $q$. In each case, we compute the relative optimality gap of MOSEK, denoted by ``Gap," which is defined as the absolute difference of the ratio of the best objective value from MOSEK over the exact value of the ECP by plugging in the optimal solution found by MOSEK and one. We align our accuracy same as the Gap of MOSEK. 
		
		In the small-scale experiment, we consider $n\in\{100, 200, \ldots, 1000\}$, $m=100$, and we solve 10 cases with $p\in\{10, 20, \ldots, 100\}$ for each $n$ and run Example \ref{emp_3} with $N=3$ and the approximate solution $2^{-6}$. By computing the Gap of MOSEK, we set the accuracy requirement as $10^{-4}$. Figure \ref{fig_conti_packing_small} illustrates the average running time for the small-scale continuous packing ECP for each $n$. 
		Both Example \ref{emp_3} and Cutting Plane methods outperform MOSEK for all the cases, and Example \ref{emp_3} spends a shorter time to solve the cases with larger $n$ compared to Cutting Plane. It is seen that the difference between the average running time of MOSEK and Example \ref{emp_3} increases as $n$ increases. 

		In the large-scale experiment, we consider $n\in\{200, 400, \ldots, 2000\}$, $m=n$, and we solve 10 cases with $p\in \{200, 400, \ldots, 2000\}$ for each $n$ and run Example \ref{emp_3} with $N=5$ and the approximate solution $2^{-4}$. By computing the Gap of MOSEK, we set the accuracy requirement as $10^{-5}$. Figure \ref{fig_conti_packing_large} illustrates the average running time for the large-scale continuous packing ECP for each $n$. 					 	
		We see that MOSEK is the best among all the methods in this experiment. Meanwhile, Example \ref{emp_3} and Cutting Plane methods work quite well, and these two methods spend a similar amount of time. Compared to MOSEK, the proposed approximation methods take a slightly longer time to solve the continuous problem with larger $n$. The differences among these three methods are relatively small.
		
		For continuous ECPs, MOSEK performs well, and the proposed methods are comparable to MOSEK. For small-scale problems, Example \ref{emp_3} and Cutting Plane are faster. MOSEK dominates when the problem size is large. While this paper focuses on MIECPs, we highlight that the proposed Cutting Plane method could be considered as an alternative approach for solving small-sized ECPs without parameter tuning.  
		
		\begin{figure}[htbp]
			\centering
			\subfloat[0.48\textwidth][Running Time for Small-Scale Continuous Packing ECP]{
				\centering
				\includegraphics[width=0.45\textwidth]{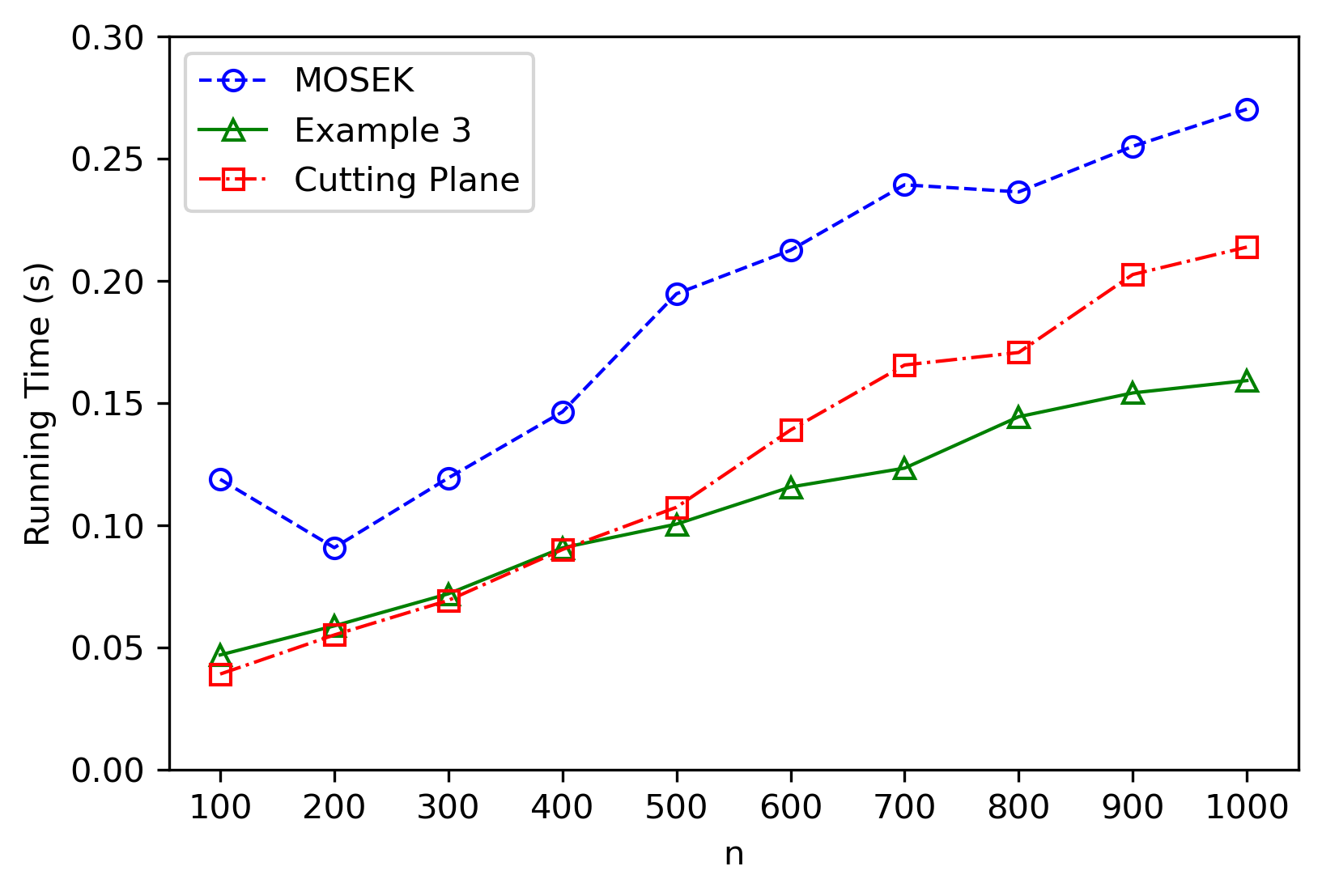}
				\label{fig_conti_packing_small}
			}\hfill
			\subfloat[0.48\textwidth][Running Time for Large-Scale Continuous Packing ECP]{
				\centering
				\includegraphics[width=0.44\textwidth]{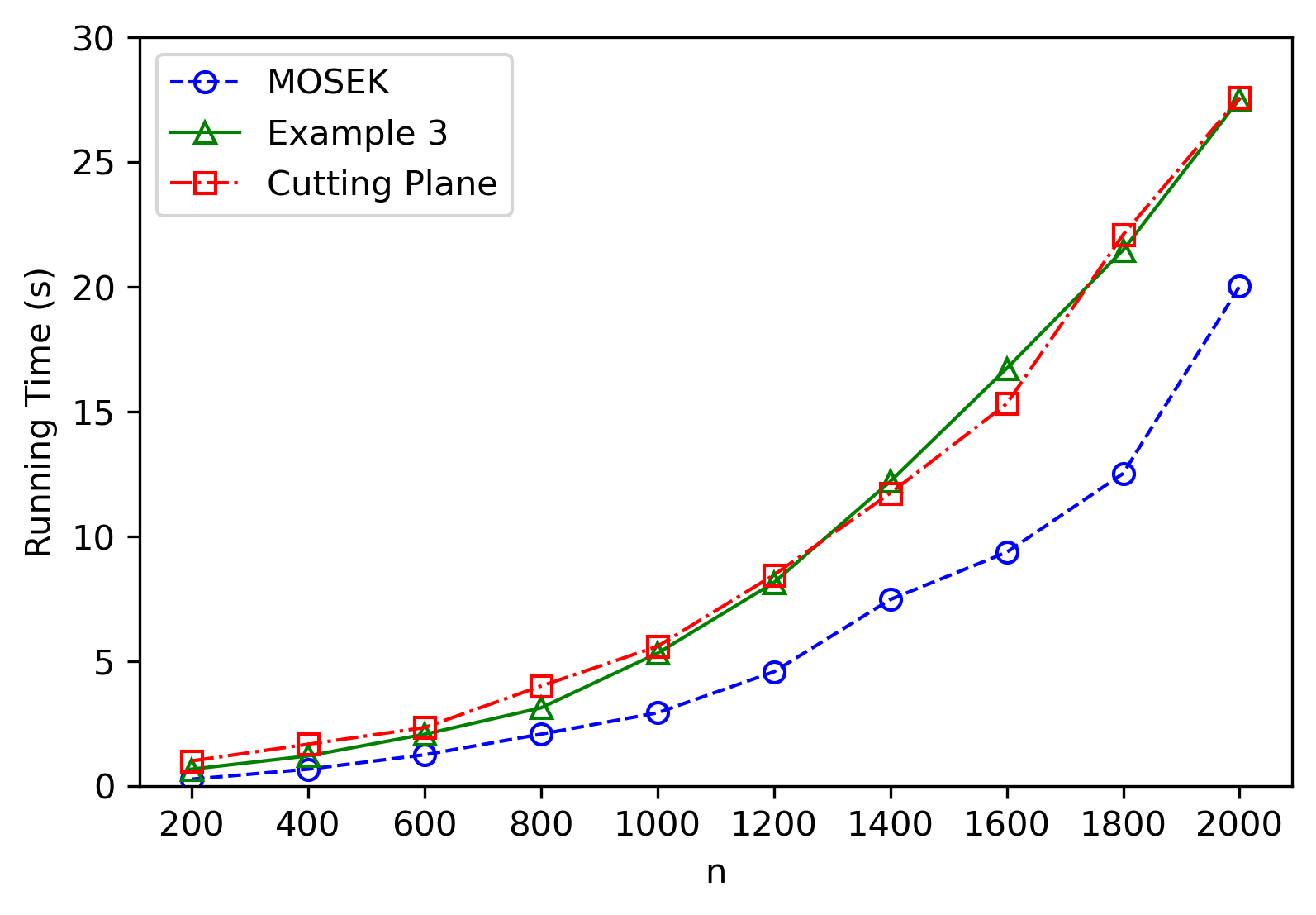}
				\label{fig_conti_packing_large}}
			\caption{Numerical Illustration of MOSEK and Approximation Methods: Continuous Packing ECP}\label{fig_conti_packing}
		\end{figure}

	\end{appendices}

\end{document}